\documentclass[a4paper,10pt,reqno]{amsart}
\numberwithin{equation}{section}
\usepackage{latexsym,amssymb}
\usepackage{hyperref}
\usepackage{amsmath,amsthm}
\usepackage{amsfonts}
\usepackage[american]{babel}
\usepackage{mathrsfs}
\usepackage{psfrag}
\usepackage{epsfig,inputenc}
\usepackage{graphpap,latexsym,epsf}
\usepackage{color}
\usepackage{amssymb,eucal,paralist,color,enumerate}
\usepackage{graphicx}

\newcommand{\bbM}{\mathbb{M}}

\newcommand{\R}{\mathbb{R}}
\newcommand{\N}{\mathbb{N}}

\mathchardef\emptyset="001F

\numberwithin{equation}{section}
\newtheorem{maintheorem}{Theorem}
\newtheorem{theorem}{Theorem}[section]

\newtheorem{lemma}[theorem]{Lemma}
\newtheorem{remark}[theorem]{Remark}

\newtheorem{definition}[theorem]{Definition}
\newtheorem{proposition}[theorem]{Proposition}
\newtheorem{problem}[theorem]{Problem}

\newtheorem{notation}[theorem]{Notation}

\newcommand{\eps}{\varepsilon}

\newcommand{\weakto}{\rightharpoonup} 

\newcommand{\forae}{\text{for a.a. }}
\newcommand{\aein}{\text{a.e.\ in }}

\newcommand{\down}{\downarrow}
\newcommand{\weaksto}{\overset{*}{\rightharpoonup}}

\newcommand{\AC}{\mathrm{AC}}

\newcommand{\dualoperator}

\def\calA{{\mathcal A}}  
 \def\calE{{\mathcal E}} \def\calF{{\mathcal F}}
\def\calG{{\mathcal G}} \def\calH{{\mathcal H}} 
 \def\calK{{\mathcal K}} \def\calL{{\mathcal L}}
  
 \def\calQ{{\mathcal Q}} \def\calR{{\mathcal R}}
 \def\calT{{\mathcal T}}

  \def\rmC{{\mathrm C}}

\def\BS{\boldsymbol} 

    \def\bfmu{{\BS\mu}}

\def\dd{\;\!\mathrm{d}} 

\newcommand{\pairing}[4]{ \sideset{_{ #1 }}{_{ #2 }}  {\mathop{\langle #3 , #4
\rangle}}}

\newcommand{\teta}{\vartheta}

\newcommand{\nchi}{{\raise.2ex\hbox{$\chi$}}}

\definecolor{ddcyan}{rgb}{0,0.1,0.9}
\definecolor{ddmagenta}{rgb}{0.8,0,0.8}
\definecolor{orange}{rgb}{0.6,0.2,0}
\definecolor{vgreen}{rgb}{0.1,0.5,0.2}
\definecolor{dred}{rgb}{.8,0,0}
\definecolor{Turk}{rgb}{0,0.7,0.4}

\newcommand{\piecewiseConstant}[2]{\overline{#1}_{\kern-1pt#2}}
\newcommand{\pwc}{\piecewiseConstant}

\newcommand{\upiecewiseConstant}[2]{\underline{#1}_{\kern-1pt#2}}

\newcommand{\upwc}{\upiecewiseConstant}
\newcommand{\piecewiseLinear}[2]{{#1}_{\kern-1pt#2}}
\newcommand{\pwl}{\piecewiseLinear}
\newcommand{\pwwll}[2]{\widehat{#1}_{\kern-1pt#2}}

\newcommand{\piecewiseVariational}[2]{\tilde{#1}_{\kern-1pt#2}}

\newcommand{\DDDn}[2]{\begin{array}[t]{c}#1\vspace*{-1em}\\_{#2}\end{array}}

\newcommand{\dddn}[2]{\DDDn{\begin{array}[t]{c}\underbrace{#1}\vspace*{.6em}\end{array}}{\text{\footnotesize #2}}}

\newcommand{\foraa}{\text{for a.a. }}

\newcommand{\ue}{u_\varepsilon}

\newcommand{\sig}[1]{E(#1)}

\newcommand{\BV}{\mathrm{BV}}
\newcommand{\BD}{\mathrm{BD}}
\newcommand{\Var}{\mathrm{Var}}

\newcommand{\Dir}{\mathrm{Dir}}
\newcommand{\Neu}{\mathrm{Neu}}

 \def\trait #1 #2 #3 {\vrule width #1pt height #2pt depth #3pt}

 \def\fin{\hfill
         \trait .3 5 0
         \trait 5 .3 0
         \kern-5pt
         \trait 5 5 -4.7
         \trait 0.3 5 0
 \medskip}
 
\newcommand{\QED}{\mbox{}\hfill\rule{5pt}{5pt}\medskip\par}

\newcommand{\bsL}{\boldsymbol{L}}

\newcommand{\bsU}{\boldsymbol{U}}
\newcommand{\bsE}{\boldsymbol{E}}

\newcommand{\bbC}{\mathbb{C}}
\newcommand{\bbI}{\mathbb{I}}
\newcommand{\bbD}{\mathbb{D}}
\newcommand{\bbE}{\mathbb{E}}
\newcommand{\bsX}{\mathbf{X}}
\newcommand{\bsV}{\mathbf{V}}
\newcommand{\bsY}{\mathbf{Y}}

\newcommand{\bbB}{\mathbb{B}}

\newcommand{\mt}{\bbM}

\newcommand{\sym}{\mathrm{sym}}
\newcommand{\dev}{\mathrm{D}}
\newcommand{\uu}{u}

\newcommand{\condu}{\kappa}

\newcommand{\dip}[3]{\mathrm{H}(#1,#2;#3)}
\newcommand{\dipname}{\mathrm{H}}
\newcommand{\Dipname}{\mathcal{H}}
\newcommand{\dipx}[4]{\mathrm{H}(#1,#2, #3;#4)}
\newcommand{\Dip}[3]{\mathcal{H}(#1,#2;#3)}
\newcommand{\did}[1]{\mathrm{R}(#1)}
\newcommand{\Did}[1]{\mathcal{R}(#1)}
\newcommand{\Didv}[1]{\mathcal{R}_{2}(#1)}
\newcommand{\didname}{\mathrm{R}}
\newcommand{\Didname}{\mathcal{R}}
\newcommand{\Didvname}{\mathcal{R}_2}

\newcommand{\Gdir}{\Gamma_{\Dir}}
\newcommand{\Gneu}{\Gamma_{\Neu}}
\newcommand{\subd}{\partial}
\newcommand{\spz}{H^{\mathrm{s}}(\Omega)}
\newcommand{\As}{A_{\mathrm{s}}}
\newcommand{\ass}{a_{\mathrm{s}}}
\newcommand{\spt}{X_\theta}
\newcommand{\bsp}{\mathbf{X}}
\newcommand{\rbsp}{\mathbf{B}}
\newcommand{\rbspv}{\mathbf{V}}

\newcommand{\Ftau}[1]{F_\tau^{#1}}
\newcommand{\Ltau}[1]{\mathcal{L}_\tau^{#1}}
\newcommand{\Gtau}[1]{G_\tau^{#1}}
\newcommand{\gtau}[1]{g_\tau^{#1}}
\newcommand{\ftau}[1]{f_\tau^{#1}}
\newcommand{\wtau}[1]{w_\tau^{#1}}
\newcommand{\utau}[1]{u_\tau^{#1}}

\newcommand{\vtau}[1]{v_\tau^{#1}} 
\newcommand{\btau}[1]{\mathfrak{h}_\tau^{#1}}

\newcommand{\ptau}[1]{p_\tau^{#1}}
\newcommand{\ztau}[1]{z_\tau^{#1}}
\newcommand{\etau}[1]{e_\tau^{#1}}
\newcommand{\tetau}[1]{\teta_\tau^{#1}}
\newcommand{\sitau}[1]{\sigma_{\tau}^{#1}}
\newcommand{\simtau}[1]{\sigma_{M,\tau}^{#1}}
\newcommand{\sidevtau}[1]{(\sigma_{\tau}^{#1})_\dev}
\newcommand{\simdevtau}[1]{(\sigma_{M,\tau}^{#1})_\dev}
\newcommand{\dtau}[2]{\mathrm{D}_{#1,\tau}(#2)}
\newcommand{\Dtau}[2]{\mathrm{D}_{#1,\tau}(#2)}
\newcommand{\Ddtau}[2]{\mathrm{D}_{#1,\tau}^2(#2)}
\newcommand{\tetaum}[1]{\teta_{M,\tau}^{#1}}
\newcommand{\utaum}[1]{u_{M,\tau}^{#1}}
\newcommand{\ptaum}[1]{p_{M,\tau}^{#1}}
\newcommand{\etaum}[1]{e_{M,\tau}^{#1}}
\newcommand{\zetau}[1]{\zeta_{\tau}^{#1}}
\newcommand{\omegatau}[1]{\omega_\tau^{#1}}
\newcommand{\zetaum}[1]{\zeta_{\tau,M}^{#1}}

\newcommand{\sie}{\sigma_\eps}
\newcommand{\siedev}{(\sigma_{\eps})_\dev}
\newcommand{\tetae}{\teta_\eps}
\newcommand{\pe}{p_\eps}

\newcommand{\uek}{u_{\eps_k}}
\newcommand{\tetaek}{\teta_{\eps_k}}
\newcommand{\pek}{p_{\eps_k}}

\newcommand{\limte}{\Theta}
\newcommand{\EEE}{\color{black}}
\newcommand{\MMM}{\color{black}}

\begin{document}
\title[A coupled viscoplastic-damage system with temperature]{Existence results for a coupled viscoplastic-damage model in thermoviscoelasticity}

\author{Riccarda Rossi}
\address{R.\ Rossi, DIMI, Universit\`a degli studi di Brescia,
via Branze 38, 25133 Brescia - Italy}
\email{riccarda.rossi\,@\,unibs.it}

\thanks{The author has been partially supported by the  Gruppo Nazionale per  l'Analisi Matematica, la
  Probabilit\`a  e le loro Applicazioni (GNAMPA)
of the Istituto Nazionale di Alta Matematica (INdAM)}

\date{December 29, 2016} 
\begin{abstract}
In this paper we address  a model coupling viscoplasticity  with damage  in thermoviscoelasticity.
The associated PDE system 
consists of the momentum balance with viscosity and inertia for the displacement variable, at small strains, of the plastic and damage flow rules,
and of the heat equation. It 
has a strongly nonlinear character and in particular features quadratic terms on the right-hand side of
the heat equation and of the damage flow rule, which have to be handled carefully. We propose two weak solution concepts for the related initial-boundary value problem, namely `entropic' and `weak energy' solutions.
Accordingly, we prove two existence results
 by passing to the limit in a carefully devised time discretization scheme. Finally, \MMM in the case of   a \emph{prescribed} temperature profile, \EEE and under a strongly simplifying condition, we provide a continuous   dependence result, yielding uniqueness of weak energy solutions. 
\end{abstract}
\maketitle

\medskip

\centerline{\emph{Dedicated to Tom\'a\v{s} Roub\'{\i}\v{c}ek on the occasion of his 60th birthday}}

\section{\bf Introduction}
\noindent
This paper focuses on two phenomena related to inelastic behavior in materials, namely
damage and plasticity. Damage  can be interpreted as a degradation of  the  elastic
properties of a material due to the failure of its microscopic structure. Such macroscopic mechanical
 effects take
their origin from the formation of cracks and cavities at the microscopic scale. 
They may be described in terms of  an internal variable, the damage parameter,
 on which the elastic modulus depends, in such a way that
  stiffness decreases with ongoing
damage. Plasticity produces residual deformations that remain after complete unloading.
  \par
 Recently,  models combining   plasticity with damage  have been proposed in the context of geophysical modeling
 \cite{Roub-Souc-Vodicka2013,Roub-Valdman2016} and, more in general, within the 
 study of the thermomechanics of damageable materials under diffusion \cite{Roub-Tomassetti}.  Perfect plasticity is featured in \cite{Roub-Souc-Vodicka2013,Roub-Valdman2016}, where the evolution of the damage variable is governed by viscosity, i.e.\ it is \emph{rate-dependent}. 
 Conversely, in \cite{Roub-Tomassetti} damage evolves rate-independently, while the evolution of plasticity is rate-dependent. 
  In a different spirit,   a fully rate-independent system for the evolution of  the damage parameter, coupled with a tensorial variable
  which stands for the trans\-for\-ma\-tion strain arising during damage evolution, is analyzed in 
  \cite{Bonetti-Rocca-Rossi-Thomas16}.
  Finally, let us mention the coupled elastoplastic damage model from \cite{AMV14,AMV15},  analyzed in 
  \cite{Crismale,Crismale16-strgr,Crismale-Lazzaroni}. While the first two papers deal with 
   the fully rate-independent case in \cite{Crismale}, in \cite{Crismale-Lazzaroni},
  the model is regularized by adding viscosity to the damage flow rule, while keeping the evolution of the
  plastic tensor rate-independent. 
   A vanishing-viscosity analysis is then carried out, leading to   an alternative solution concept (`rescaled quasistatic viscosity evolution')
   for the  rate-independent elastoplastic damage system.
   A common feature in  \cite{Roub-Souc-Vodicka2013,Roub-Valdman2016,Crismale,Crismale-Lazzaroni} is that the plastic yield surface, and thus the plastic dissipation potential, depends on the damage variable. 
   \par
  In this paper  \underline{we aim to bring temperature into the picture}. Thermoplasticity models, both in the case of rate-independent evolution of the plastic variable and of rate-dependent one, have been the object of several studies, cf.\ e.g.\ \cite{KS97,KSS02,KSS03,Roub-Bartels-1,Roub-Bartels-2,Roub-PP,HMS,Rossi2016}. In recent years, there has also been a growing literature on the analysis of (rate-independent or rate-dependent) damage models with thermal effects: We quote among others \cite{Bonetti-Bonfanti,Roub10TRIP,Rocca-Rossi2014,Heinemann-Rocca,Rocca-Rossi,LRTT}. 
  \MMM As for models  coupling  plasticity and damage with temperature, 
  one of the examples illustrating the general theory developed in  \cite{Roub-Tomassetti} concerns geophysical models of lithospheres in short time scales, which couple the  (small-strain) momentum balance, damage, \emph{rate-dependendent} plasticity, the heat equation,  as well as the porosity and the water concentration variables. Here we shall neglect the  latter two variables and 
  \EEE
tackle the weak solvability and the existence of solutions,
     for a  (\emph{fully rate-dependent}) viscoplastic  (gradient) damage model,  with viscosity and inertia in the momentum balance (the first  according to Kelvin-Voigt rheology), and with thermal effects  encompassed through the heat equation, \MMM whereas in \cite{Roub-Tomassetti} 
     the enthalpy equation  was analyzed, after a transformation of variables. \EEE
    We plan to address the  vanishing-viscosity and inertia analysis for our model, and discuss the weak solution concept thus obtained, in a future contribution.
    \par
    In what follows, we  shortly comment on the model. 
     Then, we illustrate the mathematical challenges posed by its analysis,
    motivate a suitable regularization, and
    introduce the two solution concepts, for the original system and its regularized version, for which we will prove two existence results.
  \subsection{The thermoviscoelastoplastic damage system}
\label{ss:1.1}
The PDE system, posed in  $\Omega \times (0,T)$, where  the  reference configuration  $\Omega$ is a    bounded, open, Lipschitz domain
in  $ \R^d$, $d\in \{2, 3\}$, and  $(0,T)$ is a given time interval, consists of 
\begin{itemize}
\item[-] the kinematic admissibility condition
\begin{subequations}
\label{PDE-INTRO} 
\begin{align}
\label{decomp-intro}
&
\sig u = e+p  &&  \text{ in } \Omega \times (0,T), 
\end{align}
which provides a decomposition of the linearized strain tensor $\sig{u} = \tfrac12 (\nabla u {+} \nabla u^{\top})$ into the sum of the elastic and plastic strains $e$ and $p$. In fact, $e \in \mt_\sym^{d\times d}$ (the space of symmetric $(d{\times} d)$-matrices), while $p \in  \mt_\dev^{d\times d}$ (the space of symmetric $(d{\times} d)$-matrices with null trace).
\item[-] The momentum balance 
\begin{align}
\label{mom-balance-intro}
\begin{aligned}
&
 \rho \ddot{u} - \mathrm{div}\sigma = F  && \text{ in } \Omega \times (0,T),
\\
 \text{with the stress given by } \quad  & \sigma = \mathbb{D}(z) \dot{e} + \mathbb{C}(z)e - \bbC(z)\mathbb{E}\teta  &&  \text{ in } \Omega \times (0,T)
\end{aligned}
\end{align}
according to Kelvin-Voigt rheology for materials subject to thermal expansion (with $\bbE$ the matrix of the thermal expansion coefficients). Here, $F$ is a given body force. Observe that  both the elasticity and viscosity tensors $\bbC$ and $\bbD$ depend on the damage parameter $z$,
 but we restrict to \emph{incomplete} damage. Namely, the tensors $\bbC$ and $\bbD$ are definite positive uniformly w.r.t.\ $z$, meaning that the system retains its elastic properties even when damage is maximal.
 \item[-] The damage flow  rule for $z$
 \begin{align}
\label{flow-rule-dam-intro}
&
\partial\did{\dot{z}} + \dot{z} + \As (z)+ W'(z) \ni - \tfrac12\bbC'(z)e : e +\teta  && \text{ in } \Omega \times (0,T),
\end{align}
where the   dissipation potential (density)
\[
\didname:\R \to [0,+\infty]  \  \text{ defined by } \   \did{\eta}:= \left\{ \begin{array}{ll}
|\eta| & \text{ if } \eta\leq 0,
\\
+\infty & \text{ otherwise},
\end{array}
\right.
\]
encompasses the unidirectionality in the evolution of damage. We denote by 
$\partial\didname: \R \rightrightarrows \R$ its subdifferential in the sense of convex analysis.
As in several other damage models, we confine ourselves to a \emph{gradient theory}. 
 However, along the lines of \cite{KnRoZa13VVAR} and \cite{Crismale-Lazzaroni}, we 
 adopt a special choice of the gradient regularization, i.e.\ through the  $\mathrm{s}$-Laplacian operator $\As$, with $\mathrm{s}>\tfrac d2$. This technical condition 
 ensures the compact embedding $\spz \Subset \rmC^0(\overline\Omega)$ for the associated Sobolev-Slobodeckij space $\spz:=   W^{\mathrm{s},2}(\Omega)$. Furthermore, a key role in the analysis, especially for  the regularized system with the flow rule \eqref{visc-flow-rule-regu} ahead, will be played by the \emph{linearity} of the operator $\As$.   As in  \cite{Crismale-Lazzaroni},  the term $W'(z)$ will have a singularity at $z=0$, which will  enable us to prove
 the positivity of the damage variable. Combining this with the unidirectionality constraint $\dot{z} \leq 0$ a.e.\ in $\Omega \times (0,T)$, we will ultimately infer that all $z$ emanating from an initial datum $z_0 \leq 1$  take values in the physically admissible interval $[0,1]$.
  \item[-] The flow rule for the plastic tensor reads
 \begin{align}
&
\label{flow-rule-plast-intro}
\partial_{\dot{p}}  \dip{z}{\teta}{\dot{p}} + \dot{p} \ni \sigma_{\mathrm{D}}   && \text{ in } \Omega \times (0,T),
\end{align}
with $\sigma_\dev$ the deviatoric part of the stress tensor $\sigma$. Here, the plastic dissipation potential  (density) $\dipname $ depends on 
the plastic strain rate $\dot{p}$ but also (on the space variable $x$), on the temperature, and on the  damage variable;
the symbol $\partial_{\dot{p}} \dipname$ denotes its convex  subdifferential w.r.t.\ the plastic rate.
 Typically, one may assume that the plastic yield surface decreases as damage increases, although this monotonicity property will not be needed for the analysis developed in this paper.
Observe that \eqref{flow-rule-plast-intro} is in fact a viscous  regularization of the flow rule
\[
\partial_{\dot{p}}  \dip{z}{\teta}{\dot{p}}  \ni \sigma_{\mathrm{D}}   \qquad \text{ in } \Omega \times (0,T)
\]
of perfect plasticity.
\item[-] The heat equation
\begin{align}
\label{heat-intro}
&
\begin{aligned}
\dot{\teta} - \mathrm{div}(\condu(\teta)\nabla \teta)  = & G+
\mathbb{D} \dot{e} : 
\dot{e} -\teta \mathbb{C}(z)\mathbb{E}  : \dot{e}
+ \did {\dot z} + |\dot z|^2 -\teta \dot z 
+ \dip{z}{\teta}{\dot{p}} + \dot{p}: \dot{p}
 \end{aligned}
 && \text{ in } \Omega \times (0,T),
\end{align}
with
$\condu \in \mathrm{C}^0(\R^+)$ the heat conductivity coefficient and 
 $G$ a given, positive, heat source. 
\end{subequations}
\end{itemize}
We will supplement system  \eqref{PDE-INTRO} with the boundary conditions 
\begin{subequations}
\label{BC-INTRO}
\begin{align}
\label{bc-4-u}
& 
u = w   && \text{ on } \Gdir \times (0,T),  && \sigma n = f &&  \text{ on } \Gneu \times (0,T), 
\\
\label{bc-4-z+teta}
&
\partial_n z =0 &&  \text{ on } \partial\Omega \times (0,T),  &&  \condu(\teta) \nabla \teta n = g   && \text{ on } \partial\Omega \times (0,T),
\end{align}
\end{subequations}
where $n$ is the outward unit normal to $\partial\Omega$, e $\Gdir$, $\Gneu$ the Dirichlet/Neumann parts of the boundary, respectively, and with initial conditions. 
\par
In fact, system \eqref{PDE-INTRO} can be seen as the extension of the model considered in \cite{Crismale-Lazzaroni}, featuring  the static momentum balance and a mixed rate-dependent/rate-independent character in the evolution laws for the damage/plastic variables, respectively,
to the case where viscosity is included in the plastic flow rule and in the momentum balance, the latter also with inertia, and the evolution of temperature is also encompassed.
\par
System (\ref{PDE-INTRO}, \ref{BC-INTRO}) can be  rigorously derived following, e.g., the thermomechanical modeling approach by \textsc{M.\ Fr\'emond} \cite[Chap.\ 12]{Fre02}. In this way one can also verify its compliance with the first and second principle of Thermodynamics, hence its thermodynamical consistency.
\subsection{Analytical challenges and  weak solution concepts}
\label{ss:1.2}
Despite the fact that the `viscous' plastic flow rule \eqref{flow-rule-plast-intro} does not bring all the technical difficulties attached to perfect plasticity
 (cf.\ \cite{DMDSMo06QEPL}, see also \cite{Roub-PP} for the coupling with temperature), the analysis of system (\ref{PDE-INTRO}, \ref{BC-INTRO})  still poses some mathematical difficulties.
 Namely,
 \begin{itemize}
 \item[\textbf{(1):}] The overall nonlinear character of \eqref{PDE-INTRO} and, in particular, the quadratic terms on the right-hand side of the 
 damage flow rule \eqref{flow-rule-dam-intro},  and on the right-hand side of the heat equation \eqref{heat-intro}. Both sides are, thus, only estimated in $L^1(\Omega{\times}(0,T))$ as soon as $\dot{e}$, $\dot{z}$, and $\dot{p}$ are estimated in $L^2(\Omega{\times}(0,T);\mt_\sym^{d\times d})$, $L^2(\Omega)$, and $L^2(\Omega{\times}(0,T);\mt_\dev^{d\times d})$, respectively,  as guaranteed by the \emph{dissipative estimates} associated with \eqref{PDE-INTRO}.
 \par
Observe that the particular character of the momentum balance, where the elasticity and the viscosity contributions only involve the elastic part of the strain $e$ and its rate $\dot{e}$, instead of the full strain $\sig{u}$ and strain rate $\sig{\dot u}$, does not allow for elliptic regularity arguments which could at least enhance the spatial regularity/summability of the right-hand side of the damage flow rule.
 \item[\textbf{(2):}] 
 Another obstacle is given by the presence of the \emph{unbounded} maximal monotone operator $\partial\didname$ in the damage flow rule. Because of this, no comparison estimates can be performed. In particular, a pointwise formulation of \eqref{flow-rule-dam-intro} would require a separate estimate of the terms $\As(z)$ and of  (a selection in) $\partial\didname(\dot z)$. This cannot be obtained by standard monotonicity arguments due to the nonlocal character of the operator $\As$.
 \end{itemize}
 All of these  issues shall be reflected in the weak solution concept for system  (\ref{PDE-INTRO}, \ref{BC-INTRO})  proposed in the forthcoming  Definition 
 \ref{def:entropic-sols} and referred to as `entropic solution'. This solvability notion 
  consists of the so-called \emph{entropic formulation} of the heat equation, and  of a weak formulation of the damage flow rule, in the spirit of the \emph{Karush-Kuhn-Tucker} conditions. The entropic formulation originates from the work by \textsc{E.\ Feireisl} in fluid mechanics \cite{Feireisl2007} and has been first adapted to the context of phase transition systems in \cite{FPR09}, and later extended to damage models in \cite{Rocca-Rossi}. It is given by an \emph{entropy inequality}, formally obtained by dividing  the heat equation by $\teta$ and testing the resulting relation by  a sufficiently regular, \emph{positive} test function (cf.\ the calculations at the beginning of Section \ref{ss:2.3}), combined with a \emph{total energy inequality}.  The weak formulation of the damage flow rule has been first proposed in the context of damage modeling in 
 \cite{HeiKra-1,HeiKra-2}: the subdifferential inclusion for damage is replaced by a one-sided variational inequality, with test functions reflecting  the sign constraint imposed by the dissipation potential $\didname$,
 joint with a (mechanical) energy-dissipation inequality also incorporating contributions from the momentum balance and the plastic flow rule.
\par
Clearly, one of the analytical advantages of the entropy inequality for the heat equation, and of the one-sided inequality for the damage flow rule, is that
  the troublesome quadratic terms on the right-hand sides of \eqref{flow-rule-dam-intro} and of \eqref{heat-intro}
feature as multiplied by a negative test function, \MMM cf.\ \eqref{1-sided-intro} and \eqref{entropy-ineq} ahead, respectively. \EEE  This  allows for upper semicontinuity arguments in the limit passage in  suitable
approximations of such inequalities. 
 Instead,  the total and mechanical energy inequalities can be obtained by lower semicontinuity techniques. 
 \par
 We will also consider a 
\emph{regularized version} of system (\ref{PDE-INTRO}, \ref{BC-INTRO}), where the damage flow rule features the additional term $\As(\dot z)$, modulated by a positive constant $\nu$, and, accordingly, the term $\nu \ass(\dot z, \dot z) $ 
(with $\ass $ the bilinear form associated with $\As$) occurs on the right-hand side of  the heat equation. This leads to 
  the \emph{regularized} thermoviscoelastoplastic damage system
  \begin{subequations}
 \label{regularized-system}
 \begin{align}
\label{decomp-intro-bis}
&
\sig u = e+p  &&  \text{ in } \Omega \times (0,T), 
\\
\label{mom-balance-intro-bis}
&
 \rho \ddot{u} - \mathrm{div}\sigma = F  && \text{ in } \Omega \times (0,T),
\\
\nonumber
  & \sigma = \mathbb{D}(z) \dot{e} + \mathbb{C}(z)e - \bbC(z)\mathbb{E}\teta  &&  \text{ in } \Omega \times (0,T), 
\\
 \label{visc-flow-rule-regu}
 &
 \partial\did{\dot{z}} + \dot{z} + \nu \As (\dot z) +  \As (z)+ W'(z) \ni - \tfrac12\bbC'(z)e : e +\teta &&     \text{ in } \Omega \times (0,T),
 \\
 &
  \label{pl-flow-rule-bis}
 \partial_{\dot{p}}  \dip{z}{\teta}{\dot{p}} + \dot{p} \ni \sigma_{\mathrm{D}}   && \text{ in } \Omega \times (0,T),
 \\
 \label{heat-regu}
&
\begin{aligned}
\dot{\teta} - \mathrm{div}(\condu(\teta)\nabla \teta)  = & G+
\mathbb{D} \dot{e} : 
\dot{e} -\teta \mathbb{C}(z)\mathbb{E}  : \dot{e}
\\ & 
\quad 
+ \did {\dot z} + |\dot z|^2 +\bar{\nu} \ass(\dot z, \dot z)  -\teta \dot z 
+ \dip{z}{\teta}{\dot{p}} + \dot{p}: \dot{p}
 \end{aligned}
 && \text{ in } \Omega \times (0,T),
\end{align}
\end{subequations}
with $\bar{\nu} = \nu/|\Omega|$, 
supplemented with the
boundary conditions  
\begin{subequations}
\label{more-reg-bc}
\begin{align}
& 
u = w   && \text{ on } \Gdir \times (0,T),  && \sigma n = f &&  \text{ on } \Gneu \times (0,T),  && &&
\\
&
\partial_n z =0  &&  \text{ on } \partial\Omega \times (0,T), 
&&
\partial_n \dot{z} =0   && \text{ on } \partial\Omega \times (0,T),
 &&  \condu(\teta) \nabla \teta n = g  &&   \text{ on } \partial\Omega \times (0,T).
\end{align}
\end{subequations}
For this  regularized system we will  be able to show the existence of \MMM  an enhanced type of solution. 
It features \EEE a conventional weak formulation of the heat equation with suitable test functions, and the formulation of the damage flow rule as a subdifferential inclusion in $\spz^*$. To obtain the latter, a key role is played by the regularizing term 
$\As (\dot z) $, ensuring that $\dot z \in \spz$  a.e.\ in $(0,T)$.
Thanks to this feature, it is admissible to test the subdifferential inclusion rendering \eqref{visc-flow-rule-regu} by $\dot {z}$ itself. This is at the core of the validity of the total energy \emph{balance}.  \MMM That is why, also in accordance with the nomenclature from 
\cite{Rocca-Rossi, Rossi2016}, we shall refer to these enhanced solutions as \EEE
\emph{weak energy solutions}.
\subsection{Our results}
\label{ss:1.3}
We will  prove the existence  of  \emph{entropic} solutions, see \underline{\textbf{Theorem \ref{mainth:1}}}, and of \emph{weak energy} solutions,    see \underline{\textbf{Theorem \ref{mainth:2}}},   to   (the Cauchy problems for) 
systems (\ref{PDE-INTRO}, \ref{BC-INTRO})  and (\ref{regularized-system}, \ref{more-reg-bc}), respectively, by passing to the limit in a  time-discretization scheme carefully devised in such a way as to ensure the validity of  discrete versions of the  entropy and   energy inequalities along suitable interpolants of the discrete solutions. 
One of our standing assumptions will be a suitable growth of the heat conductivity coefficient $\condu$, namely
\begin{equation}
\label{growth-condu}
\condu(\teta) \sim \teta^\mu \quad \text{for some  } \mu >1.
\end{equation}
Mimicking the calculations from \cite{FPR09, Rocca-Rossi}, we will exploit \eqref{growth-condu} in a key way to derive an estimate for $\teta$ in 
$L^2(0,T;H^1(\Omega))$ by testing the (discrete) heat equation   by a suitable negative power of $\teta$. 
\par
Under a more restrictive condition on $\mu$, in fact depending on the space dimension $d$, 
\MMM (i.e., $\mu\in (1,2)$ if $d=2$, $\mu \in (1,\tfrac53)$ if $d=3$), \EEE
we will also be able to obtain a $\mathrm{BV}$-in-time estimate for $\teta$, with values in a suitable dual space, which will be at the core of the proof of the enhanced formulation of the heat equation for the weak energy solutions to system (\ref{regularized-system}, \ref{more-reg-bc}).
\MMM Concerning the physical interpretation of our growth conditions, we refer to \cite{Klein2012} for a discussion of 
 experimental findings suggesting that
a class of polymers exhibit a subquadratic growth for $\condu$. \EEE 
\par
Finally, with \underline{\textbf{Proposition \ref{prop:continuous-dependence}}} we will provide a continuous dependence estimate, yielding uniqueness, for the weak energy solutions to  (\ref{regularized-system}, \ref{more-reg-bc}) in the \MMM case of a  \emph{prescribed} temperature profile, \EEE and with a plastic dissipation potential \emph{independent} of the state variables $z$ and $\teta$. 
\paragraph{\bf Plan of the paper.}
In Section \ref{s:2} we fix all our assumptions, motivate and state our two weak solvability notions for systems  (\ref{PDE-INTRO}, \ref{BC-INTRO}) \& (\ref{regularized-system}, \ref{more-reg-bc}),  and finally give the existence  Theorems  \ref{mainth:1} \& \ref{mainth:2}, and the continuous dependence result Proposition  \ref{prop:continuous-dependence}.
In Section \ref{s:3} we set up a  common  time discretization scheme for  systems (\ref{PDE-INTRO}, \ref{BC-INTRO}) and (\ref{regularized-system}, \ref{more-reg-bc}), 
and prove the existence of discrete solutions, while Section \ref{s:aprio} is devoted to the derivation of all the
a priori estimates on the approximate solutions, obtained by interpolation of the discrete ones. In Section \ref{s:4} we conclude the proofs of Thms.\  \ref{mainth:1} \& \ref{mainth:2} by passing to the time-continuous limit, while in Section \ref{s:6} we perform the proof of Prop.\  \ref{prop:continuous-dependence}.
\par
We conclude by fixing some notation that shall be used in the paper. 
\begin{notation}[General notation]
\label{not:2.1} \upshape
Throughout the paper,  $\R^+$ shall stand for $(0,+\infty)$. 
For a given $z\in \R$, we will use the notation $(z)^+$ for its positive part $\max\{z,0\}$. 
We will denote by $\mt^{d\times d}$  ($\mt^{d\times d\times d \times d}$) the space of $(d{\times} d)$ ($(d {\times} d{\times} d {\times} d)$, respectively) matrices. We will  consider $\mt^{d\times d}$   endowed with the  Frobenius inner product 
$A : B : = \sum_{i j} a_{ij} b_{ij}$ for two matrices $A = (a_{ij})$ and $B = (b_{ij})$, which induces the matrix norm $|\cdot|$. 
Therefore, we will often write $|A|^2$ in place of $A:A$. The symbol
$\mt_\sym^{d\times d}$ stands for the subspace of symmetric matrices, and $\mt_\dev^{d\times d}$ for the subspace of symmetric matrices with null trace. We recall that 
$\mt_\sym^{d\times d} = \mt_\dev^{d\times d} \oplus \R \bbI$ ($\bbI$ denoting the identity matrix), since every $\eta \in \mt_\sym^{d\times d}$ can be written as 
$
\eta = \eta_\dev+ \tfrac{\mathrm{tr}(\eta)}d \bbI
$
with $\eta_\dev$ the orthogonal projection of $\eta$ into $\mt_\dev^{d\times d} $. We will refer to $\eta_\dev$ as the deviatoric part of $\eta$.  
\par
For a given 
Banach space $X$,
 the symbol $\pairing{}{X}{\cdot}{\cdot}$  will stand for the duality
pairing between $X^*$ and $X$; if $X$ is a Hilbert space, $(\cdot,\cdot)_X$ will denote its inner product. For simpler notation, we shall often write $\| \cdot\|_X$ both for the norm on $X$, and on the product space  $X \times  \ldots \times X$.  With the symbol $\overline{B}_{1,X}(0)$ we will denote the closed unitary ball in $X$. 
We  shall use the symbols
 \[
\text{(i)} \  \mathrm{B}([0,T]; X), \, \qquad \text{(ii)} \  \mathrm{C}^0_{\mathrm{weak}}([0,T];X), \, \qquad \text{(iii)} \  \BV ([0,T]; X)
 \]
 for the spaces
of functions from $[0,T]$ with values in $ X$ that are defined at
\emph{every}  $t \in [0,T]$ and  (i) are measurable; (ii) are  \emph{weakly} continuous   on  $[0,T]$; (iii)  have  bounded variation on  $[0,T]$.
\par
Finally, throughout the paper we will denote various  positive constants depending only on
known quantities  by the symbols
$c,\,c',\, C,\,C'$,  whose meaning may vary even within the same   line.      Furthermore, the symbols $I_i$,  $i = 0, 1,... $,
will be used as place-holders for several integral terms (or sums of integral terms) occurring in
the various estimates: we warn the reader that we will not be
self-consistent with the numbering, so that, for instance, the
symbol $I_1$ will have  different meanings.
\end{notation}

\section{\bf Setup and main results for the thermoviscoelastoplastic damage system}
\label{s:2}
\noindent
After fixing the setup for our analysis in Section \ref{ss:2.1}, in Sec.\ \ref{ss:2.2} we motivate the notion of `weak energy' solution to system 
(\ref{regularized-system}, \ref{more-reg-bc}) by unveiling its underlying energetics. This concept is then precisely fixed in Definition \ref{def:weak-sols}. Sec.\ \ref{ss:2.3} is devoted to the introduction of the considerably weaker concept of `entropic' solutions. Our existence theorems are stated in Sec.\ \ref{ss:2.4}, while in Sec.\ \ref{ss:2.5} we confine the discussion to the  \MMM case of a  given 
temperature profile, \EEE   and give a continuous dependence result for weak energy solutions. 
\subsection{Setup}
\label{ss:2.1}
\paragraph{{\em The reference configuration}.} 
Let $\Omega \subset \R^d$ be a bounded domain, with Lipschitz boundary; we set $Q: =  \Omega  \times  (0,T) $.
 The boundary $\partial\Omega $ is given by  
\begin{equation}
\label{Omega-s2}
\tag{2.$\Omega$}
\begin{gathered}
\partial \Omega = \Gamma_\Dir \cup
\Gamma_\Neu \cup \partial\Gamma \quad \text{ with $\Gamma_\Dir, \,\Gamma_\Neu, \, \partial\Gamma$ pairwise disjoint,}
\\
\text{
 $\Gamma_\Dir$ and $\Gamma_\Neu$ relatively open in $\partial\Omega$, and $ \partial\Gamma$ their relative boundary in $\partial\Omega$,}
 \\
\text{
  with Hausdorff measure $\calH^{d-1}(\partial\Gamma)=0$.}
  \end{gathered}
  \end{equation}
  We will denote by $|\Omega|$ the Lebesgue measure of $\Omega$. 
  On the Dirichlet part $\Gamma_\Dir$, assumed  with $\calH^{d-1}(\Gamma_\Dir)>0, $   we shall prescribe the displacement, while on $\Gamma_\Neu$ we will impose a Neumann condition on the displacement.  The trace of a function $v$ on $\Gamma_\Dir$ or $\Gamma_\Neu$ shall be still  denoted by the symbol $v$. 
    \paragraph{{\em Sobolev spaces, the $\mathrm{s}$-Laplacian, Korn's inequality}.}
  In what follows,  we will use the notation 
  $
  H_\Dir^1(\Omega;\R^d): = \{ u \in H^1(\Omega;\R^d)\, : \ u{|}_{\Dir} = 0 \}.
  $
  The symbol $W_\Dir^{1,p}(\Omega;\R^d)$, $p>1,$  shall denote the analogous $W^{1,p}$-space.
    Further,  we will use the notation \begin{equation}
\label{label-added}
 W_+^{1,p}(\Omega):= \left\{\zeta \in
W^{1,p}(\Omega)\, : \ \zeta(x) \geq 0  \quad \foraa x \in
\Omega \right\}, \quad \text{ and analogously for }
W_-^{1,p}(\Omega).
\end{equation} 
Throughout the paper,
 we shall extensively resort to Korn's inequality (cf.\ \cite{GeySu86}): for every $1<p<\infty$ there exists a constant $C_K = C_K(\Omega, p)>0$ such that there holds
\begin{equation}
\label{Korn}
\| u \|_{W^{1,p}(\Omega;\R^d)} \leq C_K \| \sig u \|_{L^p (\Omega;\mt_\sym^{d \times d})} 
\qquad \text{for all } u \in  W_\Dir^{1,p}(\Omega;\R^d)\,.
\end{equation}
\par
  We will denote by 
\[
  \spz  \   \text{ the  Sobolev--Slobodeckij space }  \  W^{\mathrm{s},2}(\Omega), \text{ with } \mathrm{s} \in \left( \tfrac d2, 2\right).
 \]
      We will also use the notation
      $  H_+^{\mathrm{s}}(\Omega) = \{ z \in \spz\, : z \geq 0 \ \MMM \text{in } \EEE  \Omega\}$ and the analogously defined notation
       $  H_-^{\mathrm{s}}(\Omega) $. 
We recall that $\spz$ is a Hilbert space, with  inner product $(z_1,z_2)_{\spz}: = (z_1,z_2)_{L^2(\Omega)} + \ass(z_1,z_2)$, where    the bilinear form $\ass(\cdot,\cdot)$ is defined by 
   \begin{equation}
   \label{a-s-form}
   \ass(z_1,z_2): = \iint_{\Omega\times\Omega} \frac{\left( \nabla z_1(x){-} \nabla z_1(y)\right) \left( \nabla z_2(x){-} \nabla z_2(y)\right) }{|x{-}y|^{d+2(\mathrm{s}{-}1)}} \dd x \dd y\,.
   \end{equation}
   We denote by  $\As: \spz \to \spz^*$ the associated operator
   \begin{equation}
 \label{s-Laplacian}
 \pairing{}{\spz}{\As(z)}{v}: = \ass(z,v) \qquad \text{for all } v \in \spz\,.
 \end{equation}
\paragraph{{\em Kinematic admissibility and stress}.}  
Given a function $w \in H^1(\Omega;\R^d)$,  we say that a triple $(u,e,p)$ is \emph{kinematically admissible with boundary datum $w$}, and write $(u,e,p) \in \mathcal{A}(w)$,  if
\begin{subequations}
\label{kin-adm}
\begin{align}
&
u \in H^1(\Omega;\R^d), \quad e \in L^2(\Omega;\mt_\sym^{d\times d}), \quad p \in L^2(\Omega;\mt_\dev^{d\times d}),
\\
& \sig u = e+p \quad \aein\, \Omega,
\\
& 
u = w \quad \text{on } \Gamma_\Dir.
\end{align}
\end{subequations}
\par
 As for the elasticity and  viscosity 
   tensors,  we will suppose that 
\begin{subequations}
\label{elast-visc-tensors} 
\begin{equation}
\label{elast-visc-tensors-1} 
\tag{2.${\small (\bbC,\bbD)_1}$}
\begin{aligned}
&
\bbC , \,  \bbD
  \in \mathrm{C}^0(\overline{\Omega}\times \R; \mathrm{Lin}(\mt_\sym^{d \times d}))\,,  \text{ and } 
  \\
& \ \exists\, C_{\bbC}^1,\,  C_{\bbC}^2, \,  C_{\bbD}^1, \,  C_{\bbD}^2>0  \  \forall x \in \Omega \ \forall z \in \R  \ \forall A \in \mt_\sym^{d\times d} \, : \quad \begin{cases}
  &  C_{\bbC}^1 |A|^2 \leq  \bbC(x,z) A : A \leq  C_{\bbC}^2 |A|^2,
   \\
    & C_{\bbD}^1 |A|^2 \leq  \bbD(x,z) A : A \leq  C_{\bbD}^2 |A|^2,
  \end{cases}
  \end{aligned}  \end{equation}
  where $\mathrm{Lin}(\mt_\sym^{d \times d})$ denotes the space of linear operators from $\mt_\sym^{d \times d}$ to $\mt_\sym^{d \times d}$.
  Furthermore, we will suppose that for every $x \in \Omega$ the map $z \mapsto \bbC(x,z)$ is continuously  differentiable on $\R$ and fulfills 
  \begin{equation}
\label{elast-visc-tensors-2} 
\tag{2.${\small (\bbC,\bbD)_2}$}
\bbC'(x,0) =0 \quad \text{for all } x \in \Omega \quad \text{ and } \quad  
\forall\, C_Z>0  \ 
\exists\, L_{\bbC}>0 \  \forall\, x\in \Omega \, :
 \ |z| \leq C_Z \ \Rightarrow 
 \  | \bbC'(x,z) | \leq L_{\bbC}\,.
\end{equation}
Finally, for technical reasons (cf.\ Remark \ref{rmk:tech-comments} later on) it will be convenient to require that the map $z \mapsto \bbC(x,z)$ is convex, i.e.\ for every $A 
\in \mt_\sym^{d \times d}$ there holds
  \begin{equation}
\label{elast-visc-tensors-3} 
\tag{2.${\small (\bbC,\bbD)_3}$}
\begin{aligned}
\bbC(x,(1{-}\theta)z_1{+}\theta z_2)A : A  &  \leq (1-\theta) \bbC(x,z_1) A: A + \theta \bbC(x,z_2)A: A
\\
& 
 \quad \quad  \text{ for all } \theta \in [0,1], \, x\in \Omega, \, z_1,\,z_2 \in \R.
 \end{aligned}
\end{equation}
It follows from  
the convexity \eqref{elast-visc-tensors-3}  that 
\begin{equation}
\label{monotonicity}
\bbC'(x,z_1)(z_1{-}z_2)  A : A \geq  \bbC(x,z_1)  A : A  -   \bbC(x,z_2)  A : A  \qquad \text{for all } x\in \Omega, \, z_1,\, z_2 \in \R, \, A \in  \mt_\sym^{d \times d}\,,
\end{equation}
whence 
$(\bbC'(x,z_1){-} \bbC'(x,z_2)) (z_1{-}z_2)  A : A \geq0$. 
In particular, 
due to the first of \eqref{elast-visc-tensors-3}, 
we find that
\begin{equation}
\label{posbbC}
\bbC'(x,z) A : A \geq 0 \quad \text{for all } x \in \Omega, \, z \in [0,+\infty),  \, A \in  \mt_\sym^{d \times d}\,.
\end{equation}
Finally, we  also  suppose that the thermal expansion tensor fulfills
  \begin{equation}
  \label{thermal-expansion}
  \tag{2.${\small \bbE}$}
  \bbE \in L^\infty(\Omega; \mathrm{Lin}(\mt_\sym^{d \times d}))\,.
  \end{equation}
  \end{subequations}
Observe that with $(2.(\bbC,\bbD,\bbE))$  and \eqref{thermal-expansion}
 we  encompass in our analysis the case of
an anisotropic and inhomogeneous material. 
\paragraph{{\em External heat sources}.}
For the volume and boundary  heat sources $G$ and $g$ we require
\begin{align}
 \label{heat-source} 
 \tag{2.$\mathrm{G}_1$}
 &  G \in L^1(0,T;L^1(\Omega)) \cap L^2 (0,T; H^1(\Omega)^*), &&  G\geq 0 \quad\hbox{a.e.  in } Q\,,
 \\
 \label{dato-h}
  \tag{2.$\mathrm{G}_2$}
 & g \in L^1 (0,T; L^2(\partial \Omega)),  &&  g \geq 0 \quad\hbox{a.e.  in } (0,T)\times \partial \Omega\,.
\end{align}
Indeed, the positivity of $G$ and $g$ is crucial for obtaining the strict positivity of the temperature $\teta$. 
\paragraph{{\em Body force and traction}.}
Our basic conditions on the volume force $F$ and the assigned traction $f$ are
 \begin{equation}
\label{data-displ}
 \tag{2.$\mathrm{L}_1$}
F\in L^2(0,T; H_\Dir^1(\Omega;\R^d)^*), \qquad f \in L^2(0,T; H_{00,\Gamma_\Dir}^{1/2}(\Gamma_\Neu; \R^d)^*),
\end{equation}
where $ H_{00,\Gamma_\Dir}^{1/2}(\Gamma_\Neu; \R^d)$ is the space of functions $\gamma \in H^{1/2} (\Gamma_\Neu;\R^d)$ such that there exists $\tilde\gamma \in H_\Dir^1(\Omega;\R^d)$ with $\tilde\gamma = \gamma $ in $\Gamma_\Neu$.
\par
For technical reasons,   in order to allow for a non-zero traction $f$, 
we will need to additionally require  a  \emph{uniform safe load} type condition. Observe that  this kind of assumption
 usually occurs in the analysis of perfectly plastic systems. 
 In the present context, it will play a pivotal role in the derivation of the \emph{First a priori estimate}  for the approximate solutions 
 constructed by time discretization, cf.\ the proof of Proposition  \ref{prop:aprio} later on as well as \cite[Rmk.\ 4.4]{Rossi2016} for more detailed comments.
 Namely, we impose that there exists a function $\varrho: [0,T] \to L^2(\Omega;\mt_\sym^{d\times d})$,
with $\varrho \in W^{1,1}(0,T;  L^2(\Omega;\mt_\sym^{d\times d}))  $ and $  \varrho_\dev \in L^1(0,T;L^\infty (\Omega; \mt_\dev^{d\times d}))$,
 solving for almost all $t\in (0,T)$ the following elliptic problem 
\begin{equation}
\label{safe-load}
 \tag{2.$\mathrm{L}_2$}
- \mathrm{div}(\varrho(t)) = F(t)  \text{ in } \Omega, \qquad 
\varrho(t) \nu = f(t)  \text{ on } \Gamma_\Neu\,.
\end{equation}
\par
 When not explicitly using \eqref{safe-load}, to shorten notation we will incorporate the volume force $F$ and the traction $f$
  into the induced  total load, namely the function
$\mathcal{L}: (0,T) \to H_\Dir^1(\Omega;\R^d)^*$ given at $t\in (0,T)$ by 
\begin{equation}
\label{total-load}
\pairing{}{H_\Dir^1(\Omega;\R^d)}{\mathcal{L}(t)}{u}: = \pairing{}{H_\Dir^1(\Omega;\R^d)}{F(t)}{u} + \pairing{}{H_{00,\Gamma_\Dir}^{1/2}(\Gamma_\Neu; \R^d)}{f(t)}{u} \qquad \text{for all } u \in H_\Dir^1(\Omega;\R^d),
\end{equation} 
which fulfills $\mathcal{L} \in L^2(0,T; H_\Dir^1(\Omega;\R^d)^*)$ in view of \eqref{data-displ}.
\paragraph{{\em Dirichlet loading}.}
We will suppose that the hard device  $w$ to which the body is subject on $\Gamma_\Dir$  is the trace on $\Gamma_\Dir$  of a function, denoted by the same symbol,  fulfilling
\begin{equation}
\label{Dirichlet-loading}
\tag{2.$\mathrm{w}$} 
w \in  L^1(0,T; W^{1,\infty} (\Omega;\R^d)) \cap W^{2,1} (0,T;H^1(\Omega;\R^d)) \cap H^2(0,T; L^2(\Omega;\R^d))\,.
\end{equation}
Also condition   \eqref{Dirichlet-loading} will be used in the proof of Prop.\  \ref{prop:aprio}  ahead; we again
  refer to \cite[Rmk.\ 4.4]{Rossi2016} for more comments. 
\paragraph{{\em The weak formulation of the momentum balance}.} The variational formulation of  \eqref{mom-balance-intro}, supplemented with the boundary conditions
\eqref{bc-4-u},
reads for almost all $t\in (0,T)$
\begin{equation}
\label{w-momentum-balance}
\begin{aligned}
\rho\int_\Omega \ddot{u}(t) v \dd x + \int_\Omega \big(\bbD(z(t)) \dot{e}(t) + \bbC(z(t)) e(t)  & - \teta(t) \bbC(z(t)) \bbE \big): \sig v \dd x     = \pairing{}{H_\Dir^1(\Omega;\R^d)}{\calL(t)}{v}  \\ & \qquad \qquad \quad \forall\, v \in H_\Dir^1(\Omega;\R^d)\,.
\end{aligned}
\end{equation}
We will often use the short-hand notation $-\mathrm{div}_{\Dir}$ for the elliptic operator  defined by 
\begin{equation}
\label{div-Gdir}
\pairing{}{ H_\Dir^1(\Omega;\R^d) }{-\mathrm{div}_{\Dir}(\sigma)}{v}: = \int_\Omega \sigma : \sig v \dd x \qquad \text{for all } v \in  H_\Dir^1(\Omega;\R^d) \,.
\end{equation}
\paragraph{{\em The plastic dissipation potential}.}
Our  assumptions on the multifunction $K: \Omega \times \R \times  \R^+ \rightrightarrows \mt_\dev^{d \times d}$ involve the notions of measurability, lower semicontinuity, 
and upper semicontinuity  for general multifunctions. For such concepts and the related results, we refer, e.g., to 
  \cite{Castaing-Valadier77}. 
Hence,  we suppose that
\begin{equation}
\label{measutab-cont-K}
\tag{2.$\mathrm{K}_1$}
\begin{aligned}
&
K : \Omega \times \R \times   \R^+ \rightrightarrows \mt_\dev^{d \times d}  &&  \text{ is measurable w.r.t.\ the variables $(x,\teta,z)$,}
\\
&
K(x, \cdot,\cdot) : \R \times  \R^+ \rightrightarrows \mt_\dev^{d \times d}   && \text{ is continuous for almost all } x \in \Omega.
\end{aligned}
\end{equation}
Furthermore, we require that 
\begin{equation}
\label{elastic-domain}
\tag{2.$\mathrm{K}_2$}
\begin{aligned}
K(x,z,\teta) \text{ is a convex and compact set in }  \mt_\dev^{d \times d} \text{ for all } (z,\teta) \in \R \times  \R^+ \text{ and  for almost all } x \in \Omega,
\\
\exists\, 0<c_r<C_R \quad  \foraa x \in \Omega, \ \forall\, z \in \R, \  \forall\, \teta \in  \R^+ \, : \quad B_{c_r}(0) \subset  K(x,z,\teta) \subset B_{C_R}(0).
\end{aligned}
\end{equation}
\par
Therefore, the support function associated with the multifunction $K$, i.e.
\begin{equation}
\label{1-homogeneous-dissip}
\dipname: \Omega  \times \R \times  \R^+ \times  \mt_\dev^{d \times d} \to [0,+\infty) \quad \text{defined by } \dipx{x}z{\teta}{\dot p}: = \sup_{\pi \in K(x,z,\teta)} \pi : \dot p 
\end{equation}
is positive, with $\dipx{x}z{\teta}{ \cdot} :  \mt_\dev^{d \times d} \to   [0,+\infty)  $ convex and $1$-positively homogeneous for almost all $x \in \Omega$ and for all $(z,\teta) \in \R \times \R^+$. 
By the first of \eqref{measutab-cont-K},  the function $\dipname: \Omega \times \R  \times \R^+ \times  \mt_\dev^{d \times d} \to [0,+\infty) $ is measurable.
 Moreover,  by the second of   \eqref{measutab-cont-K},  in view of \cite[Thms.\ II.20, II.21]{Castaing-Valadier77} (cf.\ also  \cite[Prop.\ 2.4]{Sol09}) the function
\begin{subequations}
\label{hypR}
\begin{align}
& 
\label{hypR-lsc}
\dipx{x}{\cdot}{\cdot}{ \cdot}: \R \times   \R^+ \times  \mt_\dev^{d \times d} \to [0,+\infty)   \text{ is (jointly) lower semicontinuous}, 
\intertext{for almost all $x \in \Omega$, i.e.\ $\dipname$ is a \emph{normal integrand},  and} 
& \label{hypR-cont}
\dipx{x}{\cdot}{\cdot}{ \dot p}:  \R \times   \R^+ \to \R^+ \text{ is continuous for every $\dot p\in \mt_\dev^{d\times d}$}. 
\end{align}
\end{subequations}
Finally, 
it follows from the second of \eqref{elastic-domain}
that   for almost all $x\in \Omega$ and for all $ (z,\teta, \dot p) \in \R \times   \R^+ \times  \mt_\dev^{d \times d}  $ there holds
\begin{subequations}
\label{cons-lin-growth}
\begin{align}
&
\label{linear-growth}
c_r|\dot p| \leq \dipx{x}{z}{\teta}{ \dot p } \leq C_R |\dot p|,
\\
&
\label{bounded-subdiff} \partial_{\dot p}  \dipx{x}{z}{\teta}{ \dot p } \subset  \partial_{\dot p}  \dipx{x}{z}{\teta}{ 0 }  
= K(x,z,\teta) \subset B_{C_R}(0)\,.
\end{align}
\end{subequations}
\par 
Finally, we also introduce the \emph{plastic dissipation  potential} $\Dipname:L^1(\Omega) \times L^1(\Omega; \R^+) \times L^1(\Omega;\mt_\dev^{d\times d})$ given by
\begin{equation}
\label{plastic-dissipation-functional}
\Dip{z}{\teta}{ \dot p}: = \int_\Omega \dipx{x}{z(x)}{\teta(x)}{\dot p(x)} \dd x\,.
\end{equation}
 From now on, throughout the paper we will most often omit the $x$-dependence of the tensors $\bbC,\, \bbD,\, \bbE$, and of the dissipation density $\dipname$. 
\paragraph{{\em Nonlinearities in  the damage flow rule}.}
Along the footsteps of   \cite{Crismale-Lazzaroni}, we will suppose that 
\begin{equation}
\label{hyp-phi-1}
\tag{2.$\mathrm{W}_1$} 
W  \in \mathrm{C}^2 (\R^+)  \text{ is bounded from below and fulfills } z^{2d} W(z) \to +\infty \text{ as } z \down 0.
\end{equation}
The latter coercivity condition will play a key role in the proof  that the damage variable  $z$ takes values in the feasible interval $[0,1]$. In this way, we will not have to include the  indicator term  $I_{[0,1]}$ in the potential energy. 
 This will greatly simplify the analysis of the damage flow rule.
\par
Furthermore, 
we shall require that 
\begin{equation}
\label{hyp-phi-2}
\tag{2.$\mathrm{W}_2$}
\exists\, \lambda_W>0 \ \forall\, z \in \R^+ \, : \quad W''(z) \geq -\lambda_W\,.
\end{equation}
Observe that  \eqref{hyp-phi-2} is equivalent to imposing that the function $z\mapsto W(z) + \tfrac{\lambda_W}2|z|^2 =: {\beta}(z)$  is convex. Therefore, 
we have the \emph{convex/concave} decomposition
\begin{equation}
\label{decomp-W}
W(z) = {\beta}(z) - \frac{\lambda_W}2|z|^2  \text{ with } \beta  \in \mathrm{C}^2 (\R^+), \text{ convex, and fulfilling  }  z^{2d} \beta(z) \to +\infty \text{ as } z \down 0.
\end{equation}
Let us point out that  \eqref{decomp-W}  will be expedient in devising the time discretization scheme for the (regularized) thermoviscoelastoplastic
damage system, in such a way that its solutions comply with the discrete version of the total energy inequality. We refer to Remark \ref{rmk:tech-comments} ahead for more comments.
\paragraph{{\em Cauchy data}.} We will supplement the thermoviscoelastoplastic damage system with 
initial data
\begin{subequations}
\label{Cauchy-data}
\begin{align}
&
\label{initial-u}
u_0 \in H_\Dir^{1} (\Omega;\R^d), \ \dot{u}_0 \in L^2 (\Omega;\R^d),
\\
& 
\label{initial-p}
e_0 \in   L^2(\Omega;\mt_\sym^{d\times d}), \quad  p_0 \in L^2(\Omega;\mt_\dev^{d\times d}) \quad \text{such that } (u_0, e_0, p_0 ) \in \mathcal{A}(w(0)),
\\
&\label{initial-z}
z_0 \in \spz \text{ with } W(z_0) \in L^1(\Omega) \text{ and } z_0(x) \leq 1  \text{ for every } x \in \Omega\,,
\\
\label{initial-teta}
&
\begin{aligned}
\teta_0 \in L^1(\Omega),  &  \text{ fulfilling the strict positivity condition }  \exists\, \teta_*>0: \ \inf_{x\in \Omega} \teta_0(x) \geq \teta_*, \\ &  \text{ and such that } \log(\teta_0) \in L^1(\Omega).
\end{aligned}
\end{align}
\end{subequations} 
\par
In the remainder of  this section,  we shall suppose that the functions $\bbC, \ldots, W$, the data $G,\ldots, w$, and the initial data $(u_0, \dot{u}_0, e_0, p_0,z_0,\teta_0)$
 fulfill the conditions stated in Section \ref{ss:2.1}.
  We will first address the weak solvability of the regularized system in Sec.\ \ref{ss:2.2}, and then turn to examining the non-regularized one in Sec.\ \ref{ss:2.3}. We will then state our existence results for both in Sec.\ \ref{ss:2.4}. 
\subsection{Energetics and weak solvability for the  (regularized) thermoviscoelastoplastic damage system}
\label{ss:2.2}
 Prior to stating the precise notion of weak solution 
 for the regularized thermoviscoelastoplastic damage system in
 Definition \ref{def:weak-sols} ahead,
 we 
 formally derive the mechanical \& total energy balances associated with 
  systems (\ref{PDE-INTRO}, \ref{BC-INTRO}) and 
  (\ref{regularized-system}, \ref{more-reg-bc}) (in the ensuing discussion, we shall take the parameter $\nu\geq 0$). 
\paragraph{\emph{The mechanical and total energy balances.}}
The \MMM free \EEE  energy of the system is given by 
\begin{equation}
\label{stored-energy}
\calE(\teta, u, e, p,z)= \calE(\teta,e,z):=  \calF(\teta) + \calQ(e,z) + \calG(z) \quad \text{with } \quad \begin{cases} \calF(\teta) : = \int_\Omega  \teta \dd x, 
\\
\calQ(e,z): = 
 \int_\Omega \tfrac12 \bbC(z) e{:} e   \dd x  
 \\
 \calG(z): = 
   \frac12 \ass(z,z)  + \int_\Omega W(z) \dd x\,.
 \end{cases}
\end{equation}
The total energy balance can be (formally) obtained 
by 
 testing the momentum balance \eqref{mom-balance-intro-bis} by $(\dot u {-} \dot w)$,  the damage flow rule 
 \eqref{visc-flow-rule-regu}
   by $\dot z$, 
 the plastic flow rule  \eqref{pl-flow-rule-bis} by $\dot p$, 
 the heat equation \eqref{heat-regu} by $1$,  adding the resulting relations and  integrating in space and over a generic interval $(s,t) \subset (0,T)$.  
 \par
 Indeed, the tests of  \eqref{mom-balance-intro-bis}, \eqref{visc-flow-rule-regu}, and \eqref{pl-flow-rule-bis}  yield
 \begin{equation}
\label{intermediate-mech-enbal}
\begin{aligned}
& 
\frac{\rho}2 \int_\Omega |\dot{u}(t)|^2 \dd x + \int_s^t\int_\Omega \left(  \bbD(z) \dot{e}  + \bbC(z) e - \teta \bbC(z)\bbE \right)  : \sig{\dot{u}} \dd x \dd r +\nu \int_s^t \ass (\dot{z},\dot{z}) \dd r 
\\
& \quad 
+ \int_s^t \int_\Omega \left(   \did {\dot z} {+}  |\dot z|^2  \right) \dd x \dd r +  \frac12 \ass(z(t),z(t)) + \int_\Omega W(z(t)) \dd x +\int_s^t \int_\Omega  \tfrac12 \dot {z}  \bbC'(z) e{:} e \dd x \dd r 
\\
& \quad 
-\int_s^t \int_\Omega \teta \dot z \dd x \dd r
+ \int_s^t \int_\Omega \left(   \dip z{\teta}{\dot p} {+}  |\dot p|^2  \right) \dd x \dd r
\\
&  = \frac{\rho}2 \int_\Omega |\dot{u}(s)|^2 \dd x  + \int_s^t \pairing{}{H_\Dir^1 (\Omega;\R^d)}{\mathcal{L}}{\dot u- \dot w} \dd r  +
\int_s^t \int_\Omega  \left(  \bbD(z) \dot{e}  + \bbC(z) e - \teta \bbC(z) \bbE  \right)  : \sig{\dot w} \dd x \dd r 
\\
&\quad  +\rho \left( \int_\Omega \dot{u}(t) \dot{w}(t) \dd x -  \int_\Omega \dot{u}_0 \dot{w}(0) \dd x - \int_s^t \int_\Omega \dot{u}\ddot w \dd x \dd r \right) +
 \int_s^t \int_\Omega \sigma_\dev : \dot{p} \dd x \dd r\,.
\end{aligned}
\end{equation}
Now, taking into account that $ \sig{\dot u}  =\dot e + \dot p$ by the kinematical admissibility condition,  rearranging some terms one has that 
\[
\begin{aligned}
 \int_s^t\int_\Omega \left(  \bbD(z) \dot{e}  + \bbC(z) e  -  \teta \bbC(z) \bbE \right)  : \sig{\dot{u}} \dd x \dd r  = &  
 \int_s^t \int_\Omega \left( \bbD (z)\dot e : \dot e  + \bbC(z) \dot e : e  \right) \dd x \dd r   -  \int_s^t \int_\Omega \teta \bbC(z) \bbE : \dot e \dd x \dd  r
 \\
 & \quad + \int_s^t \int_\Omega \left(  \bbD(z) \dot e+ \bbC(z) e  - 
\teta \bbC(z) \bbE \right)   : \dot p   \dd x \dd r \,. 
\end{aligned}
\]
We substitute  this in \eqref{intermediate-mech-enbal} and  note that $  \int_s^t \int_\Omega \left(\bbD(z) \dot e  + \bbC(z) e - \teta \bbC(z) \bbE \right)  : \dot{p} \dd x \dd r =
 \int_s^t \int_\Omega \sigma_\dev : \dot{p} \dd x \dd r   $ since $\dot{p} \in \mt_\dev^{d\times d}$, so that the last term on the right-hand side of \eqref{intermediate-mech-enbal} cancels out.
 Furthermore, by the chain rule we have that 
 \[
 \int_s^t \int_\Omega  \left( \bbC(z) \dot e : e {+}  \tfrac12 \dot {z}  \bbC'(z) e{:} e \right)  \dd x \dd r = 
 \int_\Omega \tfrac12 \bbC(z(t))e(t){:}e(t)  \dd x -  \int_\Omega \tfrac12 \bbC(z(s))e(s){:}e(s)  \dd x \,.
 \]
 Collecting all of the above calculations, we obtain 
 the \emph{mechanical energy balance}, featuring the kinetic, dissipated, and mechanical energies
\begin{equation}
\label{mech-enbal}
\begin{aligned}
& 
\dddn{\frac{\rho}2 \int_\Omega |\dot{u}(t)|^2 \dd x}{kinetic} +\dddn{ \int_s^t\int_\Omega   \left(
 \bbD(z) \dot e: \dot e  + \did{\dot{z}} + |\dot{z}|^2 +  \dip{z}{\teta}{\dot p}  + |\dot p|^2 \right)  \dd x \dd r + \nu \int_s^t  \ass(\dot{z}, \dot{z}) \dd r  }{dissipated}
 \\
 & \quad 
+ \dddn{\calQ(e(t),z(t)) + \calG(z(t))}{mechanical}
\\
&  =  \frac{\rho}2 \int_\Omega |\dot{u}(s)|^2 \dd x   + \calQ(e(s),z(s)) + \calG(z(s)) + \int_s^t \pairing{}{H_\Dir^1 (\Omega;\R^d)}{\mathcal{L}}{\dot u{-} \dot w} \dd r  +  
 \int_s^t \int_\Omega  \left( \teta \bbC(z) \bbE : \dot e + \teta \dot{z} \right) \dd x \dd  r
 \\
& \quad   +\rho \left( \int_\Omega \dot{u}(t) \dot{w}(t) \dd x -  \int_\Omega \dot{u}(s) \dot{w}(s) \dd x - \int_s^t \int_\Omega \dot{u}\ddot w \dd x \dd r \right)  
  + 
\int_s^t \int_\Omega  \sigma : \sig{\dot w} \dd x \dd r\,. 
 \end{aligned}
\end{equation}
Let us highlight that \eqref{mech-enbal}  will also have a significant role for our analysis. 
\par
Hence, we sum \eqref{mech-enbal} with the heat  equation \eqref{heat-regu} tested by $1$ and integrated in time and space. We observe the cancelation of some terms, in particular
noting that 
\[
\bar\nu\int_s^t \int_\Omega \ass (\dot{z},\dot{z}) \cdot 1 \dd x \dd r = \nu \int_s^t \ass (\dot{z},\dot{z})  \dd r\,.
\]
All in all, we conclude
   the \emph{total energy balance} 
\begin{equation}
\label{total-enbal}
\begin{aligned}
& 
\frac{\rho}2 \int_\Omega |\dot{u}(t)|^2 \dd x +\mathcal{E}(\teta(t), e(t), z(t)) 
\\
&  = \frac{\rho}2 \int_\Omega |\dot{u}(s)|^2 \dd x +\mathcal{E}(\teta(s), e(s),z(s))  + \int_s^t \pairing{}{H_\Dir^1 (\Omega;\R^d)}{\mathcal{L}}{\dot u{-} \dot w} 
  +\int_s^t \int_\Omega G \dd x \dd r + \int_s^t \int_{\partial\Omega} g \dd S \dd r
\\
& \quad   +\rho \left( \int_\Omega \dot{u}(t) \dot{w}(t) \dd x -  \int_\Omega \dot{u}_0 \dot{w}(0) \dd x - \int_s^t \int_\Omega \dot{u}\ddot w \dd x \dd r \right)     + 
\int_s^t \int_\Omega \sigma: \sig{\dot w} \dd x \dd r
 \,.\end{aligned}
\end{equation}
\paragraph{\emph{Weak energy solutions for the regularized system.}}
     With the following definition (where the conditions from Sec.\ \ref{ss:2.1} are tacitly assumed), 
     we fix the properties of the weak solution concept for the regularized thermoviscoplastic damage system.
     Let us mention in advance that, in addition to the conventional weak formulations of the momentum balance and of the heat equation
     (in the latter case, with test functions with suitable regularity and summability properties), we will require the validity of the plastic flow rule 
     \emph{pointwise} (almost everywhere) in $\Omega \times (0,T)$, and that of the damage flow rule  as a subdifferential inclusion in $H^{\mathrm{s}}(\Omega)^*$.  
     \MMM It will be in fact possible to obtain the latter by exploiting the additional, strongly  regularizing  term $\As(\dot z)$ in the damage flow rule
     \eqref{visc-flow-rule-regu}. \EEE 
     In this connection, we now
     introduce the  dissipation potential, defined on
 $\spz$ and  induced by $\didname$, namely
\begin{equation}
\label{abstract-dissip-potential}
\Didname: \spz \to [0,+\infty], \qquad \Did{\dot z}: = \int_\Omega \did{\dot z} \dd x,
\end{equation}
and its `viscous' regularization, where $\Didname$ is augmented by the (squared) $L^2(\Omega)$-norm
\begin{equation}
\label{didvname}
\Didvname: \spz \to [0,+\infty], \qquad \Didv{\dot z}: = \int_\Omega \did{\dot z} \dd x + \frac12 \| \dot{z}\|_{L^2(\Omega)}^2.
\end{equation}
We will denote by  $\partial \Didname: \spz \rightrightarrows \spz^*$ and 
 $\partial \Didvname: \spz \rightrightarrows \spz^*$ the 
 the convex analysis  subdifferentials  of $\Didname$ and $\Didvname$, respectively.
   \begin{definition}[Weak energy solutions to the regularized thermoviscoelastoplastic  damage system]
   \label{def:weak-sols}
   $ $\\
 Given initial data $(u_0, \dot{u}_0, e_0,z_0, p_0,\teta_0)$ fulfilling \eqref{Cauchy-data}, we call a quintuple $(u,e,z,p,\teta)$
a \emph{weak energy solution} to the Cauchy problem for system (\ref{regularized-system}, \ref{more-reg-bc}),
 supplemented with the boundary conditions 
 (\ref{bc-4-u}, \ref{more-reg-bc}),   
   if
\begin{subequations}
\label{regularity}
\begin{align}
& \label{reg-u} u \in H^1(0,T; H_\Dir^{1}(\Omega;\R^d)) \cap W^{1,\infty}(0,T; L^2(\Omega;\R^d)) \cap H^2(0,T; H_\Dir^{1}(\Omega;\R^d)^*),
\\
 & \label{reg-e} e \in H^1(0,T; L^2(\Omega;\mt_\sym^{d\times d})),
 \\
 & \label{reg-z} z \in L^\infty(0,T;\spz) \cap H^1(0,T;L^2(\Omega)),
\\
& \label{reg-p} p \in H^1(0,T; L^2(\Omega;\mt_\dev^{d\times d})),
\\
& 
\label{reg-teta}  \teta \in  L^2(0,T; H^1(\Omega))\cap L^\infty(0,T;L^1(\Omega)),
\\
&
\label{further-reg-z}
z\in H^1(0,T;\spz),
\\
&
 \label{enh-teta-W11}
 \teta \in W^{1,1}(0,T; W^{1,\infty}(\Omega)^*) \text{ and }  \condu(\teta) \nabla \teta \in L^1(Q;\R^d),
\end{align}
\end{subequations}
$(u,e,z,p,\teta)$ comply with
 the initial conditions
 \begin{subequations}
 \label{initial-conditions}
\begin{align}
 \label{iniu}  & u(0,x) = u_0(x), \ \ \dot{u}(0,x) = \dot{u}_0(x) & \forae\, x \in
 \Omega,
 \\
  \label{inie}  & e(0,x) = e_0(x)  & \forae\, x \in
 \Omega,
 \\
  \label{iniz}  & z(0,x) = z_0(x) & \forae\, x \in
 \Omega,
 \\
 \label{inip}  & p(0,x) = p_0(x) & \forae\, x \in
 \Omega,
 \\
 \label{initeta}
 &
\teta(0) = \teta_0  &  \text{ in } W^{1,\infty}(\Omega)^*,
\end{align}
\end{subequations}
and with  
\begin{itemize}
\item[-] the \emph{kinematic admissibility} condition
\begin{equation}
\label{kin-admis}
(u(t,x), e(t,x), p(t,x)) \in \mathcal{A}(w(t,x)) \qquad \foraa (t,x) \in Q;
\end{equation}
\item[-] the weak formulation \eqref{w-momentum-balance}
of the \emph{momentum balance} \eqref{mom-balance-intro-bis};
\item[-]  the  feasibility and unidirectionality constraints 
   \begin{equation}
   \label{unidir-feasib}
     z \in [0,1]  \ \text{ and } \  \dot z \leq 0 \text{ a.e.\ in $ \Omega \times (0,T)$}
   \end{equation}
  and the subdifferential inclusion for \emph{damage evolution}
  \begin{equation}
\label{dam-Hs-star-intro}
\subd \Didv {\dot z} + \nu \As(\dot z)+ \As (z) + W'(z) \ni - \frac12 \bbC'(z) e :e +\teta \qquad \text{in }\spz^* \ \aein\, (0,T);
\end{equation}
\item[-] 
the \emph{plastic flow rule}
\begin{equation}
\label{plastic-flow-ptw}
\partial_{\dot{p}}  \dip{z}{\teta}{\dot{p}} + \dot{p} \ni (\bbD(z) \dot{e} + \bbC(z) e - \teta \bbC(z) \bbE)_{\mathrm{D}}   \qquad \aein\, \Omega \times (0,T),
\end{equation}
\item[-]  the 
\emph{strict positivity} of $\teta$: 
\begin{equation}
\label{teta-strict-pos}
\exists\, \bar\teta>0  \  \foraa (t,x) \in Q\, : \quad \teta(t,x) > \bar\teta;
\end{equation}
and the weak formulation  of the \emph{heat equation} \eqref{heat-regu} for every test function $ \varphi \in   W^{1,\infty}(\Omega)$:
\begin{equation} \label{eq-teta}
\begin{aligned}
   &
\pairing{}{W^{1,\infty}(\Omega)}{\dot\teta}{\varphi}
+ \int_\Omega \condu(\teta) \nabla \teta\nabla\varphi \dd
x
\\
& = \int_\Omega \left(G+\mathbb{D}(z) \dot e {:} \dot{e} {-} \teta \bbC(z) \bbE {:}  \dot{e} {+} \did{\dot z} {+}|\dot z|^2 {+} \bar\nu\ass(\dot{z}, \dot{z}) 
{+}
\dip z{\teta}{\dot p} {+} |\dot p |^2   \right) \varphi  \dd x  + \int_{\partial\Omega} g \varphi   \dd S\,.
\end{aligned}
\end{equation}
\end{itemize}
\end{definition}
\noindent
Since the `viscous' contribution $\dot{z} \mapsto  \frac12 \| \dot{z}\|_{L^2(\Omega)}^2 $ is differentiable on $\spz$, by the sum rule
 we have that 
$\subd \Didv {\dot z}  = \subd \Did{\dot z} + J(\dot{z})$, where $J : L^2(\Omega) \to \spz^*$ is the embedding operator. Nevertheless, as  the spaces 
$(\spz, L^2(\Omega),\spz^*)$  form a Hilbert triple, in what follows we will omit the symbol $J$. Therefore, \eqref{dam-Hs-star-intro} rewrites as 
  \begin{equation}
\label{no-duality}
\begin{cases}
\omega+\dot{z} + \nu \As(\dot z)+ \As (z) + W'(z) = - \frac12 \bbC'(z) e :e +\teta,
\\
\omega \in \subd \Did {\dot z} 
\end{cases}
 \qquad \text{in }\spz^* \ \aein\, (0,T)\,.
\end{equation}
\par
We 
refer to the previously developed calculations, leading to the mechanical and total energy balances \eqref{mech-enbal} and \eqref{total-enbal}, for the proof of 
the following result. 
\begin{lemma}
\label{cor:total-enid} 
Let $(u,e,z,p,\teta)$ be a   \emph{weak energy solution} to (the Cauchy problem for) system
(\ref{regularized-system}, \ref{more-reg-bc}).
 Then,    for every $0\leq s \leq t \leq T$ the functions
  $(u,e,z,p,\teta)$
 comply
 with 
 the mechanical energy balance \eqref{mech-enbal} and with 
 \begin{equation}
\label{total-enbal-delicate}
\begin{aligned}
& 
\frac{\rho}2 \int_\Omega |\dot{u}(t)|^2 \dd x + \pairing{}{W^{1,\infty}(\Omega)}{\teta(t)}{1} + \calQ(e(t),z(t)) + \calG(z(t)) \\
&  = \frac{\rho}2 \int_\Omega |\dot{u}(s)|^2 \dd x + \pairing{}{W^{1,\infty}(\Omega)}{\teta(s)}{1} + \calQ(e(s),z(s)) + \calG(z(s))
\\
& \quad 
  + \int_s^t \pairing{}{H_\Dir^1 (\Omega;\R^d)}{\mathcal{L}}{\dot u{-} \dot w} 
    +\int_s^t \int_\Omega G \dd x \dd r + \int_s^t \int_{\partial\Omega} g \dd S \dd r
\\
& \quad   +\rho \left( \int_\Omega \dot{u}(t) \dot{w}(t) \dd x -  \int_\Omega \dot{u}_0 \dot{w}(0) \dd x - \int_s^t \int_\Omega \dot{u}\ddot w \dd x \dd r \right)     + 
\int_s^t \int_\Omega \sigma: \sig{\dot w} \dd x \dd r\,.
 \end{aligned}
\end{equation}
\end{lemma}
\noindent Observe that, since $\teta\in L^\infty(0,T; L^1(\Omega))$, there holds $ \pairing{}{W^{1,\infty}(\Omega)}{\teta(t)}{1}  = \int_\Omega \teta(t) \dd x  = \calF(\teta(t))$
 for almost all $t\in (0,T)$ and for $t=0$, and in that case \eqref{total-enbal-delicate}
coincides with the total energy balance \eqref{total-enbal}.
\subsection{Entropic solutions for the thermoviscoelastoplastic damage system}
\label{ss:2.3}
For the (Cauchy problem associated with the)  non-regularized system (\ref{PDE-INTRO}, \ref{BC-INTRO}),
we will be able to prove an existence result only for a solution concept containing much less information than that
from Def.\ \ref{def:weak-sols}. In particular, we will  notably weaken the formulations of the heat equation \eqref{heat-intro}, given in terms of an entropy inequality joint with the total energy inequality, and of the damage and plastic flow rules. In order  to motivate Definition 
\ref{def:entropic-sols} ahead, we develop some preliminary considerations \MMM on the weak formulation of the damage and plastic flow rules, and on the entropy inequality. The latter will be formally obtained from the heat equation 
\eqref{heat-intro} assuming the strict positivity of the temperature  $\teta$, which   shall be rigorously proved  (cf.\ Prop.\ \ref{prop:exist-discr} ahead). \EEE
\paragraph{{\em The entropy inequality}.}
It
can be formally obtained by multiplying the heat equation 
\eqref{heat-intro}
by $\varphi/\teta$, with $\varphi $
 a smooth and \emph{positive} test function. Integrating in  space and over  a generic interval $(s,t) \subset (0,T)$  leads to the identity
\begin{equation}
\label{formal-entropy-eq}
\begin{aligned}
  & \int_s^t \int_\Omega \partial_t \log(\teta) \varphi \dd x \dd r + 
    \int_s^t \int_\Omega \left( \condu(\teta) \nabla \log(\teta) \nabla \varphi - \condu(\teta) \frac\varphi\teta \nabla \log(\teta) \nabla   \teta  \right) \dd x \dd r  
  \\ & = \int_s^t \int_\Omega 
   \left( G+ 
    \mathbb{D}(z) \dot{e} : \dot{e} -\teta \mathbb{C}(z)\mathbb{E}  : \dot{e}
+ \did {\dot z} + |\dot z|^2  
 -\teta \dot z 
+ \dip{z}{\teta}{\dot{p}} + \dot{p}: \dot{p}
 \right) \frac{\varphi}\teta \dd x \dd r
 \\
 & \qquad 
   + \int_s^t \int_{\partial\Omega} g \frac\varphi\teta \dd x \dd r
\end{aligned}
\end{equation}
The entropy inequality \eqref{entropy-ineq} ahead  is indeed the $\geq$-estimate in \eqref{formal-entropy-eq}, with \emph{positive} test functions, and where the first integral on the left-hand side is integrated  in time.
\paragraph{{\em Weak solvability of the damage and plastic flow rules}.}
Setting
  $
   \xi: = \teta  - \tfrac12\bbC'(z)e : e - \dot{z}   - \As (z) - W'(z)$, 
     the subdifferential inclusion \eqref{flow-rule-dam-intro} for damage evolution reformulates as 
  $
  \xi \in \partial\did{\dot{z}} $ in $\Omega \times (0,T)$. By the $1$-homogeneity of $\didname$,    the latter  is in turn equivalent to the system of inequalities
  \begin{subequations}
  \label{heuristic-4-damage}
  \begin{align}
      \label{heuristic-4-damage-1}
  &
  \xi \zeta \leq \did{\zeta} && \text{ in } \Omega \times (0,T)  &&  \text{for all } \zeta \in \mathrm{dom}(\didname) = (-\infty,0],
  \\
    \label{heuristic-4-damage-2}
    &
  \xi \dot{z} \geq \did{\dot z}  && \text{ in } \Omega \times (0,T)\,. &&
  \end{align}
  \end{subequations}
  Therefore, the damage flow rule  \eqref{flow-rule-dam-intro} could be formulated in terms 
  of the  constraints 
  \eqref{unidir-feasib},
  of the integrated (in space) version of \eqref{heuristic-4-damage-1} with test functions $\zeta \in H_-^{\mathrm{s}}(\Omega)$,
   and of the integrated (in space and time) version of \eqref{heuristic-4-damage-2}, which can be interpreted as an  \emph{energy-dissipation} inequality. 
   Namely,  
   \begin{subequations}
   \label{KKT}
   \begin{itemize}
   \item[-] 
the \emph{one-sided} variational inequality
  \begin{equation}
  \label{1-sided-intro}
  \begin{aligned}
  \ass(z(t),\zeta) +
  \int_\Omega &  \left(  \did \zeta +  \dot{z}(t) \zeta + W'(z(t)) \zeta 
  +\tfrac12 \bbC'(z(t)) e {\colon} e \zeta -\teta(t) \zeta \right) \dd x \geq 0 
  \\
  &
  \qquad \qquad  \text{for all } 
  \zeta \in H_-^{\mathrm{s}}(\Omega) \ \foraa\, t \in (0,T);
  \end{aligned}
  \end{equation}
  \item[-]  the energy-dissipation inequality for damage evolution
  \begin{equation}
  \label{en-diss-ineq}
  \int_s^t \int_\Omega\left( \did {\dot z} + |\dot z|^2 \right)  \dd x \dd r
  + \nu\int_s^t  \ass (\dot{z}, \dot{z}) \dd r
   + \calG(z(t)) \leq  \calG(z(s)) 
  +  \int_s^t \int_\Omega \dot z \left({-}
  \tfrac12 \bbC'(z) e{:} e {+}\teta \right) \dd x,
  \end{equation}
  on sub-intervals $(s,t) \subset (0,T)$. 
  \end{itemize}
  \end{subequations}
 To our knowledge this formulation, inspired by the Karush-Kuhn-Tucker conditions,
was first proposed in the context of damage in \cite{HeiKra-1}. 
     \par
     Analogously, the plastic flow rule \eqref{flow-rule-plast-intro} can be weakly  formulated in terms of  the system of inequalities
     \begin{subequations}
     \begin{align}
    \label{1-sided-pl}
    &
    \Dip{z(t)}{\teta(t)}{\omega} \geq \int_\Omega \left( \sigma_\dev(t) {-} \dot{p}(t) \right) : \omega \dd x \quad \text{for all } \omega \in L^2(\Omega;\mt_\dev^{d\times d}) \ \foraa\, t \in (0,T),
    \\
    &
    \label{en-diss-pl}
     \int_s^t   \Dip{z(r)}{\teta(r)}{\dot{p}(r)}  +   \int_s^t\int_\Omega   |\dot{p}(r)|^2  \dd x \dd r \leq  \int_s^t\int_\Omega \sigma_\dev(r) : \dot{p}(r) \dd x \dd r 
          \end{align}
          \end{subequations}
          on sub-intervals $(s,t) \subset (0,T)$.   
   \paragraph{\emph{Entropic solutions for the non-regularized system.}}
We are now in the position to give our (extremely) weak solution concept for the initial-boundary value problem associated with system
\eqref{PDE-INTRO}, where,
in addition to the entropic formulation of the heat equation,
 the damage and plastic flow rules shall be formulated only through the 
Kuhn-Tucker type variational inequalities \eqref{1-sided-intro} and \eqref{1-sided-pl}, combined with the 
\emph{mechanical energy inequality} \eqref{uee-mech}  ahead. Observe that the latter corresponds to  the
\emph{sum} of the energy-dissipation inequalities for damage and plastic evolution \eqref{en-diss-ineq} and  \eqref{en-diss-pl}, with the weak formulation of the momentum balance tested by $(\dot{u}{-} \dot{w})$ and integrated in time.
\begin{definition}[Entropic solution to the thermoviscoelastoplastic damage system]
\label{def:entropic-sols}
$ \\ $ 
Given initial data $(u_0, \dot{u}_0, e_0,z_0, p_0,\teta_0)$ fulfilling \eqref{Cauchy-data}, we call a quintuple $(u,e,z,p,\teta)$
an \emph{entropic solution} to the Cauchy problem for system
 (\ref{PDE-INTRO}, \ref{BC-INTRO}),   if $(u,e,z,p,\teta)$  enjoy the summability and regularity properties \eqref{reg-u}--\eqref{reg-teta},
$(u,e,z,p)$ comply with the initial conditions \eqref{iniu}--\eqref{inip},  and if there hold
\begin{itemize}
\item[-]
the \emph{kinematic admissibility} condition \eqref{kin-admis};
\item[-]
the weak \emph{momentum balance} \eqref{w-momentum-balance};
\item[-]  the feasibility and  unidirectionality constraints \eqref{unidir-feasib}, joint with the  \emph{one-sided variational inequality}
  \eqref{1-sided-intro} for damage evolution;
  \item[-] the \emph{variational inequality} \eqref{1-sided-pl} for plastic evolution;
  \item[-] the mechanical energy inequality
  \begin{equation}
\label{uee-mech}
\begin{aligned}
& 
\frac{\rho}2 \int_\Omega |\dot{u}(t)|^2 \dd x +  \int_0^t\int_\Omega   \left(
 \bbD(z) \dot e: \dot e  + \did{\dot{z}} + |\dot{z}|^2 +  \dip{z}{\teta}{\dot p}  + |\dot p|^2 \right)  \dd x \dd r
+\calQ(e(t),z(t)) + \calG(z(t))
\\
&  \leq   \frac{\rho}2 \int_\Omega |\dot{u}(0)|^2 \dd x   + \calQ(e(0),z(0)) + \calG(z(0)) 
\\
 & \quad 
+ \int_0^t \pairing{}{H_\Dir^1 (\Omega;\R^d)}{\mathcal{L}}{\dot u{-} \dot w} \dd r  +  
 \int_0^t \int_\Omega  \left( \teta \bbC(z) \bbE : \dot e + \teta \dot{z} \right) \dd x \dd  r
 \\
& \quad   +\rho \left( \int_\Omega \dot{u}(t) \dot{w}(t) \dd x -  \int_\Omega \dot{u}(0) \dot{w}(0) \dd x - \int_0^t \int_\Omega \dot{u}\ddot w \dd x \dd r \right)  
  + 
\int_0^t \int_\Omega  \sigma : \sig{\dot w} \dd x \dd r, 
 \end{aligned}
\end{equation}
for every $t\in (0,T]$;
   \item[-]
 the \emph{strict positivity} \eqref{teta-strict-pos} of $\teta$ and 
  the \emph{entropy inequality}
 \begin{equation}
\label{entropy-ineq}
\begin{aligned}
  & \int_s^t \int_\Omega  \log(\teta) \dot{\varphi} \dd x \dd r -   \int_s^t \int_\Omega \left( \condu(\teta) \nabla \log(\teta) \nabla \varphi - \condu(\teta) \frac\varphi\teta \nabla \log(\teta) \nabla \teta\right)   \dd x \dd r  
  \\ 
& \leq
  \int_\Omega \log(\teta(t)) \varphi(t) \dd x -  \int_\Omega \log(\teta(s)) \varphi(s) \dd x  \\ & \quad
    - \int_s^t \int_\Omega \left( G+ 
    \mathbb{D}(z) \dot{e} : \dot{e} -\teta \mathbb{C}(z)\mathbb{E}  : \dot{e}
+ \did {\dot z} + |\dot z|^2 -\teta \dot z 
+ \dip{z}{\teta}{\dot{p}} + \dot{p}: \dot{p}
 \right) \frac{\varphi}\teta \dd x \dd r
 \\ & \quad   - \int_s^t \int_{\partial\Omega} g \frac\varphi\teta \dd x \dd r
\end{aligned}
\end{equation}
for all $\varphi $ in $L^\infty ([0,T]; W^{1,\infty}(\Omega)) \cap H^1 (0,T; L^{6/5}(\Omega))$  with $\varphi \geq 0$, 
for almost all $t\in (0,T)$, almost all $s\in (0,t)$, and for $s=0$ (with $\log(\teta(0))$ to be understood as $\log(\teta_0)$);
\item[-] the \emph{total energy inequality}
\begin{equation}
\label{total-uee}
\begin{aligned}
& 
\frac{\rho}2 \int_\Omega |\dot{u}(t)|^2 \dd x +\mathcal{E}(\teta(t), e(t), z(t)) 
\\
&  \leq \frac{\rho}2 \int_\Omega |\dot{u}(0)|^2 \dd x +\mathcal{E}(\teta(0), e(0),z(0))  + \int_0^t \pairing{}{H_\Dir^1 (\Omega;\R^d)}{\mathcal{L}}{\dot u{-} \dot w}  \dd r 
  +\int_0^t \int_\Omega G \dd x \dd r + \int_0^t \int_{\partial\Omega} g \dd S \dd r
\\
& \quad   +\rho \left( \int_\Omega \dot{u}(t) \dot{w}(t) \dd x -  \int_\Omega \dot{u}_0 \dot{w}(0) \dd x - \int_s^t \int_\Omega \dot{u}\ddot w \dd x \dd r \right)     + 
\int_0^t \int_\Omega \sigma: \sig{\dot w} \dd x \dd r
 \end{aligned}
 \end{equation}
 for almost all $t\in (0,T)$, almost all $s\in (0,t)$, and for $s=0$ (with $\teta(0)=\teta_0$).
\end{itemize}
\end{definition}
\begin{remark}
\label{rmk:clarification}
\upshape
A few comments on the above definition are in order:
\begin{enumerate}
\item While for weak energy solutions  it is possible to a posteriori deduce the validity of  mechanical and total energy \emph{balances} via suitable tests, 
here the upper energy inequalities \eqref{uee-mech} and \eqref{total-uee} have to be both claimed, as neither of them follows from the other items of the definition. 
\item
Observe that, subtracting the weak momentum balance \eqref{w-momentum-balance},  (legally) tested by $(\dot{u}{-}\dot{w})$  and integrated in time,
from the mechanical energy inequality \eqref{uee-mech}, it would be possible to deduce 
the \emph{joint energy-dissipation inequality for damage and plastic} evolution 
  \begin{equation}
  \label{joint-damage-plastic-endiss}
  \begin{aligned}
  &
   \int_0^t  \left( \Didv{\dot{z}(r)} {+}   \Dip{z(r)}{\teta(r)}{\dot{p}(r)}  \right) \dd r   +   \int_0^t\int_\Omega   |\dot{p}(r)|^2  \dd x \dd r 
   + \calG(z(t)) 
   \\
   &
   \leq  \calG(z(0)) 
  +  \int_0^t \int_\Omega \dot z \left({-}
  \tfrac12 \bbC'(z) e{:} e {+}\teta \right) \dd x
 + \int_0^t\int_\Omega \sigma_\dev(r) : \dot{p}(r) \dd x \dd r 
 \end{aligned}
 \end{equation}
only under the validity of  the chain rule
\begin{equation}
\label{ch-rule-wished}
\calQ(e(t),z(t)) - \calQ(e(0),z(0))  = \int_0^t \int_\Omega \left( \frac12 \bbC'(z) \dot{z} e{:} e {+} \bbC(z) \dot{e}{:} e \right) \dd x \dd r\,.
\end{equation}
However, \eqref{ch-rule-wished} can be only formally written: indeed, observe that the summability properties
$z\in H^1(0,T;L^2(\Omega))$ and $e \in L^\infty(0,T;L^2(\Omega;\mt_\sym^{d\times d}))$ do not ensure that  $\bbC'(z) \dot{z} e{:} e \in L^1(Q)$. 
\par
That is why,  in Definition \ref{def:entropic-sols} we  only claim the validity of the \emph{full} inequality \eqref{uee-mech}, which shall be obtained via lower semicontinuity arguments, by  passing to the limit in a discrete version of it.
\item Combining the information that $\teta \in L^2(0,T;H^1(\Omega)) \cap L^\infty (0,T;L^1(\Omega))$ with the strict positivity \eqref{teta-strict-pos} we infer that 
$\log(\teta) $ itself belongs to $ L^2(0,T;H^1(\Omega)) \cap L^\infty (0,T;L^1(\Omega))$.  \MMM The  regularity and summability requirements 
$\varphi \in L^\infty (0,T; W^{1,\infty}(\Omega)) \cap H^1 (0,T; L^{6/5}(\Omega))$ 
on  every admissible test function  for the entropy inequality \eqref{entropy-ineq}
 in fact guarantee 
  the integrals $\iint \log(\teta) \dot{\varphi} \dd x \dd r $ and  $\int_\Omega \log(\teta) \varphi \dd x$ are well defined since, in particular,  $L^{6/5}(\Omega))$ is the dual of $L^6(\Omega)$, which is the smallest Lebesgue space into which $H^1(\Omega)$ embeds in $d=3$. \EEE
 Furthermore,  with \eqref{entropy-ineq} we are also tacitly claiming the 
 summability properties
\[
\condu(\teta) | \nabla \log(\teta)|^2 \varphi \in L^1(Q), \qquad  \condu(\teta) \nabla \log(\teta) \in L^1(Q)\,.
\]
\item We refer to  \cite[Rmk.\ 2.6]{Rocca-Rossi} for some discussion on 
the consistency between the entropic (consisting of the entropy and total energy inequalities) and the classical formulations of the heat equation. 
\end{enumerate}
\end{remark}
\subsection{Existence results}
\label{ss:2.4}
We start by stating the existence of \emph{entropic} solutions to the (non-regularized) system (\ref{PDE-INTRO}, \ref{BC-INTRO})
under a mild growth condition on the thermal conductivity $\condu$. Observe that, with \eqref{strong-strict-pos} below  we will exhibit 
a precise lower bound for the  temperature in terms of quantities related to the material tensors 
$\bbC,\, \bbD,\, \bbE$. 
For shorter notation,  in the statement below, as well as in Thm.\ \ref{mainth:2} ahead,  we shall write
$(2.(\bbC,\bbD,\bbE))$ in place of \eqref{elast-visc-tensors-1}--\eqref{elast-visc-tensors-3} and \eqref{thermal-expansion}, 
   $(2.\mathrm{G})$ in place of \eqref{heat-source}, \eqref{dato-h}, and analogously for $(2.\mathrm{L})$, $(2.\mathrm{K})$, and
   $(2.\mathrm{W})$.
\begin{theorem}
\label{mainth:1}
Let $\nu=0$.
Assume
\eqref{Omega-s2},  $(2.(\bbC,\bbD,\bbE))$,   $(2.\mathrm{G})$,  $(2.\mathrm{L})$,  
 \eqref{Dirichlet-loading},   $(2.\mathrm{K})$,  and $(2.\mathrm{W})$.  
In addition,  suppose that 
\begin{equation}
\label{hyp-K}
\tag{2.$\condu_1$}
\begin{aligned}
& \text{the function }   \condu: \R^+ \to \R^+  \  \text{
is
 continuous and}
\\
& \exists \, c_0, \, c_1>0 \quad \exists\,   \mu>1   \ \
\forall\teta\in \R^+\, :\quad
c_0 (1+ \teta^{\mu}) \leq \condu(\teta) \leq c_1 (1+\teta^{\mu})\,.
\end{aligned}
\end{equation}
\par
Then, for every $(u_0, \dot{u}_0, e_0,z_0, p_0,\teta_0) $ satisfying \eqref{Cauchy-data} there exists an \emph{entropic} solution $(u,e,z,p,\teta)$ 
to the Cauchy problem for system (\ref{PDE-INTRO}, \ref{BC-INTRO})
 such that, in addition,
 \begin{enumerate}
 \item there exists $\zeta_*>0$ such that 
 \begin{equation}
 \label{strong-pos-z}
 z(x,t) \in [\zeta_*, 1 ] \qquad \text{for all } (x,t) \in  Q;
 \end{equation}
 \item
   $\teta$ complies with the positivity property 
\begin{equation}
\label{strong-strict-pos}
 \teta(x,t) \geq \bar\teta : = \left( C_* T + \frac1{\teta_*}\right)^{-1} \quad \text{for almost all $(x,t) \in Q$},
\end{equation}
where
$\teta_*>0$ is  from \eqref{initial-teta} and $C_*:= \frac{ \overline{C}^2 |\bbE|^2}{2 C_\bbD^1} $ with
$\overline{C} = \max_{z\in [0,1]} |\bbC(z)|$ and 
 $C_\bbD^1>0$ from \eqref{elast-visc-tensors};
\item
 there holds \MMM $\log(\teta) \in L^\infty(0,T;L^p(\Omega))$ for all $p\in [1,\infty)$, \EEE
  \begin{equation}
\label{further-logteta}
\begin{gathered}
\condu(\teta)\nabla \log(\teta) \in  L^{1+\bar\delta}(Q;\R^d)   \text{ with   $ \bar\delta = \frac{\alpha}\mu $ 
and $\alpha \in \MMM [0{\vee}(2{-}\mu), 1)$, \EEE and } \qquad  
\\ \condu(\teta)\nabla \log(\teta) \in L^1(0,T;X) \quad \text{with } X=
\left \{ \begin{array}{lll}
 L^{2-\eta}(\Omega;\R^d)  &   \text{ for all } \eta \in (0,1]  & \text{if } d=2,
\\
L^{3/2-\eta}(\Omega;\R^d)  &   \text{ for all } \eta \in (0,1/2]  & \text{if } d=3,
\end{array}
\right.
\end{gathered}
\end{equation}
\MMM (where $0{\vee}(2{-}\mu)=\max\{0,(2{-}\mu)\}$), \EEE 
so that the entropy inequality \eqref{entropy-ineq} in fact holds for all positive test functions $\varphi \in L^\infty ([0,T]; W^{1,d+\epsilon}(\Omega)) \cap H^1 (0,T; L^{6/5}(\Omega))$, for every $\epsilon>0$. 
\end{enumerate}
 \end{theorem}
\par
Under a more restrictive growth condition on $\condu$, we are able to establish the existence of \emph{weak energy} solutions
 for the regularized thermoviscoelastoplastic damage system. 
Let us also point out that we will be able to
enhance the temporal regularity of the temperature  and 
 obtain
 a  variational formulation of the heat equation with a wider class of test functions.
 We will also show the validity of  the entropy inequality \eqref{entropy-ineq}: this is a result on its own, as  \eqref{entropy-ineq} cannot 
 be inferred from the weak formulation  \eqref{eq-teta} of the heat equation, not even in the enhanced form established with Thm.\
 \ref{mainth:2}.
\begin{theorem}
\label{mainth:2}
Let $\nu>0$.
Assume
\eqref{Omega-s2},  $(2.(\bbC,\bbD,\bbE))$,   $(2.\mathrm{G})$,  $(2.\mathrm{L})$,  
 \eqref{Dirichlet-loading},   $(2.\mathrm{K})$,  and $(2.\mathrm{W})$.   In addition to \eqref{hyp-K}, suppose that 
\begin{equation}
\label{hyp-K-stronger}
\tag{2.$\condu_2$}
\begin{cases}
\mu \in (1,2) & \text{if } d=2,
\\
\mu \in \left(1, \frac53\right) & \text{if } d=3.
\end{cases}
\end{equation}
\par
Then, for every $(u_0, \dot{u}_0, e_0,z_0, p_0,\teta_0) $ satisfying \eqref{Cauchy-data} there exists a \emph{weak energy}  solution $(u,e,z,p,\teta)$ 
to the Cauchy problem for system 
(\ref{regularized-system}, \ref{more-reg-bc}) such that, in addition, $z$ fulfills \eqref{strong-pos-z},  $\teta$ complies with the positivity property 
\eqref{strong-strict-pos},
and with 
 \begin{equation}
 \label{further-k-teta}
 \hat{\condu}(\teta)  \in L^{1+\tilde{\delta}}(Q)  \text{ for some $\tilde\delta \in \left( 0, \tfrac13\right)$},
 \end{equation}  
 cf.\ \eqref{specific-tilde-delta} ahead,  where $\hat{\condu}$ is  a primitive of $\condu$.
 Therefore,
  \eqref{eq-teta} in fact holds for all test functions $\varphi \in W^{1,q_{\tilde \delta}}(\Omega)$, with $q_{\tilde \delta}= 1+\tfrac 1{\tilde \delta} > 4$. Ultimately, $\teta $ 
  has the enhanced regularity $\teta \in W^{1,1}(0,T;   W^{1,q_{\tilde \delta}}(\Omega)^*)$. 
  \par
  Finally,   $(u,e,z,p,\teta)$  comply with the 
  entropy inequality  \eqref{entropy-ineq} associated with the ``regularized'' heat equation \eqref{heat-regu} (i.e., featuring the additional term
  $-\int_s^t \int_\Omega \bar\nu \ass(\dot{z},\dot{z}) \frac{\varphi}\teta \dd x \dd r $ on the right-hand side),
  for almost all $t\in (0,T)$ and almost all $s\in (0,t)$.
 \end{theorem}
\par
 The proof of Theorems  \ref{mainth:1} \&   \ref{mainth:2}
 shall be developed throughout Secs.\ \ref{s:3}--\ref{s:4} by passing to the limit in a carefully devised time discretization scheme,
 along the footsteps of the analysis previously developed in \cite{Rocca-Rossi, Rossi2016}.
 \par
 As we will illustrate in Remark \ref{rmk:1-from-2}, it would also be possible to prove the existence of entropic solutions to the  (Cauchy problem for the) thermoviscoelastoplastic system (\ref{PDE-INTRO}, \ref{BC-INTRO})  by an alternative method. Namely, we could pass to the limit  as the  regularization parameter $\nu \downarrow 0$ in the weak energy formulation of the regularized system, featuring a family $(\condu_\nu)_\nu$ of thermal conductivities
 fulfilling \eqref{hyp-K-stronger} and 
  suitably converging as $\nu \down 0$ to a function $\condu$ that only complies with  \eqref{hyp-K}. However,
 to avoid overburdening the paper,
  we have chosen not to develop this asymptotic analysis. 
  \subsection{Continuous dependence on the external and initial data \MMM in the case of a prescribed temperature profile\EEE}
  \label{ss:2.5}
  \noindent
Let us now confine the discussion to the regularized system \MMM to the case the temperature profile is prescribed. \EEE 
 Namely we consider the
the PDE system consisting of the momentum balance \eqref{mom-balance-intro-bis},  of the regularized damage flow rule 
\eqref{visc-flow-rule-regu}, and of the plastic flow rule 
\eqref{pl-flow-rule-bis},
 with a \emph{given} temperature 
\begin{equation}
\label{Theta-given}
\Theta \in L^2(Q;\R^+). 
\end{equation}
 In this context, weak energy solutions fulfill the weak momentum balance \eqref{w-momentum-balance},
 the subdifferential inclusion for damage evolution \eqref{dam-Hs-star-intro} (i.e.,\eqref{no-duality}),
  and 
  the pointwise formulation 
 \eqref{plastic-flow-ptw}  of the plastic flow rule. 
 \par
 We aim at providing 
 a continuous dependence estimate for weak energy  solutions  in terms of  the  initial and external  data, in particular obtaining their uniqueness.  
 To this end we shall have to introduce a further, quite strong simplification. Namely,  we shall assume that the plastic dissipation potential $\dipname$ neither depends on the temperature, nor on the damage variable, and we thus restrict to a functional
  \begin{equation}
  \label{H-simpler}
  \dipname: \Omega \times \mt_\dev^{d\times d} \to [0,+\infty) \text{ lower semicontinuous, convex, $1$-positively homog., and fulfilling 
  \eqref{linear-growth}.}
  \end{equation}
  Indeed, while the dependence of $\dipname$ on the fixed temperature profile could be kept, proving continuous dependence/uniqueness results in the case of a \emph{state-dependent} dissipation potential is definitely more arduous (cf.\ e.g.\ \cite{Brokate-Krejci-Schnabel}, 
  \cite{Mielke-Rossi07} for some results in the context of abstract hysteresis/rate-independent systems), and outside the scope of the present contribution. 
  Finally, for technical reasons we will have to strengthen the regularity of $\bbC$  and $\bbD$ and require that 
  \begin{equation}
\label{elast-visc-tensors-4} 
\tag{2.${\small (\bbC,\bbD)_4}$}
\begin{aligned}
&
\forall\, C_Z>0  \ 
\exists\, \tilde{L}_{\bbC}>0 \  \forall\, x\in \Omega \, :
 \ |z_1|,\, |z_2|  \leq C_Z \ \Rightarrow 
 \  | \bbC'(x,z_1){-} \bbC'(x,z_2) | \leq \tilde{L}_{\bbC} |z_1{-}z_2|\,,
\\
& 
\forall\, C_Z>0  \ 
\exists\, L_{\bbD}>0 \  \forall\, x\in \Omega \, :
 \ |z_1|,\, |z_2|  \leq C_Z \ \Rightarrow 
 \  | \bbD(x,z_1){-} \bbD(x,z_2) | \leq L_{\bbD} |z_1{-}z_2|\,.
\end{aligned}
\end{equation}
  \par
  In this context we have the following result, where we now write $(2.(\bbC,\bbD,\bbE))$ as a place-holder for 
  $(2.(\bbC,\bbD)_1) + (2.(\bbC,\bbD)_2) + (2.(\bbC,\bbD)_3) + (2.(\bbC,\bbD)_4) + (2.\bbE)$. 
  \begin{proposition}
  \label{prop:continuous-dependence}
  Let $\nu>0$.
Assume
\eqref{Omega-s2},
$(2.(\bbC,\bbD,\bbE))$, 
 $(2.\mathrm{W})$,  and \eqref{H-simpler}.  
 Let $(F_i,f_i,w_i)$  and $(u_i^0, \dot{u}_i^0, z_i^0, p_i^0)$, for $i=1,2$,  be two sets of external and initial data for the regularized viscoplastic damage system
  (\ref{mom-balance-intro-bis}, \ref{visc-flow-rule-regu}, \ref{pl-flow-rule-bis}), with boundary conditions
  \eqref{more-reg-bc}
  and 
  with  given temperature profiles $\Theta_i \in L^2(Q;\R^+)$, $i=1,2$. Suppose that
  the data  $(F_i,f_i,w_i)$  and $(u_i^0, \dot{u}_i^0, z_i^0, p_i^0)$ comply with 
  $(2.\mathrm{L})$, $(2.\mathrm{w})$, and \eqref{Cauchy-data}.
  Let $(u_i,e_i,z_i,p_i)$, $i=1,2$,  be corresponding weak energy  solutions to the initial-boundary value problem for system 
   (\ref{mom-balance-intro-bis}, \ref{visc-flow-rule-regu}, \ref{pl-flow-rule-bis}).
  \par
  Set  \MMM $P: = \max_{i=1,2}\{ \|e_i\|_{L^2(\Omega;\mt_\sym^{d\times d})} + \|z_i\|_{L^\infty(\Omega)}   \}.$
Then, there exists a positive constant $C_P$  depending on $P$ \EEE 
 such that 
  \begin{equation}
  \label{cont-dependence-estimate}
\begin{aligned}
&
\| u_1{-}u_2\|_{W^{1,\infty}(0,T;L^2(\Omega;\R^d)) \cap H^1(0,T;H^1(\Omega;\R^d))} + \|e_1{-}e_2\|_{H^1(0,T;L^2(\Omega;\mt_\sym^{d\times d})} 
\\
& \quad 
 + \|z_1{-}z_2\|_{H^1(0,T;\spz)}
   + \|p_1{-}p_2\|_{H^1(0,T;L^2(\Omega;\mt_\sym^{d\times d})}
   \\
   & \leq \MMM C_P \EEE  \Big( \| u_1^0{-} u_2^0\|_{H^1(\Omega;\R^{d\times d})} + \| \dot{u}_1^0{-} \dot{u}_2^0\|_{L^2(\Omega;\R^{d\times d})}
   + \| e_1^0{-} e_2^0\|_{L^2(\Omega;\mt_\sym^{d\times d})} 
   + \|z_1^0 {-} z_2^0 \|_{\spz} 
   \\
   & \quad \quad 
   + \| p_1^0{-} p_2^0\|_{L^2(\Omega;\mt_\sym^{d\times d})}
   + \|F_1{-}F_2\|_{L^2(0,T;H_\Dir^1(\Omega;\R^d)^*)} +  \|f_1{-}f_2\|_{L^2(0,T;H_{00,\Gdir}^{1/2}(\Gneu;\R^d)^*)}
   \\
   &
   \quad  \quad 
    + \|w_1{-}w_2\|_{H^1(0,T;H^1(\Omega;\R^d)) \cap W^{2,1}(0,T;L^2(\Omega;\R^d))}
   + \|\Theta_1{-}\Theta_2 \|_{L^2(Q)} 
   \Big) 
\end{aligned}
\end{equation} 
In particular, the initial boundary value problem for the regularized viscoplastic damage   system \MMM with prescribed temperature \EEE admits a unique solution. 
  \end{proposition}
\section{\bf Time discretization of the thermoviscoelastoplastic damage system(s)}
\label{s:3}
\noindent 
In all of the results of  this section, we will tacitly assume all of the conditions listed in Section \ref{ss:2.1}.
\subsection{The time discrete scheme}
\label{ss:3.1}
We will  consider a unified   discretization scheme for both the regularized  thermoviscoelastoplastic damage system 
(\ref{regularized-system}, \ref{more-reg-bc})
and for system (\ref{PDE-INTRO}, \ref{BC-INTRO}).
Therefore, within this section, the parameter $\nu$ 
modulating the viscous regularizing contribution to \eqref{visc-flow-rule-regu} shall be considered as \underline{$\nu\geq 0$}.
\par
Given
a partition of $[0,T]$ with constant time-step $\tau>0$ and
nodes $t_\tau^k:=k\tau$, $k=0,\ldots,K_\tau$,  we
approximate the data  $F$, $f$,  $G$, and $g$ 
 by local means:
\begin{equation}
\label{local-means} 
\begin{gathered}
\Ftau{k}:= 
\frac{1}{\tau}\int_{t_\tau^{k-1}}^{t_\tau^k}  F(s)\dd s\,,
\quad 
\ftau{k}:= 
\frac{1}{\tau}\int_{t_\tau^{k-1}}^{t_\tau^k}  g(s)\dd s\,,
\quad
 \Gtau{k}:= \frac{1}{\tau}\int_{t_\tau^{k-1}}^{t_\tau^k} G(s)
\dd s\,, \quad \gtau{k}:= \frac{1}{\tau}\int_{t_\tau^{k-1}}^{t_\tau^k} g(s)
\dd s
\end{gathered}
\end{equation}
for all $k=1,\ldots, K_\tau$.  From the terms $\Ftau{k}$ and $\ftau{k}$ one then defines the elements $ \Ltau{k}$, which are the local-mean approximations of $\calL$. 
Hereafter, given elements $(v_{\tau}^k)_{k=1,\ldots, K_\tau}$, 
we will use the notation
\begin{equation}
\label{discrete-notation}
\Dtau{k} v: = \frac{v_{\tau}^{k} - v_{\tau}^{k-1}}\tau, \qquad  \Ddtau{k} v: = \frac{\vtau{k} -2 \vtau{k-1} + \vtau{k-2}}{\tau^2}.
\end{equation}
\par
We construct discrete solutions to the (regularized) thermoviscoelastoplastic system   by recursively solving an elliptic system,  cf.\ the forthcoming Problem 
\ref{prob:discrete}, where the weak formulation of the discrete heat equation features the function space
 \begin{equation}
\label{X-space}
 \spt:= \{ \theta  \in H^1(\Omega)\, : \   \condu(\theta) \nabla \theta  \nabla v  \in L^1(\Omega)   \text{ for all } v \in H^1 (\Omega)\},
 \end{equation}
 and, for $k \in \{1,\ldots, K_\tau\}$,  the elliptic operator
 \begin{equation}
 \label{elliptic-k}
 {A}^k: \spt  \to  H^1(\Omega)^* \text{  defined by }\\
\pairing{}{H^1(\Omega)}{ {A}^k(\theta) }{v}:=
  \int_\Omega \condu(\theta) \nabla \theta \nabla v \dd x - \int_{\partial \Omega} \gtau{k} v \dd S\,.
\end{equation} 
Furthermore, for technical reasons (cf.\ Remark \ref{rmk:tech-comments} ahead), we shall  
 add the regularizing term
$-\tau \mathrm{div}(|\etau k|^{\gamma- 2} \etau k)$ to the discrete momentum equation, 
 as well as  $\tau |\ptau k|^{\gamma-2} \ptau k$ to the discrete  plastic flow rule, respectively. We will take   $\gamma>4$. 
 That is why, we will seek for discrete solutions with  $\etau k\in L^{\gamma} (\Omega;\mt_{\sym}^{d\times d}) $ 
  and $\ptau k  \in L^{\gamma} (\Omega;  \mt_{\dev}^{d\times d})$,  
  giving $\sig{\utau{k}} \in  L^{\gamma} (\Omega;\mt_{\sym}^{d\times d})$  by the kinematic admissibility condition and thus, via Korn's inequality
   \eqref{Korn}, $\utau k \in W_\Dir^{1,\gamma} (\Omega;\R^d)$.
   \par
     Because of these regularizations, it will be necessary to supplement the discrete system with  approximate initial data
\begin{subequations}
\label{approx-e_0}
\begin{align}
& 
\label{approx-e_0-1}
(e_\tau^0)_\tau  \subset   L^{\gamma} (\Omega;\mt_{\sym}^{d\times d})  & \text{ such that } \lim_{\tau\down 0} \tau^{1/\gamma} \| e_\tau^0\|_{L^{\gamma} (\Omega;\mt_{\sym}^{d\times d})} =0 &  \text{ and } e_\tau^0 \to e_0 \text{ in $L^{2} (\Omega;\mt_{\sym}^{d\times d})$},
\\
& 
\label{approx-e_0-2}
(p_\tau^0)_\tau  \subset   L^{\gamma} (\Omega;\mt_{\dev}^{d\times d})  & \text{ such that }  \lim_{\tau\down 0}   \tau^{1/\gamma} \| p_\tau^0\|_{L^{\gamma} (\Omega;\mt_{\dev}^{d\times d})} =0 &  \text{ and } p_\tau^0 \to p_0 \text{ in $L^{2} (\Omega;\mt_{\dev}^{d\times d})$}.
\end{align}
By consistency with  the kinematic admissibility condition at time $t=0$, we will also approximate the initial datum $u_0$ with a family 
$(u_\tau^0)_\tau  \subset   W^{1,\gamma} (\Omega;\R^d)$ such that 
\begin{equation}
\label{approx-e_0-3}
(u_\tau^0)_\tau  \subset   W^{1,\gamma} (\Omega;\R^d)   \text{ such that } \lim_{\tau\down 0} \tau^{1/\gamma} \| u_\tau^0\|_{W^{1,\gamma} (\Omega;\R^d)} =0   \text{ and } u_\tau^0 \to u_0 \text{ in $H^{1} (\Omega;\R^d)$}.
\end{equation}
\end{subequations}
\MMM The data $(u_\tau^0)_\tau $ may be constructed by a perturbation technique. \EEE
In connection with the 
regularization of the discrete momentum balance,  we will have to approximate the Dirichlet loading $w$ by a family $(w_\tau)_\tau \subset \mathbf{W} \cap 
W^{1,1}(0,T; W^{1,\gamma}(\Omega;\R^d))$, where we have used the place-holder
$ \mathbf{W}:=  L^1(0,T; W^{1,\infty} (\Omega;\R^d)) \cap W^{2,1} (0,T;H^1(\Omega;\R^d))  \\ \cap H^2(0,T; L^2(\Omega;\R^d))$. We will require that 
\begin{equation}
\label{discr-w-tau}
w_\tau \to w \text{ in } \mathbf{W} \text{ as } \tau \downarrow 0, \quad  \text{ as well as }  \quad
\exists\, \alpha_w \in \left(0,\frac1\gamma\right)  \text{ s.t. }    \sup_{\tau>0} \tau^{\alpha_w} \| \sig{\dot{w}_\tau}\|_{L^\gamma (\Omega;\mt_\sym^{d\times d})} \leq C<\infty\,.
\end{equation}
We will then consider the discrete data
\[
 \wtau{k}:=
\frac{1}{\tau}\int_{t_\tau^{k-1}}^{t_\tau^k} w_\tau(s)\dd s\,.
\]
\par
For technical reasons related to the proof of  Prop.\ \ref{prop:exist-discr} (cf.\ \eqref{expedient} ahead), it will be
expedient to  replace the argument of 
  the elasticity tensor $\bbC$ with its positive part. We will proceed in this way in the thermal expansion terms  contributing to the momentum balance and to the heat equation. 
%
Since we will ultimately prove that the discrete damage solutions are confined to the admissible interval $[0,1]$, 
cf.\ \eqref{discr-feasibility} in Prop.\ \ref{prop:aprio} ahead,  the restriction to the positive part in the argument of  $\bbC$ will ``disappear'' in the end. 
\par
Finally, in the discrete version of  the damage flow rule 
(where  we will stay with the notation \eqref{no-duality}), 
we will resort to the convex-concave decomposition $W (z)=\beta(z)- \tfrac{\lambda_W}2  |z|^2$ from \eqref{decomp-W}, 
 with  $\lambda_W>0$ and $\beta\in \mathrm{C}^2(\R^+)$  convex. 
 For shorter notation, in what follows  we will use the place-holders
 \[
 \begin{aligned}
 &
 \bsp: =  W_\Dir^{1,\gamma} (\Omega;\R^d)  \times L^\gamma (\Omega;\mt_\sym^{d\times d}) \times \spz\times 
L^\gamma (\Omega;\mt_\dev^{d\times d}) \times H^1(\Omega),
\\
&
 \rbsp : = W_\Dir^{1,\gamma} (\Omega;\R^d)  \times L^\gamma (\Omega;\mt_\sym^{d\times d}) \times 
L^\gamma (\Omega;\mt_\dev^{d\times d}) \times H^1(\Omega)
\end{aligned}
 \]
 to indicate the state spaces for the solutions to system \eqref{syst:discr} below.
%
\begin{problem}
\label{prob:discrete}
Let $\gamma>4$. Using the notation \eqref{discrete-notation} and  starting from data
\begin{align}\label{discr-Cauchy}
\utau{0}:=u_\tau^{0},\qquad
\utau{-1}:=u_{0}-\tau \dot{u}_\tau^0,  \qquad \etau{0}: = e_\tau^0,
\qquad \ztau{0}: = z_\tau^0, 
 \qquad \ptau{0}:=p_\tau^{0},\qquad  \tetau{0}:=\teta_{0}, 
\end{align}
for all $k=1,\ldots, K_\tau$, given 
 $(\utau{k-1},\etau{k-1},\ztau{k-1}, \ptau{k-1},\tetau{k-1}) \in \bsp
 $, find $\ztau k \in \spz$ 
 fulfilling 
 \begin{subequations}
\label{syst:discr}
\begin{itemize}
\item[-] the discrete damage flow rule
\begin{equation}
\label{discr-dam-eq}
\begin{aligned}
\omegatau k    +   \Dtau{k}z &  + \nu \As( \Dtau{k}z)   + \As (\ztau{k})+\beta'(\ztau k) -\lambda_W \ztau{k-1} 
\\
& \quad 
  = - \tfrac12 
  \bbC'(\ztau k)
   \etau{k-1} : \etau{k-1} + \tetau{k-1} \quad \text{in } \spz^*  \quad \text{with } 
  \omegatau{k} \in 
  \partial \Did{ \Dtau{k}z}.
  \end{aligned}
\end{equation}
\end{itemize}
Given $(\utau{k-1},\etau{k-1},\ztau{k-1},\ptau{k-1},\tetau{k-1}) \in \bsp$ and $\ztau k \in \spz$, 
 find 
 $(\utau{k}, \etau{k},  \ptau{k}, \tetau{k}) \in \rbsp$ fulfilling 
\begin{itemize}
\item[-] the kinematic admissibility $(\utau{k}, \etau{k}, \ptau{k}) \in \mathcal{A}(\wtau k)$ (in the sense of \eqref{kin-adm}); 
\item[-] the discrete momentum balance 
\begin{equation}
\label{discrete-momentum}
\rho\int_\Omega \Ddtau k u  v \dd x + \int_\Omega \sitau{k} : \sig v \dd x     = \pairing{}{H_\Dir^1(\Omega;\R^d)}{\Ltau k }{v} \quad
 \text{for all } v \in W_\Dir^{1,\gamma}(\Omega;\R^d),
\end{equation}
where we have used the place-holder
$
\sitau{k}:= \bbD(\ztau{k})    \Dtau k e  + \bbC(\ztau{k}) \etau k + \tau |\etau{k}|^{\gamma-2} \etau k- \tetau k{\bbC}((\ztau{k})^+)\bbE\,;
$
\item[-] the discrete plastic flow rule
\begin{equation}
\label{discrete-plastic}\zetau k  +  \Dtau kp + \tau |\ptau {k}|^{\gamma-2}\ptau k = \sidevtau{k} \quad \text{with }
 \zetau k \in \partial_{\dot p} \dip{\ztau{k}}{\tetau{k-1}}{\Dtau k p},  \qquad \aein \Omega;
\end{equation}
\item[-] $\tetau k \in \spt$ and  the discrete heat equation
\begin{equation}
\label{discrete-heat} \begin{aligned}
 \Dtau{k}\teta + A^k (\tetau{k})  &  = \Gtau{k} + 
  \bbD(\ztau{k})      \Dtau{k} e :  \Dtau{k} e -\tetau{k}  {\bbC}((\ztau{k})^+)\bbE :  \Dtau{k} e
  \\
  & \quad 
  + \did{\Dtau{k} z}  +   \left|  \Dtau{k} z \right|^2 + \bar{\nu} \ass  (\Dtau{k} z, \Dtau{k} z) - \tetau{k-1} \Dtau{k} z
  \\ & \quad 
  +
 \dip{\ztau{k}}{\tetau{k-1}}{ \Dtau{k} p} +  \left|  \Dtau{k} p \right|^2
   \quad \text{in } H^1(\Omega)^*;
\end{aligned} \end{equation} 
\end{itemize}
\end{subequations}
\end{problem}
\begin{remark}
\label{rmk:tech-comments}
\upshape
The discrete system \eqref{syst:discr} has been designed in such a way as to ensure the validity of
 a discrete total energy inequality, cf.\ \eqref{total-enid-discr} ahead.  The latter  will be proved by exploiting
  suitable cancellations 
   of the various terms contributing to  \eqref{syst:discr}, as well as 
  the  convex-concave decomposition
  \eqref{decomp-W} of $W$
   in the  discrete damage flow rule \eqref{discr-dam-eq}, where the contribution $\beta'(\ztau k)$ from the convex part has been kept implicit, 
   while the term $-\lambda_W \ztau{k-1}$ related to the concave part is explicit. 
The convexity of $z\mapsto {\bbC}(z)$ will also be a  key ingredient in the proof of  \eqref{total-enid-discr}, cf.\ the calculations in the proof of Lemma \ref{l:aprio-M}. 
\par
Several terms in  \eqref{syst:discr} have been kept \emph{implicit}, not only towards the validity   of \eqref{total-enid-discr}, but also in view of the strict positivity property \eqref{discr-strict-pos} ahead for the discrete temperature. Our proof  of \eqref{discr-strict-pos}  requires that $\tetau k$ is implicit in the  thermal expansion coupling term on the r.h.s.\ of the discrete heat equation,
cf.\ the calculations leading to \eqref{var-ineq-4-pos}, which also rely on the truncation of the elasticity tensor $\bbC$. 
 Therefore it has to be  implicit
in the corresponding terms in the discrete momentum balance and in the discrete plastic flow rule, which cannot  be thus  decoupled one from another. 
Instead, still compatibly with the proof of \eqref{discr-strict-pos},  the discrete damage flow rule  is decoupled from the other equations.
 This will greatly simplify the proof of existence of solutions to   \eqref{syst:discr}.
\par Because of this implicit character of the thermoviscoplastic subsystem, in order to prove the existence of solutions (to an approximate version of it),
we will have to resort to a (nonconstructive) existence result, of fixed point type, for elliptic systems
involving pseudo-monotone coercive operators. 
 The regularizing terms
 $-\tau \mathrm{div}(|\etau k|^{\gamma- 2} \etau k)$ and   $\tau |\ptau k|^{\gamma-2} \ptau k$, 
  guaranteeing enhanced integrability properties, have to ensure the coercivity of the pseudo-monotone operator
underlying the (approximate versions of  the) 
discrete momentum balance, plastic flow rule, and heat equation.
These terms will vanish in the limit as $\tau\downarrow 0$.
Let us point out that, thanks to them, the right-hand side  of the discrete heat equation is indeed an element in $H^1(\Omega)^*$.
Thus all the calculations 
performed in the proof of Proposition  \ref{prop:aprio}  and involving suitable tests of the discrete heat equation  will be rigorous. 
\end{remark}
\begin{proposition}[Existence of discrete solutions]
\label{prop:exist-discr}
Under  the growth condition \eqref{hyp-K} on $\condu$, Problem \ref{prob:discrete} admits a solution 
 $\{(\tetau{k}, \utau{k}, \etau{k}, \ptau{k})\}_{k=1}^{K_\tau}$. Furthermore, any solution  to  Problem \ref{prob:discrete} fulfills
\begin{equation}
\label{discr-strict-pos}
\tetau{k}\geq \bar{\teta}>0 \qquad \text{for all } k =1, \ldots, K_\tau  \text{ with  $\bar{\teta}$ from \eqref{strong-strict-pos}}\,.  
\end{equation}
\end{proposition} 
\subsection{Proof of Proposition \ref{prop:exist-discr}}
\label{ss:3.2}
First of all, let us point out that the \underline{strict positivity} \eqref{discr-strict-pos} ensues by the very same argument developed in the proof of \cite[Lemma 4.4]{Rocca-Rossi}. We shortly recapitulate it: 
 From the discrete heat equation \eqref{discrete-heat} we deduce the variational inequality
\begin{equation}
\label{var-ineq-4-pos}
\int_\Omega   \Dtau{k}\teta  w \dd x + \int_\Omega \condu(\tetau k) \nabla \tetau k \nabla w \dd x \geq - C_* \int_\Omega (\tetau k)^2 w \dd x \quad \text{ for all } w \in H_+^1(\Omega)\,.
\end{equation}
 with $C_*= \tfrac{ \overline{C}^2 |\bbE|^2}{2 C_\bbD^1} $. 
To establish \eqref{var-ineq-4-pos}, we estimate
\begin{equation}
\label{expedient}
\begin{aligned}
\bbD(\ztau{k})      \Dtau{k} e :  \Dtau{k} e -\tetau{k}  {\bbC}((\ztau{k})^+)\bbE :   \Dtau{k} e   &  \geq C_\bbD^1 | \Dtau{k} e |^2 - \overline{C} |\bbE| |\tetau{k}|| \Dtau{k} e |
\\ & 
\geq \frac{C_\bbD^1}2 | \Dtau{k} e |^2 - C_*|\tetau{k}|^2\,.
\end{aligned}
\end{equation}
For this,  we have used:  (1) the coercivity of $\bbD$ from \eqref{elast-visc-tensors-1},  (2) the fact that
$ 0 \leq (\ztau{k})^+ \leq 1 $ in $\Omega$ (due to $\ztau{k} \leq \ztau{k-1} \leq .... \leq z_0 \leq	 1$ by the unidirectionality enforced 
by the dissipation potential $\didname$ and condition  \eqref{initial-z} on $z_0$), so that 
 $|{\bbC}((\ztau{k})^+)|\leq \overline{C}: = \max_{z\in [0,1]} |\bbC(z)|$, and    (3) Young's inequality. 
 We also take into account the positivity of all the other terms on the right-hand side of \eqref{discrete-heat}: in particular, note that 
\begin{equation}
\label{quoted-before}
- \tetau{k-1} \Dtau{k}{z} \geq 0 \qquad \aein \, \Omega\,
\end{equation}
as we may suppose by induction  that $\tetau{k-1}>0$ a.e.\ in $\Omega$, whereas $ \Dtau{k}{z} \leq 0$ by unidirectionality. 
In view of \eqref{var-ineq-4-pos}, we may compare the elements $(\tetau k)_{k=1}^{K_\tau}$ with the  decreasing sequence $(\theta^k)_{k=1}^{K_\tau}$, recursively defined by
\[
\frac{\theta^k -\theta^{k-1}}{\tau} = - C_* (\theta^k)^2, \qquad \theta^0: = \teta_*>0\,,
\] 
and conclude, on the one hand,  that $\tetau k \geq \theta^k$ for all $k =1,\ldots, K_\tau$. On the other hand, the argument from 
\cite[Lemma 4.4]{Rocca-Rossi} yields that $\theta^k \geq  \ldots \ge \theta^{K_\tau} \geq  \bar\teta$ for all $k =1,\ldots, K_\tau$, which leads to  \eqref{discr-strict-pos}.
\par
In proving the \emph{existence} of solutions to  Problem \ref{prob:discrete}, we will perform the following steps:
\begin{itemize}
\item[\textbf{Step $1$:}]
While    the  discrete damage flow rule will be solved by a variational argument,
 we will approximate  the discrete thermoviscoplastic subsystem by truncating 
  the heat conductivity coefficient in the elliptic operator.
  On the one hand, this will allow us to apply the existence result from \cite{Roub05NPDE}, based on the theory of pseudomonotone operators, 
  for proving the existence of solutions. On the other hand, due to this truncation we will no longer be able 
to  exploit the growth of $\condu$
 in order to handle the thermal expansion term on the r.h.s.\ of \eqref{discrete-heat}. Therefore, in this term
  we  will replace  $\tetau k$ 
by its truncation $\calT_M(\tetau k)$. We shall do the same for the corresponding coupling terms in the discrete momentum balance and plastic flow rule. 
\item[\textbf{Step $2$:}] Existence of solutions to the approximate discrete thermoviscoplastic subsystem. 
\item[\textbf{Step $3$:}] A priori estimates on the discrete solutions, uniform with respect to the truncation parameter $M$
\item[\textbf{Step $4$:}]  Limit passage as $M\to\infty$.
\end{itemize}
\par
In completing Steps 2--4 we will most often have to adapt analogous arguments developed in \cite{Rocca-Rossi,Rossi2016}, to which we will refer for all details.
\EEE
\par
\noindent 
\textbf{Step $1$:  The approximate discrete system}  will 
feature the truncation operator 
\begin{equation}
\label{def-truncation-m} 
\mathcal{T}_M : \R \to \R, \qquad 
\mathcal{T}_M(r):= \left\{ \begin{array}{ll} -M &
\text{if } \  r <-M,
\\
r   & \text{if } \  |r| \leq M,
\\
M  & \text{if } \  r >M,
\end{array}
\right.
\end{equation}
where  we suppose that  $M\in \N\setminus\{0\}$. 
We thus introduce the  truncated heat conductivity
\begin{equation}
\label{def-k-m} \condu_M(r):= \condu(\mathcal{T}_M(r)) := \left\{ \begin{array}{ll} \condu(-M) & \text{if
} \  r <-M,
\\
\condu(r)   & \text{if } \  |r| \leq M,
\\
\condu(M) & \text{if } \ r >M,
\end{array}
\right.
\end{equation}
and, accordingly, the approximate elliptic operator
\begin{equation}
\label{M-operator}
{A}_M^k: H^1(\Omega)  \to H^1(\Omega)^*
  \text{  defined  by } \quad
\pairing{}{H^1(\Omega)} { {A}_M^k(\theta)}{v}:= \int_\Omega \condu_M(\theta) \nabla \theta  \nabla v \dd x -
\int_{\partial \Omega} \gtau{k}v \dd S.
\end{equation}
\par
We are now in the position to introduce the approximate discrete system \eqref{syst:approx-discr}. For notational simplicity, we will omit to indicate the dependence of the solution  quintuple 
on the index $M$.
\begin{problem}
\label{prob:approx-discrete}
Let $\gamma>4$. Starting from the discrete Cauchy data 
\eqref{discr-Cauchy},
for all $k=1,\ldots, K_\tau$, given 
 $(\utau{k-1},\etau{k-1},\ztau{k-1}, \ptau{k-1},\tetau{k-1}) \in \bsp,
 $
 find $\ztau k \in \spz$ 
 fulfilling the discrete damage flow rule \eqref{discr-dam-eq}. 
Given $(\utau{k-1},\etau{k-1},\ztau{k-1},\ptau{k-1},\tetau{k-1}) \in \bsp$ and $\ztau k \in \spz$, 
 find 
 $(\utau{k}, \etau{k},  \ptau{k}, \tetau{k}) \in \rbsp$ fulfilling 
 \begin{subequations}
\label{syst:approx-discr}
\begin{itemize}
\item[-] the kinematic admissibility $(\utau{k}, \etau{k}, \ptau{k}) \in \mathcal{A}(\wtau k)$;
\item[-] the approximate  discrete momentum balance 
\begin{equation}
\label{discrete-momentum-approx}
\rho\int_\Omega \Ddtau k u  v \dd x + \int_\Omega \simtau{k} : \sig v \dd x     = \pairing{}{H_\Dir^1(\Omega;\R^d)}{\Ltau k }{v} \quad
 \text{for all } v \in W_\Dir^{1,\gamma}(\Omega;\R^d),
\end{equation}
with the place-holder
$
\simtau{k}:= \bbD(\ztau{k})    \Dtau k e  + \bbC(\ztau{k}) \etau k + \tau |\etau{k}|^{\gamma-2} \etau k- \calT_M(\tetau k) \bbC((\ztau{k})^+)\bbE\,;
$
\item[-] the approximate discrete plastic flow rule
\begin{equation}
\label{discrete-plastic-approx}
\zetau k  +  \Dtau kp + \tau |\ptau {k}|^{\gamma-2}\ptau k \ni \simdevtau{k}, \quad \text{with }
 \zetau k \in \partial_{\dot p} \dip{\ztau{k}}{\tetau{k-1}}{\Dtau k p}   \qquad \aein \Omega;
\end{equation}
\item[-]  the approximate  discrete heat equation
\begin{equation}
\label{discrete-heat-approx} \begin{aligned}
 \Dtau{k}\teta + A_M^k (\tetau{k})  &  = \Gtau{k} + 
  \bbD(\ztau{k})      \Dtau{k} e :  \Dtau{k} e -\calT_M(\tetau{k}) {\bbC}((\ztau{k})^+)\bbE :  \Dtau{k} e 
  \\
  & \quad 
  + \did{\Dtau{k} z}  +   \left|  \Dtau{k} z \right|^2 + \bar{\nu} \ass  (\Dtau{k} z, \Dtau{k} z)  -  \tetau {k-1} \Dtau{k} z
  \\ & \quad 
  +
 \dip{\ztau{k}}{\tetau{k-1}}{ \Dtau{k} p} +  \left|  \Dtau{k} p \right|^2
   \quad \text{in } H^1(\Omega)^*.
\end{aligned} 
 \end{equation} 
\end{itemize}
\end{subequations}
\end{problem}

\par
\noindent 
\textbf{Step $2$:  Existence of solutions to  system \eqref{syst:approx-discr}:} 
we have the following result.
\begin{lemma}
\label{l:exist-approx-discr} 
Under  the growth condition \eqref{hyp-K}, there exists $\bar\tau>0$ such that  for $0<\tau< \bar \tau$ 
and  for every $k=1,\ldots, K_\tau$ there exists a solution $(\utau{k}, \etau{k},\ztau k, \ptau{k},\tetau k) \in
\bsp  $ to system \eqref{syst:approx-discr}. Furthermore, for any solution  $(\utau{k}, \etau{k},\ztau k, \ptau{k},\tetau k) $ the function
  $\tetau k$ complies with the positivity property \eqref{discr-strict-pos}. 
\end{lemma}
\begin{proof}
The positivity property \eqref{discr-strict-pos} (note that the constant providing the lower bound for $\tetau k$ is 
  independent of the truncation parameter $M$),
 follows from the analogue of estimate \eqref{var-ineq-4-pos}, with the  same  comparison argument. 
 Let us now address the existence of solutions. 
 \par
 First of all, we  find a solution $\ztau k$ to \eqref{discr-dam-eq}  via the minimum problem
 \begin{equation}
 \label{min-prob-4-z}
 \begin{aligned}
 &
 \mathrm{Min}_{z \in \spz} \left(\int_\Omega  \did{z{-}\ztau{k-1}} \dd x  + 
 \frac1{2\tau} \int_\Omega  |z{-}\ztau{k-1}|^2  \dd x + \frac{\nu}{2\tau} \ass(z{-}\ztau{k-1}, z{-}\ztau{k-1}) + \mathscr{F}_k(z)
 \right),
 \\
 &\quad \text{with } 
  \mathscr{F}_k(z): = \frac12 \ass (z,z) + \int_\Omega \left(\beta(z) {-} \lambda_W \ztau{k-1} z {+} \tfrac12 {\bbC}(z) \etau{k-1}{:}\etau{k-1} -\tetau{k-1}z \right) \dd x\,.
  \end{aligned}
 \end{equation}
 With the direct method in the calculus of variations, it is easy to check that \eqref{min-prob-4-z}, whose Euler-Lagrange equation is \eqref{discr-dam-eq},
  has a solution $\ztau{k}$.
 \par
Let us now briefly address the solvability of the approximate 
discrete thermoviscoplastic system \eqref{syst:approx-discr}, for fixed $k \in \{1,\ldots, K_\tau\}$ with
$\ztau{k}$ given. 
 To this end,
we reformulate system \eqref{syst:approx-discr} in the form 
\begin{equation}
\label{pseudo-monot}
\partial\Psi_k (\utau k -\wtau k ,  \ptau k,\tetau k) +\mathscr{A}_k (\utau k -\wtau k ,  \ptau k,\tetau k ) \ni \mathscr{B}_k \qquad \text{in } \rbspv^*,
\end{equation}
 with   the dissipation potential 
 $
 \Psi_k : \rbspv\to [0,+\infty) $ defined by $  \Psi_k (\tilde{u}, p,\teta)  = \Psi_k(p): =  \dip{\ztau{k}}{\tetau{k-1}}{p{-}\ptau{k-1}},
 $
  \MMM the space $ \rbspv:= W_\Dir^{1,\gamma} (\Omega;\R^d)  \times 
L^\gamma (\Omega;\mt_\dev^{d\times d}) \times H^1(\Omega) $,
 and the operator \MMM  $\mathscr{A}_k : \rbspv \to  \rbspv^* $ \EEE  given component-wise by 
\begin{subequations}
\label{oper-scrA}
\begin{align}
& 
\label{oper-scrA-1}
\begin{aligned}
\mathscr{A}_k^1 (\tilde u,p,\teta): =   \rho (\tilde u -\wtau k) - \mathrm{div}_{\Dir}\Big(  & \tau \bbD(\ztau{k}) \left( \sig{\tilde u + \wtau k} -p   \right) 
+\tau^2 {\bbC}(\ztau k)   \left( \sig{\tilde u + \wtau k} -p   \right) \\ & + \tau^3 \left|   \sig{\tilde u + \wtau k} -p    \right|^{\gamma-2}   \left( \sig{\tilde u + \wtau k} -p   \right) 
\\
& 
-\tau^2 \mathcal{T}_M(\teta) {\bbC}((\ztau{k})^+)\bbE \Big),
\end{aligned}
\\
& 
\label{oper-scrA-2}\begin{aligned}
\mathscr{A}_k^2(\tilde u, p,\teta): =  p + \tau^2 |p|^{\gamma-2} p
& 
- \Big(   \bbD (\ztau{k})\left( \sig{\tilde u + \wtau k} -p   \right) +\tau  {\bbC}(\ztau k)   \left( \sig{\tilde u + \wtau k} -p   \right) \\  & \quad  + 
\tau^2 \left|   \sig{\tilde u + \wtau k} -p    \right|^{\gamma-2}   \left( \sig{\tilde u + \wtau k} -p   \right) -\tau \mathcal{T}_M(\teta) {\bbC}((\ztau k)^+) \bbE \Big)_\dev,
\end{aligned}
\\
& 
\label{oper-scrA-3}
\begin{aligned}
\mathscr{A}_k^3 ( \tilde u,p,\teta): =
   \teta & + A_M^k(\teta)  
      -\frac1\tau \bbD(\ztau{k}) \left(\sig{\tilde u + \wtau k} -p \right){:} \left(\sig{\tilde u + \wtau k} -p \right)
      \\ & \quad
       -\frac2\tau \bbD(\ztau{k}) \left(\sig{\tilde u + \wtau k} -p \right){:}\etau{k-1} 
 + \mathcal{T}_M (\teta) {\bbC}((\ztau{k})^+)\bbE \left(\sig{\tilde u + \wtau k} -p - \etau{k-1} \right)
\\
&  \quad - \dip{\ztau{k}}{\tetau{k-1}}{p{-}\ptau{k-1}} -  \frac1\tau |p|^2   - \frac2\tau p : \ptau{k-1}, 
\end{aligned}
\end{align}
where $-\mathrm{div}_{\Dir}$  is defined by \eqref{div-Gdir}, 
 \end{subequations}
while the vector $\mathscr{B}_k \in \mathbf{B}^*$ on the right-hand side  of \eqref{pseudo-monot} has components
\begin{subequations}
\label{vector-B}
\begin{align}
\label{vect-B-1}
&
\mathscr{B}_k^1: = \Ltau k +2\rho\utau{k-1} - \rho\utau{k-1} - \mathrm{div}_{\Dir}(\tau \bbD(\ztau{k}) \etau{k-1}), 
\\
&
\label{vect-B-2}
\mathscr{B}_k^2: = \ptau{k-1} - (\bbD(\ztau{k}) \etau{k-1})_\dev,
\\
&
\label{vect-B-3}
\begin{aligned}
\mathscr{B}_k^3: = &  \Gtau k  + \frac1\tau \bbD(\ztau{k-1}) \etau{k-1}: \etau{k-1}
\\ & \quad 
 + 
\did{\ztau{k}{-} \ztau{k-1}} + 
\frac1\tau |\ztau{k}{-}\ztau{k-1}|^2 + \frac{\bar{\nu}}{\tau} \ass (\ztau{k}{-}\ztau{k-1},\ztau{k}{-}\ztau{k-1})   + \frac1\tau |\ptau{k-1}|^2 \,.
\end{aligned}
\end{align}  
 \end{subequations} 
The arguments therefore reduce to proving 
 the existence of a solution to the abstract subdifferential inclusion \eqref{pseudo-monot}.
 This follows from the very arguments developed in the proof of  \cite[Lemma 3.4]{Rossi2016},   to which we  refer for all details. 
 Let us only mention that the latter proof 
 is in turn based
 on  the existence result \cite[Cor.\ 5.17]{RoubNPDE2ed} 
 for elliptic systems featuring \MMM coercive \EEE pseudomonotone operators.
\MMM  In order to check coercivity of the operator  $\mathscr{A}_k : \rbspv \to  \rbspv^* $,   we  show
 that 
 \[
 \begin{aligned}
 &
 \exists\, c,\, C>0  \ \forall\, (\tilde u,p,\teta) \in \rbspv \, : 
 \\
 &
 \begin{aligned}
\pairing{}{\rbspv}{\mathscr{A}_k (\tilde u,p,\teta)}{(\tilde u,p,\teta)}  &   = \pairing{}{W_\Dir^{1,\gamma} (\Omega;\R^d) }{\mathscr{A}_k^1 (\tilde u,p,\teta)}{\tilde u} + \int_\Omega  \mathscr{A}_k^2 (\tilde u,p,\teta) : p \dd x + 
\pairing{}{H^1(\Omega)}{\mathscr{A}_k^3 (\tilde u,p,\teta)}{\teta} \\ &  \geq c\| (\tilde u,p,\teta)\|_{\rbspv}^2 - C\,.
\end{aligned}
\end{aligned}
 \]
  The calculations for  \cite[Lemma 3.4]{Rossi2016}    show the key role of the regularizing  terms
 $-\tau \mathrm{div}(|\etau k|^{\gamma- 2} \etau k)$  
 and   $\tau |\ptau k|^{\gamma-2} \ptau k$,
 added  to the discrete momentum equation and plastic flow rule,
  for proving the above  estimate. \EEE
\end{proof}
\par
\noindent 
\textbf{Step $3$:  A priori estimates on the solutions of  system \eqref{syst:approx-discr}:}  in order to 
pass to the limit as $M\to +\infty$ in Problem \ref{prob:approx-discrete}, for  \underline{fixed $k \in \{1,\ldots, K_\tau\}$ and $\ztau k$} solving the discrete damage flow rule
\eqref{discr-dam-eq},
  we need to establish suitable a priori estimates on a family
$(\utaum k, \etaum k,\ptaum k,\tetaum k)_{M}$
 of solutions to system \eqref{syst:approx-discr}. Along the footsteps of 
 \cite{Rossi2016},  we shall  derive them from a
 discrete version of the total energy inequality \eqref{total-uee}, cf.\ \eqref{discr-total-ineq} below,  featuring the discrete \MMM free \EEE energy
 (recall that the energy functional $\calE$ was defined in \eqref{stored-energy})
\begin{equation}
\label{discr-total-energy}
\calE_\tau (\teta,e,z,p) : = \calE(\teta,e,z)
 +\frac\tau\gamma \int_\Omega \left(|e|^\gamma + |p|^\gamma \right)  \dd x\,.
\end{equation}
\begin{lemma}
\label{l:aprio-M}
Assume \eqref{hyp-K}.
 Let $ k \in \{1,\ldots, K_\tau\} $ and $\tau \in (0,\bar \tau)$  be fixed.
Let $\ztau{k} \in \spz$ solve \eqref{discr-dam-eq}.
Then, the solution quadruple  $(\tetaum k, \utaum k, \etaum k,\ptaum k) $ to \eqref{syst:approx-discr} satisfies 
\begin{equation}
\label{discr-total-ineq}
\begin{aligned}
& \frac{\rho}2 \int_\Omega \left|\frac{\utaum k - \utau{k-1}}{\tau} \right|^2 \dd x + 
 \calE_\tau (\tetaum k,\etaum k, \ztau k, \ptaum k) \\  &  \leq  \frac{\rho}2 \int_\Omega \left| \frac{\utau{k-1} - \utau{k-2}}\tau \right |^2 \dd x +  \calE_\tau (\tetau {k-1},\etau {k-1},
 \ztau{k-1},
 \ptau{k-1})   
+ \tau \int_\Omega \Gtau k \dd x    + \tau \int_{\partial\Omega} \gtau{k} \dd x 
  \\ & \quad    + \tau \pairing{}{H_\Dir^{1}(\Omega;\R^d)}{\Ltau k}{\frac{\utaum k - \utau{k-1}}{\tau} {-} \Dtau kw}+ \tau \int_\Omega  \simtau {k} : \sig{\Dtau kw} \dd x   \\ & \quad  +\rho \int_\Omega \left(  \frac{\utaum k - \utau{k-1}}{\tau}  - \Dtau {k-1} u \right)   \Dtau kw \dd x\,.
\end{aligned}
\end{equation}
Moreover, there exists a constant $C>0$ such that  for all $M>0$
\begin{subequations}
\label{estimates-M-indep} 
\begin{align}
& 
\label{est-M-indep1}
\| \tetaum k\|_{L^1(\Omega)} + \| \utaum k \|_{L^2(\Omega;\R^d)} + \| \etaum k  \|_{L^2(\Omega;\mt_\sym^{d\times d})}  
\leq C,
\\
& \label{est-M-indep2}
\tau^{1/\gamma}   \| \utaum k \|_{W^{1,\gamma} (\Omega;\R^d)} 
+  \tau^{1/\gamma}   \| \etaum k \|_{L^\gamma(\Omega;\mt_\sym^{d\times d})}
+  \tau^{1/\gamma}   \| \ptaum k \|_{L^\gamma(\Omega;\mt_\dev^{d\times d})}  \leq C, 
\\ & 
\label{est-M-indep3}
\| \tetaum k \|_{H^1(\Omega)}  
 \leq C,
\\
& \label{est-M-indep4}
\| \zetau{k} \|_{L^\infty(\Omega;\mt_\dev^{d\times d})} \leq C,
\end{align}
\end{subequations}
where $\zetau k $ is a selection in $ \partial_{\dot p} \dip{\ztau k}{\tetau{k-1}}{\ptaum{k}{ -} \ptau{k-1}}$ fulfilling \eqref{discrete-plastic-approx}. 
\end{lemma}
\begin{proof} Inequality \eqref{discr-total-ineq} follows  by
testing 
\eqref{discr-dam-eq} by $\ztau{k} - \ztau{k-1}$, 
\eqref{discrete-momentum-approx} by  $(\utaum k {-}\wtau k) - (\utau {k-1} {-} \wtau {k-1})$, 
\eqref{discrete-plastic-approx} by $\ptaum k - \ptaum{k-1}$, and 
by multiplying  \eqref{discrete-heat-approx} by $\tau$ and integrating it in space. 
 We add  the resulting relations.
  We  develop  the following estimates  for the terms arising from the  test of the momentum balance \eqref{discrete-momentum-approx}
\begin{subequations}
\label{est-mimick}
\begin{align}
&
\label{est-mimick-1}
\begin{aligned}
&
\frac{\rho}{\tau^2} \int_\Omega \left(\utaum {k} {-} \utau{k-1} {-} (\utau{k-1}{ -} \utau{k-2} )   \right) (\utaum{k} {- }  \utau{k-1} )  \dd x \\ & \geq  \frac{\rho}2 \left\| \frac{\utaum k - \utau{k-1}}{\tau} \right\|_{L^2(\Omega)}^2 -  \frac{\rho}2 \left\| \frac{\utau {k-1} - \utau{k-2}}{\tau} \right\|_{L^2(\Omega)}^2,
\end{aligned}
\\
&    
\label{est-mimick-2}
\begin{aligned}
&
\int_\Omega \bbD(\ztau k) \left(\frac{\etaum k - \etau{k-1}}\tau
\right){:} \sig{\utaum k - \utau{k-1}} \dd x 
\\ &  = \tau \int_\Omega \bbD(\ztau k)   \frac{\etaum k - \etau{k-1}}\tau {:}\frac{\etaum k - \etau{k-1}}\tau  \dd x +\int_\Omega \bbD(\ztau k)
   \frac{\etaum k - \etau{k-1}}\tau{ :} (\ptaum k{ -} \ptau {k-1}) \dd x, 
\end{aligned}
\\
& 
\label{est-mimick-3}
\begin{aligned}
&
\int_\Omega {\bbC} (\ztau k) \etaum k {:}   \sig{\utaum k - \utau{k-1}} \dd x   
\\
&
 = \int_\Omega  {\bbC}(\ztau k)  \etaum {k} : (\etaum k - \etau{k-1}) +   {\bbC} (\ztau k)
\etaum {k} : (\ptaum k {-} \ptau{k-1}) \dd x 
\\  & \geq \int_\Omega \left(  \tfrac12 {\bbC}(\ztau k)  \etaum {k} {:}  \etaum {k}  - \tfrac12 {\bbC} (\ztau k) \etau{k-1}{:} \etau{k-1} +{\bbC}(\ztau k) \etaum {k} : (\ptaum k {-} \ptau{k-1}) \right) \dd x 
 \,, \end{aligned}
\\
& 
\label{est-mimick-4}
\begin{aligned}
&
\int_\Omega |\etaum k|^{\gamma-2} \etaum k   :   \sig{\utaum k - \utau{k-1}}  \dd x 
\\ & =  \int_\Omega |\etaum k|^{\gamma-2} \etaum k   :   (\etaum k {-} \etau{k-1})  \dd x + \int_\Omega   |\etaum k|^{\gamma-2} \etaum k   :   (\ptaum k {-} \ptau{k-1})  \dd x  \\ & 
 \geq\int_\Omega \left(  \tfrac1{\gamma} |\etaum k|^{\gamma} {-} \tfrac1{\gamma} |\etau {k-1}|^\gamma {+}  |\etaum k|^{\gamma-2} \etaum k   :   (\ptaum k{ -} \ptau{k-1})  \right) \dd x \,,
 \end{aligned}
 \end{align}
 \end{subequations}
 where have exploited the kinematic admissibility condition in  \eqref{est-mimick-2} and  \eqref{est-mimick-3}, 
 as well as elementary convexity inequalities to establish estimates \eqref{est-mimick-1}, \eqref{est-mimick-3}, and  \eqref{est-mimick-4}. 
 As for the terms arising from the test of the discrete damage flow rule \eqref{discrete-plastic-approx},  we observe that 
 \begin{subequations}
  \label{est-mimick-5}
  \begin{align}
   \label{est-mimick-5-1}
 &
 \pairing{}{\spz}{\omegatau{k}}{\ztau k-\ztau{k-1}} \geq \tau \Did{\Dtau {k}z},
 \\
 &
  \label{est-mimick-5-2}
  \begin{aligned}
  &
  \int_\Omega \Dtau {k}z (\ztau k-\ztau{k-1}) \dd x + \nu \pairing{}{\spz}{\As(\Dtau k z)}{\ztau{k}{-}\ztau{k-1}} \\ & = \frac1{\tau} |\ztau{k}{-}\ztau{k-1}|^2 + \frac{\nu}{\tau} 
 \ass (\ztau{k}{-}\ztau{k-1},\ztau{k}{-}\ztau{k-1}),
 \end{aligned}
 \\
 & 
  \label{est-mimick-5-3}
 \pairing{}{\spz}{\As(\ztau k)}{\ztau{k} {-} \ztau{k-1}} \geq \frac12 \ass (\ztau k,\ztau k) - \frac12 \ass (\ztau{k-1},\ztau{k-1}),
 \\
 &
  \label{est-mimick-5-4}
  \begin{aligned}
  &
 \int_\Omega \left(  \beta'(\ztau k)(\ztau{k} {-} \ztau{k-1}) {-} \lambda_W \ztau{k-1} (\ztau{k} {-} \ztau{k-1})  \right)   \dd x  
 \\
 & \geq \int_\Omega\left( \beta(\ztau k) {-} \beta(\ztau {k-1}) \right)  \dd x -\lambda_W \int_\Omega \left(\tfrac12  |\ztau {k}|^2 {-} \tfrac12  |\ztau {k-1}|^2  \right) \dd x 
  = \int_\Omega \left( W'(\ztau k) {-} W'(\ztau{k-1})  \right) \dd x\,,
 \end{aligned}
 \\
 & 
  \label{est-mimick-5-6}
 \int_\Omega \tfrac12 {\bbC}'(\ztau k)  (\ztau{k} {-} \ztau{k-1})  \etau{k-1}{:} \etau{k-1}  \dd x\geq \int_\Omega \tfrac12 {\bbC}(\ztau k)  \etau{k-1}{:} \etau{k-1}  \dd x - \int_\Omega \tfrac12 {\bbC}(\ztau{k-1})  \etau{k-1}{:} \etau{k-1}  \dd x\,,
 \end{align}
\end{subequations}
again by convexity arguments,
also relying on \eqref{elast-visc-tensors-3}.
 In particular, in \eqref{est-mimick-5-4} we 
 have  exploited  the convex-concave decomposition \eqref{decomp-W} of $W$.
 We now observe several cancellations. Indeed, 
 the two  terms on the  right-hand side of  \eqref{est-mimick-2} respectively
  cancel with the second term on the r.h.s.\ of  the heat equation
  \eqref{discrete-heat-approx}, multiplied by $\tau$, and with the analogous term deriving from   \eqref{discrete-plastic-approx}, tested by $\ptaum{k} - \ptau{k-1}$. 
 As for   \eqref{est-mimick-3},  the second term on its r.h.s.\ 
     cancels out with the first summand on the r.h.s.\ of  \eqref{est-mimick-5-6}; the third term on its r.h.s.\ cancels with the one deriving from  \eqref{discrete-plastic-approx}, and
      so does the third term on the r.h.s.\ of \eqref{est-mimick-4}. 
      Also the terms on the r.h.s.\ of \eqref{est-mimick-5-1} and \eqref{est-mimick-5-2} cancel out with the right-hand side of  
      \eqref{discrete-heat-approx} multiplied by $\tau$: in particular, observe that 
      \[
      \nu \pairing{}{\spz}{\As(\Dtau kz)}{\ztau k{-} \ztau {k-1}} = \bar{\nu}  \tau \int_\Omega  \ass(\Dtau kz,\Dtau kz) \dd x \,.
      \]
In fact, with the exception of $\tau \Gtau k$, all the terms on the r.h.s.\ of 
  \eqref{discrete-heat-approx}
cancel out.
Thus, straightforward calculations lead to  \eqref{discr-total-ineq}. 
\par
We refer to the proofs of  \cite[Lemma 3.5]{Rossi2016} and \cite[Lemma 4.4]{Rocca-Rossi} for all the detailed calculations leading to estimates \eqref{estimates-M-indep}: \MMM let us only mention that 
one has to first test   the discrete heat equation \eqref{discrete-heat} by $\calT_M(\tetaum k )$ and then by $\tetaum k $, and 
exploit the growth properties of $\condu$. \EEE
\end{proof}
\par
\noindent 
\textbf{Step $4$: Limit passage as $M\to+\infty$:} We again refer to \cite{Rossi2016} (cf.\ Lemma 3.6 therein) for the proof of the following result
on the limit passage
in the approximate discrete thermoviscoplastic subsystem \eqref{syst:approx-discr}. \MMM Let us only mention here that the strong convergences \eqref{conves-as-M-u}--\eqref{conves-as-M-p} arise from standard
 $\limsup$-arguments, developed by testing   the discrete momentum balance by $\utaum k - \wtau k$ 
 and the discrete plastic flow rule by $\ptaum k$.  \EEE 
\begin{lemma}
\label{l:3.6}
 Let $ k \in \{1,\ldots, K_\tau\} $  and $\tau\in (0,\bar\tau)$  be fixed. Under  the growth condition \eqref{hyp-K}, there exist a (not relabeled) subsequence
  of   $(\utaum k, \etaum k,\ptaum k,\tetaum k)_{M}$  and of $(\zetaum k)_M$,  a quadruple $(\utau k, \etau k, \ptau k,\tetau k) \in
  \rbsp$, with $\tetau k \in \spt$, and  a function
  $\zetau k \in   L^\infty(\Omega;\mt_\dev^{d\times d}) $,   such that the following convergences hold as $M\to\infty$
 \begin{subequations}
 \label{conves-as-M}
 \begin{align}
  \label{conves-as-M-u}
 &  \utaum k \to \utau  k  &&  \text{in } W_\Dir^{1,\gamma}(\Omega;\R^d),
 \\
  \label{conves-as-M-e}
 &  \etaum k \to \etau  k  &&  \text{in } L^{\gamma}(\Omega;\mt_\sym^{d\times d}),
 \\
  \label{conves-as-M-p}
 &  \ptaum k \to \ptau  k &&  \text{in } L^{\gamma}(\Omega;\mt_\dev^{d\times d}),
  \\
  \label{conves-as-M-zeta}
 &  \zetaum k \weaksto \zetau k  &&  \text{in } L^{\infty}(\Omega;\mt_\dev^{d\times d}),
 \\
  \label{conves-as-M-teta}
 &  \tetaum k \weakto \tetau k && \text{in }  \MMM H^1(\Omega), \EEE 
 \end{align}
\end{subequations}
and the quintuple $(\utau k, \etau k, \ptau k,\zetau k,\tetau k)$ fulfills system (\ref{discrete-momentum}, \ref{discrete-plastic}, \ref{discrete-heat}). 
\end{lemma}
With this result, \underline{we conclude the proof of Proposition \ref{prop:exist-discr}}. 
\QED

\section{\bf A priori estimates}
\label{s:aprio}
\noindent
Again, 
throughout this section we will tacitly assume all of the conditions listed in Section \ref{ss:2.1}.
We start by fixing some notation for the approximate solutions. 
\begin{notation}[Interpolants]
\upshape
For a given Banach space $X$ and a
$K_\tau$-tuple $( \mathfrak{h}_\tau^k )_{k=0}^{K_\tau}
\subset X$, we  introduce
 the left-continuous and  right-continuous piecewise constant, and the piecewise linear interpolants
\[
\left.
\begin{array}{llll}
& \pwc  {\mathfrak{h}}{\tau}: (0,T] \to X  & \text{defined by}  &
\pwc {\mathfrak{h}}{\tau}(t): = \mathfrak{h}_\tau^k,
\\
& \upwc  {\mathfrak{h}}{\tau}: (0,T] \to X  & \text{defined by}  &
\upwc {\mathfrak{h}}{\tau}(t) := \mathfrak{h}_\tau^{k-1},
\\
 &
\pwl  {\mathfrak{h}}{\tau}: (0,T] \to X  & \text{defined by} &
 \pwl {\mathfrak{h}}{\tau}(t):
=\frac{t-t_\tau^{k-1}}{\tau} \mathfrak{h}_\tau^k +
\frac{t_\tau^k-t}{\tau}\mathfrak{h}_\tau^{k-1}
\end{array}
\right\}
 \qquad \text{for $t \in
(t_\tau^{k-1}, t_\tau^k]$,}
\]
setting $\pwc  {\mathfrak{h}}{\tau}(0)= \upwc  {\mathfrak{h}}{\tau}(0)= \pwl  {\mathfrak{h}}{\tau}(0): = \mathfrak{h}_\tau^0$. 
We also introduce the piecewise linear interpolant    of the values
$\{ \Dtau  k {\mathfrak{h}}  = \tfrac{\mathfrak{h}_\tau^k - \mathfrak{h}_{\tau}^{k-1})}\tau\}_{k=1}^{K_\tau}$ (i.e.\  the
values taken by  the piecewise constant function $\pwl
{\dot{\mathfrak{h}}}{\tau}$), viz.
\[
\pwwll  {\mathfrak{h}}{\tau}: (0,T) \to X  \ \text{ defined by } \ \pwwll
{\mathfrak{h}}{\tau}(t) :=\frac{(t-t_\tau^{k-1})}{\tau}
\Dtau k  {\mathfrak{h}}  +
\frac{(t_\tau^k-t)}{\tau} \Dtau {k-1} {\mathfrak{h}}   \qquad \text{for $t \in
(t_\tau^{k-1}, t_\tau^k]$.}
\]
Note that $ {\partial_t \pwwll  {{\mathfrak{h}}}{\tau}}(t) =  \Ddtau k {\mathfrak{h}}  $ for $t \in
(t_\tau^{k-1}, t_\tau^k]$.

Furthermore, we   denote by  $\pwc{\mathsf{t}}{\tau}$ and by
$\upwc{\mathsf{t}}{\tau}$ the left-continuous and right-continuous
piecewise constant interpolants associated with the partition, i.e.
 $\pwc{\mathsf{t}}{\tau}(t) := t_\tau^k$ if $t_\tau^{k-1}<t \leq t_\tau^k $
and $\upwc{\mathsf{t}}{\tau}(t):= t_\tau^{k-1}$ if $t_\tau^{k-1}
\leq t < t_\tau^k $. Clearly, for every $t \in [0,T]$ we have
$\pwc{\mathsf{t}}{\tau}(t) \downarrow t$ and
$\upwc{\mathsf{t}}{\tau}(t) \uparrow t$ as $\tau\to 0$.
\end{notation}

It follows from conditions  \eqref{heat-source}, \eqref{dato-h}, and \eqref{data-displ}  that the piecewise constant
interpolants $(\pwc  G{\tau}
)_{\tau}$, $(\pwc  g{\tau} )_{\tau}$, and $(\pwc \calL{\tau})_\tau$   of the values 
$\Gtau{k}$,  $\gtau{k}$, and $\Ltau k $, cf.\  \eqref{local-means}, fulfill as $\tau \down
0$
\begin{subequations}
\label{convs-interp-data}
\begin{align}
\label{converg-interp-g}  & \pwc G{\tau}  \to G
  \text{ in $L^1(0,T;L^1(\Omega))\cap L^2(0,T;H^1(\Omega)^*)$,}
  \\
  \label{converg-interp-h}  & \pwc g{\tau}  \to g
  \text{ in $L^1(0,T;L^2(\partial\Omega))$,}
  \\
  \label{converg-interp-L}  & \pwc \calL{\tau} \to \calL \text{ in $ L^2(0,T; H_\Dir^1(\Omega;\R^d)^*).$} 
\end{align}
Furthermore, it follows from \eqref{Dirichlet-loading} and \eqref{discr-w-tau} that 
\begin{equation}
\label{converg-interp-w}
\begin{gathered}
\pwc w\tau \to w \quad \text{in  $L^1(0,T; W^{1,\infty} (\Omega;\R^d))$}, \qquad \pwl w\tau \to w \quad \text{in $W^{1,p} (0,T; H^1(\Omega;\R^d))$ for all } 1 \leq p<\infty, \\ \pwwll w\tau \to w \quad \text{in  $W^{1,1}(0,T; H^1(\Omega;\R^d) \cap H^1(0,T; L^2(\Omega;\R^d))$},
\\
\sup_{\tau>0}   \tau^{\alpha_w} \| \sig{\pwl{\dot w}{\tau}}\|_{L^\gamma(\Omega;\mt_\sym^{d\times d})} \leq C <\infty \text{ with } \alpha_w \in (0,\tfrac1\gamma)\,.
 \end{gathered}
\end{equation}
\end{subequations}
\par We  now reformulate  system \eqref{syst:discr} 
 in terms of the 
 interpolants of
  the   discrete solutions $(\utau k, \etau k,\ztau k, \ptau k,\tetau k)_{k=1}^{K_\tau}$. Therefore, we have
 for almost all $t\in (0,T)$
\begin{subequations}
\label{syst-interp}
\begin{align}
&
 \label{eq-u-interp}
 \begin{aligned}
\rho\int_\Omega  \partial_t\pwwll {\uu}{\tau}(t)  v \dd x + \int_\Omega \pwc\sigma \tau(t)  {:} \sig{v} \dd x  = \pairing{}{H_\Dir^{1}(\Omega;\R^d)}{\pwc\calL \tau(t)}{v}
\quad  \text{for all } v \in W_\Dir^{1,\gamma}(\Omega;\R^d)
\end{aligned}
\intertext{with the notation $\pwc\sigma\tau(t): = \bbD(\pwc z\tau(t)) \pwl {\dot e}\tau (t) + {\bbC}(\pwc z\tau (t))  \pwc e\tau (t)+ \tau |\pwc e\tau (t)|^{\gamma-2} \pwc e\tau (t)-
 \pwc\teta\tau (t) \overline{\bbC}((\pwc z\tau (t))^+) \bbE$,}
& 
\label{eq-z-interp}
\begin{aligned}
 \pwc \omega\tau(t) + \pwl {\dot z}{\tau}(t)   + \nu \As (\pwl{\dot z}{\tau}(t))   + \As(\pwc z\tau(t)) + \beta'(\pwc z\tau(t)) - \lambda_W \upwc z\tau(t)
 = -\frac12 \bbC'(\pwc z\tau) \upwc e\tau {:} \upwc e\tau + \upwc \teta\tau  \quad \text{ in } \spz^* 
\end{aligned}
\intertext{with $\pwc \omega\tau(t)  \in \partial\Did{\pwl {\dot{z}}\tau(t)}$ in $\spz^*$,}
& \label{eq-p-interp}
\begin{aligned}
\pwc\zeta\tau(t)+ \pwl{\dot p}\tau (t) + \tau | \pwc p\tau(t)|^{\gamma-2} \pwc p\tau(t) \ni (\pwc\sigma \tau(t)  )_\dev \quad \aein\, \Omega
  \end{aligned}
\intertext{ 
with  $\pwc\zeta\tau(t) \in \partial_{\dot p} \dip{\pwc z\tau(t)}{\upwc \teta\tau(t)}{ \pwl{\dot p}\tau(t) }$ a.e.\ in $\Omega$,}
&
\label{eq-teta-interp}
\begin{aligned} 
\partial_t \pwl \teta{\tau}(t) 
+   \mathcal{A}^{\frac{\bar{\mathsf{t}}_\tau(t)}{\tau}}( \pwc{\teta}\tau (t) )    = \pwc G{\tau}(t) &  + \bbD(\pwc z\tau(t)) \pwl{\dot e}\tau(t) {:}   \pwl{\dot e}\tau(t) 
- \pwc \teta\tau(t) \bbC((\pwc z\tau(t))^+) \bbE  : \pwl{\dot e}\tau(t)
\\
& + \did{\pwl {\dot z}\tau(t)} + |\pwl {\dot z}\tau(t)|^2 +\bar\nu \ass (\pwl {\dot z}\tau(t), \pwl {\dot z}\tau(t)) 
-\upwc \teta{\tau}(t) \pwl {\dot z}\tau(t)
\\
&
+
\dip{\pwc z\tau(t)}{\upwc \teta\tau(t)}{ \pwl{\dot p}\tau (t)} + |\pwl{\dot p}\tau (t)|^2  
 \quad \text{in $H^1(\Omega)^*$,}
\end{aligned}
\end{align}
\end{subequations}
\MMM cf.\ \eqref{elliptic-k} for the definition of the operator $A^{\frac{\bar{\mathsf{t}}_\tau(t)}{\tau}}$. \EEE
\par
Our next  result collects the discrete versions of the  entropy, total energy, and mechanical energy inequalities
satisfied by the approximate solutions. 
In order to state
 the discrete entropy inequality  \eqref{entropy-ineq-discr} below, 
 we need to 
 construct suitable approximations of the 
 test functions for the limiting  entropy inequality \eqref{entropy-ineq}. 
 Following \cite{Rossi2016}, we will in fact approximate \emph{positive}
 test functions $\varphi $ 
  with  $\varphi  \in \rmC^0 ([0,T]; W^{1,\infty}(\Omega)) \cap H^1 (0,T; L^{6/5}(\Omega)) $, which is 
   a slightly stronger  temporal regularity than that required by Def.\ \ref{def:entropic-sols}. 
   We set
     \begin{equation}
 \label{discrete-tests-phi}
   \varphi_\tau^k:= 
\frac{1}{\tau}\int_{t_\tau^{k-1}}^{t_\tau^k}  \varphi(s)\dd s \qquad \text{for } k=1, \ldots, K_\tau,
 \end{equation}
and consider the piecewise constant and  linear interpolants
$\pwc \varphi\tau$ and $\pwl \varphi\tau$ of the values
$(\varphi_\tau^k)_{k=1}^{K_\tau}$.
 We can show that 
\begin{equation}
\label{convergences-test-interpolants}
\pwc \varphi\tau  \to \varphi \quad
\text{ in } L^\infty (0,T; W^{1,\infty}(\Omega)) \text{ and }  \quad \partial_t \pwl \varphi\tau  \to \partial_t \varphi \quad \text{ in } L^2 (0,T; L^{6/5}(\Omega)).
\end{equation} 
\par
We are now in the position to give the discrete versions of the entropy and energy inequalities in which we will pass to the limit
to conclude the existence of \emph{weak energy} solutions to the regularized thermoviscoplastic system. 
The discrete total energy inequality \eqref{total-enid-discr} follows by adding up \eqref{discr-total-energy}. 
The proof of the other two inequalities can be obtained  
by trivially adapting the arguments for \cite[Prop.\ 4.8]{Rocca-Rossi}  and \cite[Lemma 4.2]{Rossi2016}. 
Their proof relies on  the following  \emph{discrete by-part integration} formula,
which we recall for later use, 
 holding
for all $K_\tau$-uples  $\{\mathfrak{h}_\tau^k \}_{k=0}^{K_\tau} \subset B,\, \{ v_\tau^k \}_{k=0}^{K_\tau} \subset B^*$ in a given Banach space $B$:
\begin{equation}
\label{discr-by-part} \sum_{k=1}^{K_\tau} \tau
\pairing{}{B}{\vtau{k}}{\Dtau{k}{\mathfrak{h}}} =
\pairing{}{B}{\vtau{K_\tau}}{\btau{K_\tau}}
-\pairing{}{B}{\vtau{0}}{\btau{0}} -\sum_{k=1}^{K_\tau}\tau\pairing{}{B}{\dtau{k}{v} }{ \btau{k-1}}\,.
\end{equation}
  \begin{lemma}[Discrete entropy, mechanical, and  total energy inequalities]
\label{lemma:discr-enid}
The interpolants of the   discrete solutions
 $(\utau{k}, \etau k, \ztau k, \ptau{k}, \tetau{k})_{k=1}^{K_\tau}$ to Problem \ref{prob:discrete}
 fulfill
 \begin{itemize}
 \item[-]
  the \emph{discrete} entropy inequality 
 \begin{equation}\label{entropy-ineq-discr}
\begin{aligned}
& \int_{\pwc{\mathsf{t}}{\tau}(s)}^{\pwc{\mathsf{t}}{\tau}(t)}
\int_\Omega \log(\upwc\teta\tau (r)) 
\pwl {\dot\varphi}\tau (r) \dd x \dd r - \int_{\pwc{\mathsf{t}}{\tau}(s)}^{\pwc{\mathsf{t}}{\tau}(t)}  \int_\Omega  \condu(\pwc\teta\tau (r)) \nabla \log(\pwc\teta\tau(r))  \nabla \pwc\varphi\tau(r)  \dd x \dd r \\ & 
\leq \int_\Omega \log(\pwc\teta\tau(t))
\pwc\varphi\tau(t) \dd x - \int_\Omega
\log(\pwc\teta\tau(s)) 
\pwc\varphi\tau(s) \dd x
 -  \int_{\pwc{\mathsf{t}}{\tau}(s)}^{\pwc{\mathsf{t}}{\tau}(t)}  \int_\Omega \condu(\pwc \teta\tau(r)) \frac{\pwc\varphi\tau(r)}{\pwc \teta\tau(r)}
\nabla \log(\pwc \teta\tau(r)) \nabla \pwc \teta\tau (r) \dd x \dd r\\ & 
\begin{aligned}
\quad 
 -
  \int_{\pwc{\mathsf{t}}{\tau}(s)}^{\pwc{\mathsf{t}}{\tau}(t)}   \int_\Omega \Big( \pwc G{\tau} (r)    & +
    \bbD(\pwc z\tau(r)) \pwl{\dot e}\tau (r) {:} \pwl{\dot e}\tau(r) - \pwc\teta\tau(r) \bbC((\pwc z\tau (r))^+) \bbE  {:}    \pwl{\dot e}\tau(r) 
   + \dip{\pwc z\tau(r)}{\upwc \teta\tau(r)}{ \pwl{\dot p}\tau(r)}  
   \\
   & + |\pwl{\dot p}\tau(r)|^2   + \did {\pwl{\dot z}\tau(r)} + |\pwl{\dot z}\tau(r)|^2 +\bar{\nu} \ass(\pwl{\dot z}\tau(r),\pwl{\dot z}\tau(r)) - \upwc \teta\tau(r) \pwl{\dot z}\tau(r)  
  \Big)
 \frac{\pwc\varphi\tau(r)}{\pwc\teta\tau(r)} \dd x \dd r 
 \end{aligned}
   \\ & \quad -  \int_{\pwc{\mathsf{t}}{\tau}(s)}^{\pwc{\mathsf{t}}{\tau}(t)}   \int_{\partial\Omega} \pwc g{\tau} (r)   \frac{\pwc\varphi\tau(r)}{\pwc\teta\tau(r)} \dd S \dd r
\end{aligned}
\end{equation}
for all $0 \leq s \leq t \leq T$ and
for all $\varphi \in \mathrm{C}^0 ([0,T]; W^{1,\infty}(\Omega)) \cap H^1 (0,T; L^{6/5}(\Omega)) $ with
$\varphi \geq 0$;
\item[-] the \emph{discrete} total energy inequality for all $ 0 \leq s \leq t
\leq  T$, viz.
\begin{equation}
\label{total-enid-discr}
\begin{aligned} & 
\frac{\rho}2 \int_\Omega 
|\pwwll {u}{\tau}(\pwc{\mathsf{t}}{\tau}(t))|^2 
\dd x +  \calE_\tau(\pwc\teta\tau(t),\pwc e\tau(t), \pwc z\tau(t), \pwc p\tau(t))  \\ & \leq \frac{\rho}2 \int_\Omega |\pwwll {u}{\tau}(\pwc{\mathsf{t}}{\tau}(s))|^2  \dd x 
 +  \calE_\tau(\pwc\teta\tau(s),\pwc e \tau(s), \pwc z\tau(s), \pwc p\tau(s)) +   \int_{\pwc{\mathsf{t}}{\tau}(s)}^{\pwc{\mathsf{t}}{\tau}(t)} \pairing{}{H_\Dir^1(\Omega;\R^d)}{\pwc \calL \tau (r) }{\pwl{\dot u}\tau (r) {-} \pwl{ \dot w}\tau (r)} \dd r
 \\ & \quad
 + \int_{\pwc{\mathsf{t}}{\tau}(s)}^{\pwc{\mathsf{t}}{\tau}(t)} \left( 
\int_\Omega  \pwc G\tau  \dd x + 
 \int_{\partial\Omega} \pwc g\tau  \dd S \right)  \dd r \\ & \quad 
 +\rho \left(  \int_\Omega \pwl{\dot u}\tau(t) \pwl{\dot w }\tau(t) \dd x -   \int_\Omega \pwl{\dot u}\tau(s)   \pwl{\dot w }\tau(s) \dd x 
  -   \int_{\pwc{\mathsf{t}}{\tau}(s)}^{\pwc{\mathsf{t}}{\tau}(t)} \int_\Omega \pwl{\dot u}\tau(r{-}\tau) \partial_t \pwwll{ w}{\tau} (r)  \dd x \dd r \right) 
  \\ & \quad   
   +   \int_{\pwc{\mathsf{t}}{\tau}(s)}^{\pwc{\mathsf{t}}{\tau}(t)}\int_\Omega \pwc \sigma\tau (r) {:} \sig{\pwl{\dot w}\tau (r) }  \dd x \dd r 
 \end{aligned}
\end{equation}
 with the discrete total energy functional $\calE_\tau$ from 
 \eqref{discr-total-energy}; 
 \item[-] the \emph{discrete} mechanical energy inequality  for all $ 0 \leq s \leq t
\leq  T$,  featuring the energy functionals $\calQ$ and $\calG$ from \eqref{stored-energy}
 \begin{equation}
\label{mech-ineq-discr}
\begin{aligned} & 
\frac{\rho}2 \int_\Omega | \pwwll {u}{\tau}(\pwc{\mathsf{t}}{\tau}(t))|^2 \dd x 
\\
&\quad +  \int_{\pwc{\mathsf{t}}{\tau}(s)}^{\pwc{\mathsf{t}}{\tau}(t)} \int_\Omega
 \left( \bbD(\pwc z\tau(r)) \pwl{\dot e}\tau(r) {:} \pwl{\dot e}\tau(r) 
+\did{\pwl{\dot z}\tau(r)} + |\pwl{\dot z}\tau(r)|^2  + \dip{\pwc z\tau(r)}{\upwc\teta\tau(r)}{
 \pwl{\dot p}\tau(r)} + |\pwl{\dot p}\tau(r)|^2  \right)  \dd x \dd r
 \\
 & \quad 
+ \int_{\pwc{\mathsf{t}}{\tau}(s)}^{\pwc{\mathsf{t}}{\tau}(t)}  \nu \ass (\pwl{\dot z}\tau(r), \pwl{\dot z}\tau(r)) \dd r  +\calQ(\pwc e\tau(t), \pwc z\tau(t)) + \calG(\pwc z\tau(t)) 
 + \frac{\tau}\gamma\int_\Omega \left(  |\pwc e\tau(t)|^\gamma +  |\pwc p\tau(t)|^\gamma \right) \dd x 
\\ &  \leq 
 \frac{\rho}2 \int_\Omega | \pwwll {u}{\tau}(\pwc{\mathsf{t}}{\tau}(s))|^2 \dd x + \calQ(\pwc e\tau(s), \pwc z\tau(s)) + \calG(\pwc z\tau(s)) 
 + \frac{\tau}\gamma\int_\Omega \left(  |\pwc e\tau(s)|^\gamma +  |\pwc p\tau(s)|^\gamma \right) \dd x 
\\ & \quad +
    \int_{\pwc{\mathsf{t}}{\tau}(s)}^{\pwc{\mathsf{t}}{\tau}(t)} \pairing{}{H_\Dir^1(\Omega;\R^d)}{\pwc \calL \tau (r) }{\pwl{\dot u}\tau (r){-} \pwl{\dot w}\tau (r)  } \dd r
    + \int_{\pwc{\mathsf{t}}{\tau}(s)}^{\pwc{\mathsf{t}}{\tau}(t)} \int_\Omega \pwc\teta\tau(r) \bbC (\pwc z\tau(r)) \bbE {:} \pwl{\dot e}\tau \dd x  \dd r
     \\ &  \quad
       + \int_{\pwc{\mathsf{t}}{\tau}(s)}^{\pwc{\mathsf{t}}{\tau}(t)} \int_\Omega \upwc\teta\tau(r)   \pwl {\dot z}\tau(r) \dd x \dd r 
      +\rho \int_\Omega \pwl{\dot u}\tau(t) \pwl{\dot w }\tau(t) \dd x -  \rho \int_\Omega \pwl{\dot u}\tau(s)   \pwl{\dot w }\tau(s) \dd x 
     \\ & \quad
   - \rho  \int_{\pwc{\mathsf{t}}{\tau}(s)}^{\pwc{\mathsf{t}}{\tau}(t)} \int_\Omega  \pwl{\dot u}\tau(r{-}\tau)  \partial_t \pwwll{w}\tau (r)  \dd x \dd r   +   \int_{\pwc{\mathsf{t}}{\tau}(s)}^{\pwc{\mathsf{t}}{\tau}(t)}\int_\Omega \pwc \sigma\tau (r) {:} \sig{\pwl{\dot w}\tau (r) }  \dd x \dd r\,.  
 \end{aligned}
\end{equation} 
 \end{itemize}
 \end{lemma}
 \noindent 
 In fact, observe that $\pwwll {u}{\tau}(\pwc{\mathsf{t}}{\tau}(t)) = \pwl{\dot u}{\tau}(t)$  at  almost  all  $t\in (0,T)$ (i.e., at $t \in [0,T] \setminus \{ t_\tau^1, \ldots, t_\tau^{K_\tau}\}$, where  $ \pwl{\dot u}{\tau}(t)$  is defined. Using the interpolant $\pwwll {u}{\tau}$ in the discrete energy inequalities \eqref{total-enid-discr} and \eqref{mech-ineq-discr} allows us to write them at \emph{every} couple of time instants $ 0 \leq s \leq t \leq T$. 
\par
The \underline{main result of this section} collects  all the a priori estimates on the approximate solutions. 
\begin{proposition}
\label{prop:aprio} 
Assume \eqref{hyp-K}. Then, there exists a constant $\zeta_*>0$ such that 
\begin{equation}
\label{discr-feasibility}
\pwc z\tau(x,t), \, \pwl z \tau(x,t) \in [\zeta_*, 1] \quad \text{for every } (x,t) \in \Omega \times [0,T] \text{ and all } \tau>0
\end{equation}
and 
 there exists a constant $S>0$ such that for all $\tau>0$ the following estimates hold
\begin{subequations}
\label{aprio}
\begin{align}
& \label{aprioU1}
\|\pwc \uu{\tau}\|_{L^\infty(0,T;H^1(\Omega;\R^d))}
 \leq S,
\\
& \label{aprioU2}
\|\pwl \uu{\tau}\|_{H^1(0,T; H^1(\Omega;\R^d) )} + \|\pwl \uu{\tau}\|_{
W^{1,\infty}(0,T; L^2(\Omega;\R^d))}
 \leq S,
\\
& \label{aprioU3} \|\pwwll \uu{\tau}
\|_{L^2(0,T; H^1(\Omega;\R^d) )} +   \|\pwwll \uu{\tau}
\|_{L^\infty (0,T;  L^2(\Omega;\R^d))} + \|\pwwll \uu{\tau}
\|_{W^{1,\gamma/(\gamma-1)}(0,T;W^{1,\gamma}(\Omega;\R^d)^*)} \leq S,
\\
& \label{aprioE1}
\|\pwc
e{\tau}\|_{L^\infty(0,T;L^2(\Omega;\mt_\sym^{d\times d}))} \leq S,
\\
& \label{aprioE2}
\|\pwl
e{\tau}\|_{H^1(0,T;L^2(\Omega;\mt_\sym^{d\times d}))} \leq S,
\\
& \label{aprioE3}
\tau^{1/\gamma}\|\pwc
e{\tau}\|_{L^\infty(0,T;L^\gamma(\Omega;\mt_\sym^{d\times d}))} \leq S,
\\
& 
\label{aprioZ1}
\| \pwc z\tau\|_{L^\infty (0,T;\spz)} \leq S,
\\
&
\label{aprioZ2}
\| \pwl z\tau\|_{H^1(0,T;L^2(\Omega)} \leq S,
\\
& \label{aprioZ3}
\| \pwl z\tau\|_{H^1(0,T;\spz)} \leq S \qquad \text{ if } \nu>0,
\\
& \label{aprioZ4}
\| \pwc \omega \tau\|_{L^2(0,T;\spz^*)} \leq S  \qquad \text{ if } \nu>0,
\\
& \label{aprioP1}
\|\pwc
p{\tau}\|_{L^\infty(0,T;L^2(\Omega;\mt_\dev^{d\times d}))} \leq S,
\\
& \label{aprioP2}
\|\pwl
p{\tau}\|_{H^1(0,T;L^2(\Omega;\mt_\dev^{d\times d}))} \leq S,
\\
& \label{aprioP3}
\tau^{1/\gamma}\|\pwc
p{\tau}\|_{L^\infty(0,T;L^\gamma(\Omega;\mt_\dev^{d\times d}))} \leq S,
\\
& \label{aprioP4}
\| \pwc \zeta \tau\|_{L^\infty(Q; \mt_\dev^{d\times d})} \leq S,
\\
& \label{aprio6-discr}
 \|\pwc
\teta{\tau}\|_{L^{\infty}(0,T;L^{1}(\Omega))}+
 \|\pwc
\teta{\tau}\|_{L^2(0,T; H^1(\Omega))}    \leq S,
\\
&
\label{log-added}
\| \log(\pwc \teta{\tau}) \|_{L^\infty (0,T; L^p(\Omega))} +
\| \log(\pwc \teta{\tau}) \|_{L^2 (0,T; H^1(\Omega))}
 \leq S \quad \text{for all } 1 \leq p<\infty,
\\
&  \label{est-temp-added?-bis}
  \| (\pwc \teta\tau)^{(\mu+\alpha)/2} \|_{L^2(0,T; H^1(\Omega))} + \| (\pwc \teta\tau)^{(\mu-\alpha)/2} \|_{L^2(0,T; H^1(\Omega))} \leq C \quad \text{for all } \alpha \in  \MMM [0{\vee}(2{-}\mu), 1), \EEE
\\ 
& \label{aprio_Varlog}
\sup_{\varphi \in W^{1,d+\epsilon}(\Omega), \ \| \varphi \|_{W^{1,d+\epsilon}(\Omega)}\leq 1}
\mathrm{Var}(\pairing{}{W^{1,d+\epsilon}(\Omega)}{\log(\pwc\teta\tau)}{\varphi}; [0,T]) \leq S \quad 
\text{ for every } \epsilon>0
\end{align}
\MMM where we recall that $\gamma>4$ and  refer to \EEE   \eqref{var-notation} ahead for the definition of $\mathrm{Var}(\pairing{}{W^{1,d+\epsilon}(\Omega)}{\log(\pwc\teta\tau)}{\varphi}; [0,T]) $). 
Furthermore, if  $\condu$ fulfills \eqref{hyp-K-stronger}, there holds in addition
\begin{align}
\label{aprio7-discr}
&
\sup_{\tau>0} \| \pwl \teta{\tau} \|_{\mathrm{BV} ([0,T]; W^{1,\infty} (\Omega)^*)} \leq S.
\end{align}
\end{subequations}
\end{proposition}
We will not develop all the calculations leading to estimates 
\eqref{aprio}, but rather only give a sketch of the proof, referring to the proofs of \cite[Prop.\ 4.10]{Rocca-Rossi} and \cite[Prop.\ 4.3]{Rossi2016} for all 
details. 
Nevertheless, let us  mention in advance the main ingredients of the various calculations:
\begin{enumerate}
\item  The starting point in the derivation of the a priori estimates is 
 the discrete total energy inequality \eqref{total-enid-discr}. 
Indeed, conditions \eqref{data-displ}, \eqref{safe-load}, and \eqref{Dirichlet-loading} allow us to suitably estimate the terms on the right-hand side of \eqref{total-enid-discr},
based on the calculations from the proof of  \cite[Prop.\ 4.3]{Rossi2016}.
We thus deduce from \eqref{total-enid-discr} the uniform energy bound
\begin{equation}
\label{unif-energy-bound}
\sup_{t\in [0,T]} \calE_\tau (\pwc \teta\tau(t), \pwc e\tau(t), \pwc z\tau(t), \pwc p \tau(t)) \leq S.
\end{equation}
In view of the  coercivity
properties of the discrete energy $\calE_\tau$, \eqref{unif-energy-bound}
yields estimates  \eqref{aprioE1}, \eqref{aprioE3}, \eqref{aprioP3}, and the first of \eqref{aprio6-discr}.
\MMM We also establish   \eqref{aprioZ1}.
 We further \EEE 
infer a  bound for the kinetic energy term in \eqref{total-enid-discr}, which leads to the 
first of \eqref{aprioU3}. 
\item Estimate \eqref{unif-energy-bound}  in particular implies that $\sup_{t\in [0,T]} \int_\Omega \beta(\pwc z\tau(t)) \dd x \leq S$. Exploiting the coercivity condition
$z^{2d} \beta(z) \to +\infty$ as $z \down 0$, which in turn originates from \eqref{hyp-phi-1}, and repeating an argument from \cite[Lemma 3.3]{Crismale-Lazzaroni}, 
we thus conclude  the feasibility condition \eqref{discr-feasibility}. 
\item The crucial estimate for $\teta$ in $L^2(0,T;H^1(\Omega))$  ensues from testing the discrete heat equation \eqref{eq-teta-interp} by a suitable negative power of $\pwc \teta\tau$, 
as suggested in \cite{FPR09}, cf.\ also the proof of \cite[Prop.\ 4.10]{Rocca-Rossi}.
\item The \emph{dissipative estimates} \eqref{aprioU2}, \eqref{aprioE2}, \eqref{aprioZ2} \& \eqref{aprioZ3}, as well as \eqref{aprioP2}, derive from the discrete mechanical energy 
inequality \eqref{mech-ineq-discr}. \MMM We then establish \eqref{aprioU1}. \EEE
\item The total variation-type estimate \eqref{aprio_Varlog} 
 is deduced from the discrete entropy inequality \eqref{entropy-ineq-discr}  with the very same calculations as in the proofs of 
 \cite[Prop.\ 4.10]{Rocca-Rossi} and \cite[Prop.\ 4.3]{Rossi2016}.
 \item The enhanced $\BV$-estimate \eqref{aprio7-discr} derives from a comparison argument in the discrete heat equation  \eqref{eq-teta-interp}, again following the proofs of 
  \cite[Prop.\ 4.10]{Rocca-Rossi} and \cite[Prop.\ 4.3]{Rossi2016}.
  \end{enumerate}
  \noindent
  We will now give a \underline{sketch of the proof of Proposition \ref{prop:aprio}:}
  \par
  \noindent
  \textbf{Step $1$: First a priori estimate.} We write the discrete total energy inequality \eqref{total-enid-discr} for $s=0$ and $t\in (0,T]$. 
  We estimate 
    the terms on its right-hand side resorting to conditions \eqref{Cauchy-data}
and  \eqref{approx-e_0}
on the initial data $(u_\tau^0,e_\tau^0, z_0,p_\tau^0,\teta_0), $ which ensure a bound for the kinetic energy term $\int_\Omega |\pwl{\dot u}\tau(0)|^2 \dd x $, as well as the estimate
$\sup_{\tau>0} \calE_\tau(\teta_0,e_\tau^0, z_0,p_\tau^0) \leq C$. 
We use  the safe load condition \eqref{safe-load}, condition \eqref{Dirichlet-loading} on the Dirichlet loading,  as well as estimates \eqref{converg-interp-g}, \eqref{converg-interp-h}, and \eqref{converg-interp-w} to handle all the other terms on the r.h.s.\ of \eqref{total-enid-discr},
arguing in the very same way as throughout the proof of \cite[Prop.\ 4.3]{Rossi2016}.
In turn, we observe that, due to the coercivity property \eqref{elast-visc-tensors-1}, 
 and by 
the (strict) positivity \eqref{discr-strict-pos}  of $\pwc \teta\tau$, there holds
\[
\begin{aligned}
\calE_\tau (\pwc \teta\tau(t), \pwc e\tau(t), \pwc z\tau(t), \pwc p \tau(t)) \geq c \Big( &  \| \pwc \teta\tau(t)\|_{L^1(\Omega)} + \| \pwc e\tau(t)\|_{L^2(\Omega;\mt_\sym^{d\times d})}^2 
+ \ass (\pwc z \tau(t), \pwc z\tau(t)) + \int_\Omega W(\pwc z \tau(t)) \dd x 
\\
& \quad 
+\tau  \| \pwc e\tau(t)\|_{L^\gamma(\Omega;\mt_\sym^{d\times d})}^\gamma +\tau  \| \pwc p\tau(t)\|_{L^\gamma(\Omega;\mt_\dev^{d\times d})}^\gamma \Big) -C
\end{aligned}
\]
Combining these facts, and repeating the very same calculations from the proof of  \cite[Prop.\ 4.3]{Rossi2016}, we arrive at the following estimate
\begin{equation}
\label{very1calc}
\begin{aligned}
&
\begin{aligned}
\int_\Omega | \pwl{\dot u}\tau (t)|^2 \dd x &  +  \| \pwc\teta\tau(t)\|_{L^1(\Omega)} +   \| \pwc e\tau(t)\|_{L^2(\Omega;\mt_\sym^{d\times d})}^2  + \tau
\| \pwc e\tau(t)\|_{L^\gamma(\Omega;\mt_\sym^{d\times d})}^\gamma 
\\
& 
+ \ass (\pwc z \tau(t), \pwc z\tau(t)) + \int_\Omega W(\pwc z \tau(t)) \dd x 
+
  \tau 
\| \pwc p\tau(t)\|_{L^\gamma(\Omega;\mt_\dev^{d\times d})}^\gamma
\end{aligned}
\\ & 
  \leq   C + 
C  \int_{0}^{\pwc{\mathsf{t}}{\tau}(t)}   \|  \pwl{\dot \varrho}\tau (r)  \|_{L^2(\Omega;\mt_\sym^{d\times d})}  \|   \upwc e{\tau}(r)\|_{L^2(\Omega;\mt_\sym^{d\times d})} \dd r +
  C  \int_{0}^{\pwc{\mathsf{t}}{\tau}(t)}  \| (\pwc \varrho\tau (r))_\dev \|_{L^\infty (\Omega;\mt_\dev^{d\times d})} \| \pwc \teta\tau(r) \|_{L^1(\Omega)} \dd r\\ & 
  \quad 
+  C  \int_{0}^{\pwc{\mathsf{t}}{\tau}(t)-\tau}   \|  \partial_t\pwwll{w}\tau(s+\tau) \|_{L^2(\Omega;\R^d)}  \|  \pwl{\dot u}\tau(s) \|_{L^2(\Omega;\R^d)}   \dd  s 
\\ &  \quad 
+  C   \int_{0}^{\pwc{\mathsf{t}}{\tau}(t)    }   \left( \| \sig{\partial_t \pwwll w\tau(r)}\|_{L^2(\Omega;\mt_\sym^{d\times d}) }+   \|  \sig{\pwl{\dot w}\tau (r) }   \|_{L^2(\Omega;\mt_\sym^{d\times d}) } \right)    \| \pwc e\tau(r) \|_{L^2(\Omega;\mt_\sym^{d\times d}) }  \dd r \\ & \quad+   C   \int_{0}^{\pwc{\mathsf{t}}{\tau}(t)}  \|   \sig{\pwl{\dot w}\tau (r) } \|_{L^\infty (\Omega;\mt_\sym^{d\times d})} \| \pwc\teta\tau(r) \|_{L^1(\Omega)}  \dd r  \,,
\end{aligned}
\end{equation}
where $\pwc \varrho\tau$ and  $\pwl \varrho\tau$ denote the piecewise constant/linear interpolants of the local means of the safe load function $\varrho$. 
Then, taking into account that  $W$ is bounded from below (cf.\ \eqref{hyp-phi-1}),  and applying  the Gronwall Lemma, we arrive at the energy bound \eqref{unif-energy-bound}, joint with the estimate
$\|\pwl{\dot u}\tau \|_{L^\infty(0,T; L^2(\Omega;\R^d))} \leq C$. 
Estimates \eqref{aprioU2}(2), \eqref{aprioU3}(2),   
  \eqref{aprioE1}, \eqref{aprioE3}, \eqref{aprioP3}, and  \eqref{aprio6-discr}(1) ensue.
  Observe that the latter implies the first bound in \eqref{log-added}, \MMM taking into account the elementary estimate
   \[
 \forall\, p\in [1,\infty) \ \exists\, C_p>0 \ \forall\, \theta>0 \, : \quad |\log(\theta)|^p \leq \theta + \frac1\theta + C_p\,.
 \]
 and the previously proved \eqref{strong-strict-pos}.\EEE
 We also infer that 
\begin{equation}
\label{by-product-of-first}
\sup_{t\in [0,T]} \ass (\pwc z\tau(t), \pwc z\tau(t)) + \sup_{t\in [0,T]} \left| \int_\Omega W(\pwc z\tau(t)) \dd x \right| \leq C\,.
\end{equation}
\MMM Observe that $0< \pwc z\tau \leq 1$, with the upper bound due to  the fact that, by unidirectionality of the evolution
of damage,
for every $x \in \Omega$ we have   $\pwc z\tau (x,t) \leq \pwc z\tau (x,0) = z_0(x) \leq 1$, where the last inequality ensues from
 \eqref{initial-z}.  Then, in view of \eqref{by-product-of-first}, we conclude
 \eqref{aprioZ1}. 
 \EEE
Finally, estimate \eqref{aprioP4} follows from the fact that $\pwc \zeta \tau \in \partial_{\dot p} \dip{\pwc z\tau}{\upwc \teta \tau}{\pwl{\dot p}\tau} \subset B_{C_R}(0) $ a.e.\ in $\Omega 
\times (0,T)$ by \eqref{bounded-subdiff}.
\par
\noindent 
\textbf{Step $2$: ad \eqref{discr-feasibility}.} 
The upper bound  $\pwc z\tau \leq 1$ follows from  To obtain the lower bound, 
we repeat the argument from the proof of \cite[Lemma 3.3]{Crismale-Lazzaroni}: due to the coercivity condition \eqref{hyp-phi-1}, 
for every $M>0$ there exists $\tilde{\zeta}>0 $ such that for all $0<z \leq \tilde \zeta$ there holds $W(z) \geq Mz^{-2d}$. 
Now, by contradiction suppose that $\pwc z\tau$ does not comply with the lower bound in   \eqref{discr-feasibility}. Then, 
corresponding to $\tilde\zeta$, 
there exist $\tilde\tau>0$, $t\in [0,T]$, and $x \in \Omega$ such that $\pwc z{\tilde\tau} (x,t) < \frac{\tilde \zeta}2$.
We now use that  $\pwc z{\tilde\tau}(\cdot, t) \in \spz \subset \mathrm{C}^{0,1/2}(\overline\Omega)$, which yields that 
\[
\exists\, \tilde{C}>0 \ \  \forall\, x, y\in \Omega\, : \quad |\pwc z{\tilde\tau}(x,t) {-} \pwc z{\tilde\tau}(y,t)| \leq \tilde{C}|x{-}y|^{1/2}\,,
\]
to deduce that $\pwc z{\tilde\tau}(y,t) <\tilde\zeta$ for every $y \in B_{\tilde r}(0)$, with $\tilde{r} = \left(\tfrac{\delta}{2\tilde C}\right)^2$. Hence
\[
 \int_\Omega W(\pwc z{\tilde\tau}(t)) \dd x \geq
  M  \int_{B_{\tilde r}(0)} (\pwc z{\tilde\tau}(t))^{-2d} \dd x \geq M (\tilde\delta)^{-2d} |B_{\tilde r}(0)| = M \frac{\omega_d}{(2\tilde C)^{2d}},
\]
with $\omega_d$ the Lebesgue measure of the unit ball in $\R^d$. Since $M$ is arbitrary, this contradicts estimate \eqref{by-product-of-first}.
We thus conclude 
 the lower bound for  the piecewise constant interpolant $\pwc z\tau$. The analogous statement
immediately  follows 
for the interpolant $\pwl z\tau$, since it is given by a convex combination of the values of $\pwc z\tau$.
\par
\noindent 
\textbf{Step $3$: Second a priori estimate.} We test the discrete heat equation 
\eqref{discrete-heat} by $(\tetau{k})^{\alpha-1}$, with $\alpha \in (0,1)$, thus obtaining   
\begin{equation}
\label{ad-est-temp1}
\begin{aligned}
& 
\begin{aligned}
\int_\Omega \Big(  \Gtau{k} {+} \bbD (\ztau k)  \Dtau{k} e :  \Dtau{k} e    {+}  & \did{\Dtau k z} {+}   \left|  \Dtau{k} z \right|^2 {+} \bar{\nu} \ass   (\Dtau{k} z,\Dtau{k} z) 
\\
&
{+} \dip{\ztau k}{\tetau{k-1}}{\Dtau{k} p}  {+}  \left|  \Dtau{k} p \right|^2\Big)  (\tetau{k})^{\alpha-1} \dd x
\end{aligned}
  \\ & \quad
- \int_\Omega \condu(\tetau k) \nabla \tetau k \nabla  (\tetau{k})^{\alpha-1} \dd  x + 
\int_{\partial\Omega} \gtau k  (\tetau{k})^{\alpha-1} \dd S  
\\
& 
\leq \int_\Omega  \left( \frac1{\alpha}\frac{(\tetau k)^\alpha - (\tetau {k-1})^\alpha}{\tau} {+} 
\tetau{k} \bbC(\ztau k) \bbE : \Dtau k e (\tetau{k})^{\alpha-1} {+} \tetau{k-1} \Dtau{k} z (\tetau{k})^{\alpha-1}  \right)  \dd  x,
\end{aligned}
\end{equation}
where we have used  the concavity of the function
$\psi(\teta)=\tfrac1{\alpha}\teta^\alpha$, leading to the  estimate  $(\tetau{k} {-} \tetau{k-1}) (\tetau k)^{\alpha-1} \leq  \psi(\tetau k)  - \psi (\tetau {k-1})$.  Note that we have omitted the positive part $(\ztau k)^+$ in the argument of $\bbC$ in the thermal expansion term, since $\ztau k >0$ on $\Omega$ by virtue of  \eqref{discr-feasibility}. 
Therefore,  multiplying by $\tau$, summing over the index $k$,   neglecting some  positive terms on the left-hand side of \eqref{ad-est-temp1}
and observing that
$ \tetau{k-1} \Dtau{k} z (\tetau{k})^{\alpha-1}  \leq 0 $ a.e.\ in $\Omega$ since $\tetau{k-1},\, \tetau k>0$ while $\Dtau k z \leq 0$, 
  we obtain
 for all $t \in (0,T]$
 \begin{equation} 
 \label{ad-est-temp2}
 \begin{aligned}
 &
 \frac{4(1-\alpha)}{\alpha^2} \int_0^{\pwc{\mathsf{t}}{\tau}(t)} \int_\Omega \condu(\pwc \teta{\tau}) |\nabla ((\pwc \teta{\tau})^{\alpha/2}) |^2 \dd x \dd s  +   C_\bbD^1  \int_0^{\pwc{\mathsf{t}}{\tau}(t)}  \int_\Omega   
 | \pwl {\dot e}\tau|^2 (\pwc \teta \tau)^{\alpha-1}  
 \dd x \dd s
\\
&
\leq
 \frac1\alpha\int_\Omega  (\pwc \teta\tau(t) )^{\alpha} \dd x -    \frac1\alpha\int_\Omega  (\teta_0)^{\alpha} \dd x +\int_0^{\pwc{\mathsf{t}}{\tau}(t)}  \int_\Omega   \pwc \teta\tau(t) \bbC(\pwc z\tau) \bbE : \pwl{\dot e}\tau(t)  ( \pwc \teta\tau(t) )^{\alpha-1}  \dd  x 
 \doteq
 I_1+I_2+I_3\,.
 \end{aligned}
 \end{equation}
 Next, we use that  
 \[
  I_1  \leq \frac1\alpha \| \pwc \teta\tau\|_{L^\infty (0,T; L^1(\Omega))} + C \leq C
 \]
 via Young's inequality, using that  $\alpha \in (0,1)$, and taking into account the previously obtained bound  \eqref{aprio6-discr}(1). 
 We have $I_2 \leq 0$, whereas for $I_3$ we observe  that 
 \begin{equation}
 \label{bound-bbC}
 \exists\, C>0 \ \forall\, (x,t) \in \Omega \times [0,T] \, \quad
 |\bbC(\pwc z\tau(x,t))| \leq C\,.
 \end{equation}
  by the continuity of the function $\bbC$, combined with the previously obtained
 property \eqref{discr-feasibility}. Therefore, also taking into account that $\bbE \in L^\infty (\Omega;\mathrm{Lin}(\mt_\sym^{d\times d}))$, we obtain 
   \begin{equation}
 \label{est-temp-I3}
 I_3 \leq 
\frac{ C_\bbD^1 }4   \int_0^{\pwc{\mathsf{t}}{\tau}(t)}  \int_\Omega  | \pwl {\dot e}\tau|^2 (\pwc \teta \tau)^{\alpha-1}   \dd x \dd s + C  \int_0^{\pwc{\mathsf{t}}{\tau}(t)}  \int_\Omega
(\pwc \teta \tau)^{\alpha+1}   \dd x \dd s\,.
 \end{equation}
 Absorbing the first term on the right-hand side of \eqref{est-temp-I3} into  the left-hand side of \eqref{ad-est-temp2} and taking into account the coercivity condition \eqref{hyp-K} on $\condu$,  
we infer from  \eqref{ad-est-temp2} that 
 \begin{equation}
  \label{ad-est-temp3}
 c   \int_0^{\pwc{\mathsf{t}}{\tau}(t)} \int_\Omega  | \nabla  (\pwc \teta \tau)^{(\mu+\alpha)/2} |^2 \dd x \dd s \leq C +  C  \int_0^{\pwc{\mathsf{t}}{\tau}(t)}  \int_\Omega
(\pwc \teta \tau)^{\alpha+1}   \dd x \dd s\,.
 \end{equation}
 From now on, 
 we can repeat the calculations developed in   \cite[(3.8)--(3.12)]{Rocca-Rossi} (and based on techniques from   \cite{FPR09}),  for the analogous estimate. We refer to \cite{Rocca-Rossi} for all the detailed calculations, \MMM leading to an estimate for $ (\pwc \teta \tau)^{(\mu+\alpha)/2}$ in $L^2(0,T;H^1(\Omega))$, i.e.\ the first of \eqref{est-temp-added?-bis}.
 From this bound, relying on the constraint $\alpha \geq 2-\mu$ (cf.\ again  \cite{Rocca-Rossi}  for all details)
 we deduce 
 estimate   \eqref{aprio6-discr}(2).  The latter   in turns yields  \eqref{log-added}(2). 
The second  of  \eqref{est-temp-added?-bis} ensues from
$  (\pwc \teta\tau)^{(\mu-\alpha)/2} \leq  (\pwc \teta\tau)^{(\mu+\alpha)/2}  +1 $ a.e.\ in $Q$
 and 
from
  \[
\int_\Omega |\nabla  (\pwc \teta\tau)^{(\mu-\alpha)/2}|^2 \dd x  = C  \int_\Omega  ( \pwc \teta\tau)^{\mu-\alpha - 2} | \nabla \pwc \teta \tau|^2 \dd x \leq \frac{C}{\bar\teta^{2\alpha}} 
 \int_\Omega  ( \pwc \teta\tau)^{\mu+\alpha - 2} | \nabla \pwc \teta \tau|^2 \dd x \leq C
 \]
where we have used that $\alpha\geq 0$, the strict positivity of $\pwc\teta\tau$, and the previously obtained estimate for   $(\pwc \teta \tau)^{(\mu+\alpha)/2}$. \EEE 
\par\noindent
\textbf{Step $4$: Third a priori estimate.} We consider the mechanical energy inequality \eqref{mech-ineq-discr} written for $s=0$.   Since most of the terms on  its right-hand side
can be handled   by the very same calculations developed for the right-hand side terms of \eqref{total-enid-discr}, we refer to  the proof of \cite[Prop.\ 4.3]{Rossi2016} for all details and only mention
how to estimate the third and the fourth integral terms. 
We use that 
\begin{align}
&
 \int_{0}^{\pwc{\mathsf{t}}{\tau}(t)} \int_\Omega \pwc\teta\tau \bbC(\pwc z\tau) \bbE {:} \pwl{\dot e}\tau \dd x  \dd r \leq   \delta \int_0^{\pwc{\mathsf{t}}{\tau}(t)} \int_\Omega | \pwl{\dot e}\tau |^2 \dd x \dd r 
 + C_\delta  \| \pwc \teta \tau \|_{L^2(0,T; L^2(\Omega))}^2
 \nonumber
\intertext{
via Young's inequality,
 the previously obtained bound \eqref{bound-bbC}, and the fact that $|\bbE(x) | \leq C$, and that} 
 &
 \int_{0}^{\pwc{\mathsf{t}}{\tau}(t)} \int_\Omega \upwc \teta\tau \pwl{\dot {z}}\tau \dd x \dd r \leq \delta  \int_0^{\pwc{\mathsf{t}}{\tau}(t)} \int_\Omega
 | \pwl{\dot {z}}\tau|^2 \dd x \dd r + C_\delta   \| \upwc \teta \tau \|_{L^2(0,T; L^2(\Omega))}^2
 \nonumber
 \end{align}
  with the constant $\delta>0$ chosen in such a way as to absorb the terms $\iint   | \pwl{\dot e}\tau |^2$ and
$  \iint   | \pwl{\dot z}\tau |^2$ 
    into the left-hand side of  \eqref{mech-ineq-discr}. 
  Since $ \| \pwc \teta \tau \|_{L^2(Q)},\,  \| \upwc \teta \tau \|_{L^2(Q)}  \leq C$ thanks to  \eqref{aprio6-discr},  we conclude a uniform bound 
for all   
 the  terms on the right-hand side of  \eqref{mech-ineq-discr}. Therefore, estimates 
  \eqref{aprioE2},  \eqref{aprioZ2},  \eqref{aprioZ3}, 
  and \eqref{aprioP2} ensue. 
  We then obtain  \eqref{aprioU2}(1) and \eqref{aprioU3}(1)  via   kinematic admissibility. Furthermore, 
  \eqref{aprioP2} clearly implies estimate
\eqref{aprioP1}, and then  \eqref{aprioU1} again by kinematic admissibility. 
\par
It follows from  \eqref{aprioE1}, \eqref{aprioE2}, \eqref{aprioE3}, and  \eqref{aprio6-discr}, also taking into account the bound  \eqref{bound-bbC} and its analogue for $\bbD(\pwc z\tau)$, 
 that the stresses $(\pwc \sigma\tau )_\tau$ are uniformly bounded in
$L^{\gamma/(\gamma{-}1)}(Q; \mt_\sym^{d\times d})$. Therefore, also taking into account \eqref{converg-interp-L},    a comparison argument in the discrete momentum balance  \eqref{eq-u-interp} yields  
\eqref{aprioU3}(3).
\par
Finally, estimate \eqref{aprioZ4} ensues from a comparison argument in \eqref{eq-z-interp}. 
\par\noindent 
\textbf{Step $5$: Fourth a priori estimate. } Let us now shortly sketch the argument for \eqref{aprio_Varlog}.
The very same calculations as in the proof of \cite[Prop.\ 4.10]{Rocca-Rossi}
lead us to deduce, from  the discrete  entropy inequality
\eqref{entropy-ineq-discr} written on
the generic sub-interval  $[\mathsf{s}_{i-1},\mathsf{s}_i]$ of a partition $0 =\mathsf{s}_0 < \mathsf{s}_1 < \ldots < \mathsf{s}_J =T$ of  $[0,T]$,
 the following estimate  for the total variation of $\log(\pwc\teta\tau)$:
\begin{equation}
\label{genialata}
\begin{aligned} & 
\sum_{i=1}^J \left|  \pairing{}{W^{1,d+\epsilon}(\Omega)}{\log(\pwc\teta\tau (\mathsf{s}_{i}))  - \log(\pwc\teta\tau (\mathsf{s}_{i-1})) }{\varphi} \right| 
\\ & \leq  \sum_{i=1}^J   
\int_\Omega (\log(\pwc\teta\tau (\mathsf{s}_{i}))  - \log(\pwc\teta\tau (\mathsf{s}_{i-1}))) |\varphi| \dd x
+ \Lambda_{i,\tau} (|\varphi|) + |\Lambda_{i,\tau} (\varphi^+)| + |\Lambda_{i,\tau} (\varphi^-)| 
\end{aligned} 
\end{equation}
for all $\varphi \in W^{1,d+\epsilon}(\Omega)$, with $\epsilon>0$ arbitrary.  Here we  have used the place-holder
 \begin{equation}
 \label{pl-ho-varphi}
\begin{aligned}
&
\Lambda_{i,\tau} (\varphi) :=
   \int_{\pwc{\mathsf{t}}{\tau}(\mathsf{s}_{i-1})}^{\pwc{\mathsf{t}}{\tau}(\mathsf{s}_i)}  \int_\Omega  \condu(\pwc\teta\tau) \nabla \log(\pwc\teta\tau) \nabla \varphi \dd x \dd r
  + \int_{\pwc{\mathsf{t}}{\tau}(\mathsf{s}_{i-1})}^{\pwc{\mathsf{t}}{\tau}(\mathsf{s}_i)} \int_\Omega \bbC (\pwc z\tau) \bbE  {:} \pwl{\dot e}\tau  \varphi  \dd x \dd r
\\
& \ -  \int_{\pwc{\mathsf{t}}{\tau}(\mathsf{s}_{i-1})}^{\pwc{\mathsf{t}}{\tau}(\mathsf{s}_i)} \int_\Omega  \condu(\pwc\teta\tau) \frac{\varphi}{\pwc\teta\tau} \nabla (\log(\pwc\teta\tau)) \nabla \pwc\teta\tau \dd x \dd r -  \int_{\pwc{\mathsf{t}}{\tau}(\mathsf{s}_{i-1})}^{\pwc{\mathsf{t}}{\tau}(\mathsf{s}_i)}   \int_{\partial\Omega} \pwc g{\tau}  \frac{\varphi}{\pwc\teta\tau} \dd S \dd r
 \\
 & \ - \int_{\pwc{\mathsf{t}}{\tau}(\mathsf{s}_{i-1})}^{\pwc{\mathsf{t}}{\tau}(\mathsf{s}_i)} \int_\Omega \left(\pwc G\tau+
  \bbD(\pwc z\tau)  \pwl{\dot e}\tau{:}  \pwl{\dot e}\tau + \did{\pwl{\dot z}\tau} + |\pwl{\dot z}\tau|^2  + \bar\nu \ass (\pwl{\dot z}\tau, \pwl{\dot z}\tau) 
  -\upwc \teta\tau \pwl{\dot z}\tau 
  + \dip{\pwc z\tau}{\upwc\teta\tau}{\pwl{\dot p}\tau}
 + | \pwl {\dot p}\tau|^2  \right)
 \frac{\varphi}{\pwc\teta\tau} \dd x \dd r.
\end{aligned}
\end{equation}
The second, third, and fourth terms on the r.h.s.\ of \eqref{genialata} can be estimated by  adapting  the  computations  from the proof of 
\cite[Prop.\ 4.10]{Rocca-Rossi}, taking into account the previously obtained bounds.
 All in all, from \eqref{genialata} we obtain that 
\[
\sum_{i=1}^J \left|  \pairing{}{W^{1,d+\epsilon}(\Omega)}{\log(\pwc\teta\tau (\mathsf{s}_{i}))  - \log(\pwc\teta\tau (\mathsf{s}_{i-1})) }{\varphi} \right|  \leq 
\int_\Omega \left( \log(\pwc\teta\tau(T)) - \log(\teta_0) \right) |\varphi| \dd x + C \leq C
\]
for every $\varphi \in W^{1,d+\epsilon}(\Omega) $  with $ \|\varphi\|_{ W^{1,d+\epsilon}(\Omega)} \leq 1$, where for the last estimate we have used the bound for $\log(\pwc\teta\tau)$ in
$L^\infty (0,T; L^1(\Omega))$ from \eqref{log-added}.  Thus, \eqref{aprio_Varlog} ensues.
\par\noindent 
\textbf{Step $6$: Fifth a priori estimate:} We now assume  the stronger condition \eqref{hyp-K-stronger}.
We multiply the discrete heat equation \eqref{discrete-heat} by a test function  $\varphi \in W^{1,\infty}(\Omega)$ and  integrate in space.  
We thus obtain for a.a.\ $t\in (0,T)$
\begin{equation}
\label{analog-7th-est}
\left|\int_\Omega \pwl{\dot \teta}\tau (t)  \varphi \dd x \right| 
\leq \left| \int_\Omega \condu(\pwc \teta\tau(t)) \nabla  \pwc \teta\tau(t) \nabla \varphi \dd x \right| + \left| \int_\Omega \pwc J{\tau}(t) \varphi \dd x \right|
  + \left| \int_{\partial\Omega} \pwc g\tau(t)  \varphi \dd S\right| \doteq I_1+I_2+I_3\,,
\end{equation}
with 
$\pwc J{\tau}: =
   \pwc G\tau + \bbD(\pwc z\tau)   \pwl {\dot e}\tau  : \pwl {\dot e}\tau 
     -\pwc\teta\tau \bbC(\pwc z \tau) \bbE : \pwl {\dot e}\tau
     + \did{\pwl{\dot z}\tau} + |\pwl{\dot z}\tau|^2 + \bar\nu \ass (\pwl{\dot z}\tau,\pwl{\dot z}\tau)
     -\upwc \teta\tau \pwl{\dot z}\tau
    + \dip{\pwc z\tau}{\upwc\teta\tau}{ \pwl {\dot p}\tau}  +  \left|  \pwl {\dot p}\tau  \right|^2$. 
   With the very same calculations as in the proof of \cite[Prop.\ 4.10]{Rocca-Rossi}, relying on  \eqref{converg-interp-g},  
    \eqref{converg-interp-h},  on  \eqref{hyp-K-stronger}, and on the previously obtained estimates 
 \eqref{aprioE2}, \eqref{aprioZ2},  \eqref{aprioZ3}, \eqref{aprioP2},     \eqref{aprio6-discr},  and  \eqref{est-temp-added?-bis}, 
 we infer that 
 \[
I_1+I_2+I_3 \leq \calL_\tau(t) \| \varphi \|_{W^{1,\infty} (\Omega;\R^d)}
 \]
 for a family $(\mathcal{L}_\tau)_\tau$ \MMM that \EEE  is uniformly bounded in $L^1(0,T)$. 
 Hence, 
 estimate \eqref{aprio7-discr}
 follows. This concludes the proof of Proposition \ref{prop:aprio}. 
\QED

\section{\bf Proofs of Theorems \ref{mainth:1} and   \ref{mainth:2}}
\label{s:4}

\noindent
We start by fixing the compactness properties of a family
\[
(\pwc u{\tau_k}, \pwl u{\tau_k},   \pwwll  u{\tau_k},  \pwc e{\tau_k}, \upwc e{\tau_k},  \pwl e{\tau_k}, \pwc z{\tau_k}, \pwl z{\tau_k},  \pwc p{\tau_k}, \pwl p{\tau_k},
 \pwc\teta{\tau_k},  \upwc\teta{\tau_k}, \pwl{\teta}{\tau_k},  \pwc \omega{\tau_k}, \pwc \zeta{\tau_k})_k,
\]
 of approximate solutions
 in the following result, where we again tacitly assume the validity of
the conditions from Sec.\ \ref{ss:2.1}.
We will only distinguish the case where we only require \eqref{hyp-K}, from that where \eqref{hyp-K-stronger}   is also imposed and we are thus in the position to enhance the convergence properties of the temperature variables by virtue of the additional estimate \eqref{aprio7-discr}.
\begin{lemma}[Compactness]
\label{l:compactness}
Assume \eqref{hyp-K}. Then, for any sequence $\tau_k \downarrow 0$ there exist a (not relabeled) subsequence and a seventuple $(u, e, z, p,\teta, \omega,\zeta)$ such that the following convergences hold
\begin{subequations}
\label{convergences-cv}
\begin{align}
&
\label{cvU1}
 \pwl \uu{\tau_k} \weaksto \uu  &&  \text{ in $H^1(0,T;  H^1(\Omega;\R^d)) \cap
W^{1,\infty}(0,T;L^2(\Omega;\R^d))$,}
\\
&
\label{cvU2}
  \pwc \uu{\tau_k},\, \upwc \uu{\tau_k}  \to \uu &&
 \text{ in $L^\infty(0,T;H^{1-\epsilon}(\Omega;\R^d)) $ for all $\epsilon \in (0,1]$,}
 \\
 &
 \label{cvU3}
  \pwl \uu{\tau_k} 
\to \uu  && \text{ in $\mathrm{C}^0([0,T];H^{1-\epsilon}(\Omega;\R^d))  $ for all $\epsilon \in (0,1]$,}
 \\
  &
 \label{cvU3-bis}
\pwwll {u}{\tau_k}(\pwc{\mathsf{t}}{\tau_k}(t)) \weakto \dot{\uu}(t)  && 
\text{ in $L^2(\Omega;\R^d)$ for every $t\in [0,T]$,}
\\
 &
 \label{cvU3-ter}
\pwwll {u}{\tau_k}(\pwc{\mathsf{t}}{\tau_k}(t))=\pwl {\dot u}{\tau_k}(t) \weakto \dot{\uu} (t) && 
\text{ in $H^1(\Omega;\R^d)$ for almost all  $t\in (0,T)$,}
\\
&
\label{cvU4}
\partial_t \pwwll {\uu}{\tau_k} \weakto \ddot u &&
 \text{ in $L^{\gamma/(\gamma-1)}(0,T;W^{1,\gamma}(\Omega;\R^d)^*) $,}
 \\
&
\label{cvE1}
\pwc e{\tau_k}, \, \upwc e{\tau_k}  \weaksto e && \text{ in $L^\infty(0,T; L^2(\Omega;\mt_\sym^{d\times d}))$,}
\\
&
\label{cvE2}
\pwl e{\tau_k} \weakto e && \text{ in $H^1(0,T; L^2(\Omega;\mt_\sym^{d\times d}))$,}
\\
&
\label{cvE3-bis}
\pwl e{\tau_k} \to e && \text{ in $\rmC_{\mathrm{weak}}^0([0,T]; L^2(\Omega;\mt_\sym^{d\times d}))$,}
\\
&
\label{cvE3}
\tau^\beta |\pwc e{\tau_k}|^{\gamma-2}  \pwc e{\tau_k} \to 0 && \text{ in $L^{\infty}(0,T; L^{\gamma/(\gamma{-}1)}(\Omega; \mt_\sym^{d\times d}))$ for all } \beta > 1-\frac1\gamma,
 \\
 &
 \label{cvZ1}
 \pwc z{\tau_k}, \pwl z{\tau_k} \weaksto z && \text{ in $L^\infty (0,T;\spz)$},
 \\
 & 
 \label{cvZ2}
 \pwl{ z}{\tau_k} \weakto z && \text{ in $H^1(0,T;L^2(\Omega))$}, 
 \\
 & 
 \label{cvZ3}
 \pwl {z}{\tau_k} \weakto z && \text{ in $H^1(0,T;\spz)$ if $\nu>0$},
 \\
 & 
 \label{cvZ4}
  \pwc z \tau,\, \pwl z\tau \to z && \text{ in $L^\infty(0,T;\mathrm{C}^0(\overline\Omega))$},    
 \\
 & 
 \label{cvZ5}
 \pwc {\omega}{\tau_k} \weakto \omega && \text{ in $L^2(0,T;\spz^*)$ if $\nu>0$},
 \\
&
\label{cvP1}
\pwc p{\tau_k} \weaksto p && \text{ in $L^\infty(0,T; L^2(\Omega;\mt_\sym^{d\times d}))$,}
\\
&
\label{cvP2}
\pwl p{\tau_k} \weakto p && \text{ in $H^1(0,T; L^2(\Omega;\mt_\sym^{d\times d}))$,}
\\
&
\label{cvP3-bis}
\pwl p{\tau_k} \to p && \text{ in $\rmC_{\mathrm{weak}}^0([0,T]; L^2(\Omega;\mt_\dev^{d\times d}))$,} \\
&
\label{cvP3}
\tau^\beta|\pwc p{\tau_k}|^{\gamma-2}  \pwc p{\tau_k} \to 0 && \text{ in $L^{\infty}(0,T;   L^{\gamma/(\gamma{-}1)}(\Omega; \mt_\dev^{d\times d}))$ for all } \beta > 1-\frac1\gamma,
\\
 &
 \label{cvT1}
 \pwc \teta{\tau_k}, \, \upwc \teta{\tau_k}\ \weakto \teta && \text{ in $L^2 (0,T; H^1(\Omega))$},
 \\
&
\label{cvT2}
\log(\pwc\teta{\tau_k}),\,  \log(\upwc\teta{\tau_k})   \weaksto  \log(\teta)  &&
\text{ in } L^2 (0,T; H^1(\Omega))  \cap  L^\infty (0,T; W^{1,d+\epsilon}(\Omega)^*)   \quad \text{for every } \epsilon>0,
\\
&
\label{cvT4}
 \log(\pwc\teta{\tau_k}(t)),\,   \log(\upwc\teta{\tau_k}(t))  \weakto \log(\teta(t))   &&   \text{ in $H^1(\Omega)$ for almost all $t \in (0,T)$,}  
\\
&
\label{cvT5}
\pwc\teta{\tau_k},\,  \upwc\teta{\tau_k}\to \teta &&  \text{  in $L^h(Q)$
for all $h\in [1,8/3) $ if $d=3$ and all $h\in [1, 3)$ if $d=2$,}
\\
& 
\label{cvZ}
\pwc\zeta{\tau_k} \weaksto \zeta && \text{ in $L^\infty (Q;\mt_\dev^{d \times d})$}.
\end{align}
\end{subequations}
The functions $z$ and   $\teta$  also fulfill \eqref{strong-pos-z}, with $\zeta_*$ from \eqref{discr-feasibility}, and 
\begin{equation}
\label{additional-teta}
\teta \in L^\infty (0,T; L^1(\Omega)) \text{ and }
 \teta \geq \bar{\teta} \text{ a.e.\ in $Q$}
 \end{equation}
  with $\bar{\teta}$  from \eqref{discr-strict-pos}.
\par
Furthermore, under condition \eqref{hyp-K-stronger} we also have $\teta \in \mathrm{BV} ([0,T]; W^{1,\infty} (\Omega)^*)$, and
\begin{subequations}
\begin{align}
& \label{cvT6}
\pwc\teta{\tau_k} ,\, \upwc\teta{\tau_k}  \to \teta && \text{ in } L^2 (0,T; Y) \text{ for all $Y$ such that $H^1(\Omega) \Subset Y \subset W^{1,\infty} (\Omega)^*$},
\\
&
\label{cvT7}
\pwc\teta{\tau_k}(t),\,  \upwc\teta{\tau_k}(t) \weaksto \teta(t) && \text{ in } W^{1,\infty} (\Omega)^* \text{ for all } t \in [0,T].
\end{align}
\end{subequations}
\end{lemma}
\noindent
\begin{proof}[Sketch of the proof]
Convergences \eqref{cvZ1}--\eqref{cvZ5} follow from standard weak and strong compactness arguments, the latter based, e.g., on the Aubin-Lions type compactness
tools from \cite{simon86}.
\par
Let us  comment on \eqref{cvU3-bis}: It follows from estimate \eqref{aprioU3} and the aforementioned compactness results that the sequence  $(\pwwll{u}{\tau_k})_k$ admits a subsequence  converging to some limit $v$, weakly$^*$  in $L^2(0,T;H^1(\Omega))
 \cap L^\infty(0,T;L^2(\Omega)) \cap W^{1,\gamma/(\gamma-1)}(0,T;W^{1,\gamma}(\Omega;\R^d)^*)$, and strongly in 
 $\rmC^0([0,T];X)$ for any space $X$ with $L^2(\Omega) \Subset X$. Therefore, 
 for every $t\in [0,T]$ we have that $\pwwll {u}{\tau_k}(\pwc{\mathsf{t}}{\tau_k}(t))  \to v(t)$ in $X$. We combine this with the information that 
 $(\pwwll{u}{\tau_k})_k$  is bounded in $ L^\infty(0,T;L^2(\Omega)) $ to extend the latter  statement to weak convergence in $L^2(\Omega)$. The identification $v = \dot{u}$ follows from the estimate 
\[
 \| \pwwll \uu{\tau_k} - \pwl {\dot \uu}{\tau_k}\|_{L^\infty
(0,T;  W^{1,\gamma}(\Omega;\R^d)^* )}\leq \tau_k^{1/\gamma}  \| \partial_t \pwwll
 {\uu}{\tau_k} \|_{L^{\gamma/(\gamma{-}1)} (0,T;W^{1,\gamma}(\Omega;\R^d)^*)} \leq S {\tau_k}^{1/\gamma}.
\]
As for \eqref{cvU3-ter},  we apply the compactness result 
stated in Theorem \ref{th:mie-theil} below,
 with the choices
$\ell_k =  \pwwll \uu{\tau_k}  \circ \pwc{\mathsf{t}}{\tau_k}$, $\bsV = H^1(\Omega;\R^d)$, and $\bsY = W^{1,\gamma}(\Omega;\R^d)$. 
We thus deduce (cf.\ \eqref{weak-ptw-B} ahead) that 
\[
 \pwwll \uu{\tau_k}(\pwc{\mathsf{t}}{\tau_k} (t)) \weakto \dot{u}(t)  \text{ in } H^1(\Omega;\R^d) \ \foraa\, t \in (0,T). 
\]
Hence, \eqref{cvU3-ter} ensues, taking into account that $ \pwwll \uu{\tau_k} \circ \pwc{\mathsf{t}}{\tau_k}  \equiv \pwl{\dot u}{\tau_k}$ a.e.\ in $(0,T)$.
\par
 For  all the other convergence statements, the reader is referred to the proof of \cite[Lemma 4.6]{Rossi2016}: let us only mention that the pointwise convergences 
 \eqref{cvT4} follow from 
  Theorem 
  \ref{th:mie-theil}, as well. 
  \end{proof}
  \par
  We refer to \cite{Rocca-Rossi} and \cite{Rossi2016} (for a slight refinement of) the proof of the following compactness result.
\begin{theorem}
\label{th:mie-theil}
Let $\bsV$ and $\bsY$ be two (separable) reflexive Banach spaces  such that
$\bsV \subset \bsY^*$ continuously. Let
 $(\ell_k)_k \subset L^p
(0,T;\bsV) \cap \mathrm{B} ([0,T];\bsY^*)$ be  bounded  in $L^p
(0,T;\bsV) $ and suppose in addition  that
\begin{align}
\label{ell-n-0}
&
\text{$(\ell_k(0))_k\subset \bsY^*$ is bounded},
\\
&
\label{BV-bound}
\exists\, C>0 \  \ \forall\, \varphi \in \overline{B}_{1,\bsY}(0)  \ \  \forall\, k \in \N\, : \quad
 \mathrm{Var}(\pairing{}{\bsY}{\ell_k}{ \varphi}; [0,T] )  \leq C,
\end{align}
where, for  given $\ell \in
\mathrm{B}([0,T];\bsY^*)$ and $\varphi \in \bsY$ we set
\begin{equation}
\label{var-notation}
\begin{aligned}
\mathrm{Var}(\pairing{}{\bsY}{\ell}{ \varphi}; [0,T] ) : =  \sup \{  \sum_{i=1}^J
\left |\pairing{}{\bsY}{\ell(\mathsf{s}_{i})}{ \varphi} - \pairing{}{\bsY}{\ell(\mathsf{s}_{i-1})}{ \varphi}  \right|
\, :  \
0 =\mathsf{s}_0 < \mathsf{s}_1 < \ldots < \mathsf{s}_J =T \} \,.
\end{aligned}
\end{equation}
\par
Then, there exist  a  (not relabeled) subsequence
 $(\ell_{k})_k$
 and a function $\ell \in L^p (0,T;\bsV) \cap L^\infty (0,T; \bsY^*) $ 
 such that  as $k\to \infty$
 \begin{align}
 \label{weak-LpB}
 &
 \ell_{k} \weaksto \ell \quad \text{ in } L^p (0,T;\bsV) \cap L^\infty (0,T;\bsY^*),
 \\
\label{weak-ptw-B}
&
\ell_{k}(t) \weakto \ell(t) \quad \text{ in } \bsV\quad \foraa t \in (0,T).
\end{align}
\par
Furthermore, for almost all $t \in (0,T)$ and any sequence $(t_k)_k \subset [0,T]$ with $t_k \to t$ there holds 
\begin{equation}
\label{enhSav}
 \ell_{k}(t_k) \weakto \ell(t)  \qquad \text{ in $\bsY^*$. }
 \end{equation}
\end{theorem}
\par 
For expository reasons, 
in developing the proofs of our existence results
 we will
reverse the order with which we have presented them. More precisely, we will  start with  the
\textbf{\underline{proof of Theorem  \ref{mainth:2}}}
and develop the existence of weak energy solutions and, in addition, establish the validity of the entropy inequality. 
 Let us consider a null sequence $(\tau_k)_k$ and, correspondingly, a sequence
\[
(\pwc u{\tau_k}, \pwl u{\tau_k},   \pwwll  u{\tau_k},  \pwc e{\tau_k}, \upwc e{\tau_k},  \pwl e{\tau_k}, \pwc z{\tau_k}, \pwl z{\tau_k},  \pwc p{\tau_k}, \pwl p{\tau_k},
 \pwc\teta{\tau_k},  \upwc\teta{\tau_k}, \pwl{\teta}{\tau_k},  \pwc \omega{\tau_k}, \pwc \zeta{\tau_k})_k,
\]
of solutions to the approximate thermoviscoelastoplastic damage system  \eqref{syst-interp}
along which  convergences \eqref{convergences-cv} to a seventuple $(u,e,z,p,\teta, \omega,\zeta)$ hold.
 Exploiting them we will pass to the limit 
in the time-discrete version of the momentum balance. In order to take the limit  of the damage and  plastic flow rules, and of the temperature equation, we will need to enhance the convergence properties 
of the approximate solutions. We will do so by establishing the validity of the mechanical energy balance. Therefrom  we will strengthen some weak convergence properties, derived by compactness, via standard ``limsup''-arguments.
 We will thus show that the quintuple $(u,e,z,p,\teta)$ 
is a weak energy solution of the (initial-boundary value problem for the) regularized  thermoviscoelastoplastic damage system. We will also deduce the validity of the entropy inequality.
\par
\noindent
\emph{Step $0$: ad the initial conditions \eqref{initial-conditions} and the kinematic admissibility \eqref{kin-admis}}. 
We use the uniform  convergences \eqref{cvU3}, \eqref{cvE3-bis}, \eqref{cvZ4},  and \eqref{cvP3-bis}, as well as the pointwise convergences
 \eqref{cvU3-bis}, 
\eqref{cvT7},
 to pass to the limit in the discrete initial conditions \eqref{discr-Cauchy}, also taking into account convergences \eqref{approx-e_0}
  for the approximate initial data $(e_{\tau_k}^0, p_{\tau_k}^0)_k$.
  Also exploiting  \eqref{converg-interp-w}
  for $(\pwc w{\tau_k})_k$, we  pass to the limit in the discrete version of the kinematic admissibility condition. We thus conclude \eqref{kin-admis}. 
  \par
 \par\noindent
\emph{Step $1$: ad the momentum balance \eqref{w-momentum-balance}.} 
Combining convergence \eqref{cvZ4} with the uniform continuity of the mappings
 $z\in [0,1]\mapsto \bbC(z),\, \bbD(z)$, 
we infer that 
\begin{equation}
\label{cv_CC_DD}
\bbC(\pwc z{\tau_k}) \to \bbC(z), \quad
\bbD(\pwc z{\tau_k}) \to \bbD(z)  \quad \text{ in $L^\infty(0,T; \mt_\sym^{d\times d})$.}
\end{equation}
 Therefore, in view of 
convergences \eqref{cvE1}--\eqref{cvE3} and \eqref{cvT1} we have that 
\begin{equation}
\label{cvS}
\pwc \sigma{\tau_k} = \bbD(\pwc z{\tau_k})  \pwl{\dot e}{\tau_k} + \bbC(\pwc z{\tau_k})  \pwc e{\tau_k} 
+ \tau |\pwc e{\tau_k}|^{\gamma-2} \pwc e{\tau_k}  - \pwc{\teta}{\tau_k} \bbC(\pwc z{\tau_k}) \bbE  \weakto \sigma = \bbD(z) \dot e +\bbC(z) e - \teta  \bbC(z)\bbE  
\end{equation}
 in 
$L^{\gamma/(\gamma-1)} (Q;\mt_\sym^{d\times d})$.
 With this stress convergence, with convergence \eqref{cvU4} 
 for $(\partial_t \pwwll u{\tau_k})_k$, 
 and with \eqref{converg-interp-L} for $(\pwc \calL{\tau_k})_k$, we pass to the limit in the  \MMM integrated version of the discrete momentum balance \eqref{eq-u-interp}. With a localization argument we 
  \EEE conclude
  that $(u,e, z,\teta)$ fulfill \eqref{w-momentum-balance} with test functions in $W_\Dir^{1,\gamma}(\Omega;\R^d)$. \MMM Taking into account that $\sigma =  \bbD(z) \dot e +\bbC(z) e - \teta  \bbC(z)\bbE \in L^2(0,T;L^2(\Omega;\mt_\sym^{d\times d}))$,  \EEE
 a  comparison argument in  \eqref{w-momentum-balance}  yields  that $\ddot{u} \in L^2(0,T; H_{\Dir}^1(\Omega;\R^d)^*)$, whence \eqref{reg-u},
 and by density we conclude that 
 \eqref{w-momentum-balance} holds with test functions in $H_{\Dir}^1(\Omega;\R^d)$. We have thus established   the momentum
 balance. 
 \par
 \noindent 
 \emph{Step $2$: ad the entropy inequality \eqref{entropy-ineq}.} 
  Let us fix a positive test  function 
  $\varphi \in \rmC^0 ([0,T]; W^{1,\infty}(\Omega)) \cap H^1(0,T; L^{6/5}(\Omega))$ 
  for \eqref{entropy-ineq}, and approximate it with the discrete
   test functions from \eqref{discrete-tests-phi}: 
   their interpolants $\pwc \varphi\tau, \, \pwl \varphi\tau$ 
   converge to $\varphi$ in the sense of 
    \eqref{convergences-test-interpolants}.
    We now take the limit of 
     the discrete entropy inequality \eqref{entropy-ineq-discr},
     tested by  $\pwc \varphi\tau, \, \pwl \varphi\tau$.
     \par
      We pass to the limit of the first integral term on the left-hand side of
        \eqref{entropy-ineq-discr} relying on convergence \eqref{cvT2} for $\log(\upwc\teta{\tau_k})$.
For the second integral term, 
we  prove that 
\begin{equation}
\label{weak-nabla-logteta}
\condu (\pwc\teta{\tau_k}) \nabla \log(\pwc\teta{\tau_k}) \weakto 
\condu(\teta) \nabla \log(\teta) \qquad \text{ in } 
  L^{1+\bar\delta}(Q;\R^d)   \text{ with   $ \bar\delta = \frac{\alpha}\mu $ 
and  \MMM $\alpha \in [0{\vee}(2{-}\mu), 1)$, \EEE}  \end{equation}
 by repeating the very same arguments from the proofs
of \cite[Thm.\ 1]{Rocca-Rossi} and \cite[Thm.\ 1]{Rossi2016}, to which we refer the reader for all details. 
\par
To take the limit in the right-hand side terms of
 \eqref{entropy-ineq-discr}, for the first two integrals we use
 the pointwise
  convergence \eqref{cvT4}  at almost all $t\in (0,T)$ and almost all $s\in (0,t)$,  combined with \eqref{convergences-test-interpolants} for 
  $(\pwc \varphi{\tau_k})$.
   A lower semicontinuity argument also based on the Ioffe theorem \cite{Ioff77LSIF} and on convergences \eqref{convergences-test-interpolants},   \eqref{cvT2},   and \eqref{cvT5} gives that
\[
\begin{aligned} 
\limsup_{k\to\infty} \left(  -  \int_{\pwc{\mathsf{t}}{\tau_k}(s)}^{\pwc{\mathsf{t}}{\tau_k}(t)}  \int_\Omega \condu(\pwc \teta{\tau_k}) \frac{\pwc\varphi{\tau_k}}{\pwc \teta{\tau_k}}
\nabla \log(\pwc \teta{\tau_k})  \nabla \pwc \teta{\tau_k} \dd x \dd r \right) 
 & = - \liminf_{k\to\infty} \int_{\pwc{\mathsf{t}}{\tau_k}(s)}^{\pwc{\mathsf{t}}{\tau_k}(t)}  \int_\Omega \condu(\pwc \teta{\tau_k}) \pwc\varphi{\tau_k} | \nabla  \log(\pwc \teta{\tau_k}) |^2 \dd x \dd r
\\
  & \leq  -  \int_s^t \int_\Omega \condu(\teta) \varphi
|\nabla \log(\teta)|^2\dd x \dd r.
\end{aligned}
\] 
which allows us to deal with  the third integral term on the r.h.s.\ of   \eqref{entropy-ineq-discr}. 
For the limit passage in the  fourth and fifth integral terms on the r.h.s.\ of 
 \eqref{entropy-ineq-discr}, we preliminarily observe that  
 \begin{equation}
 \label{strong-1-teta}
 \frac1{\pwc\teta{\tau_k}} \to \frac1\teta \qquad \text{ in } L^p(Q) \quad \text{for all } 1 \leq p<\infty.
\end{equation}
 This follows from the pointwise convergence $\tfrac1{\pwc\teta{\tau_k}} \to \tfrac1\teta$ a.e.\ in $Q$,
 combined with the 
 Dominated Convergence Theorem, since $\left| \tfrac1{\pwc\teta{\tau_k}} \right| \leq \tfrac1{\bar\teta}$ 
 by \eqref{strong-strict-pos}. 
Furthermore, since $\nabla \left( \tfrac1{\pwc\teta{\tau_k}} \right) =
 \tfrac{\nabla  \pwc\teta{\tau_k}}{|\pwc\teta{\tau_k}|^2}$, combining \eqref{strong-strict-pos} with estimate 
 \eqref{aprio6-discr} we infer that the sequence  $\left( \tfrac1{\pwc\teta{\tau_k}} \right)_k$ is  bounded in $L^2(0,T;H^1(\Omega))$. 
 All in all, we have 
  \begin{equation}
 \label{weak-1-teta}
  \frac1{\pwc\teta{\tau_k}} \weakto \frac1\teta \qquad \text{ in } L^2(0,T;H^1(\Omega)). 
\end{equation}
 Therefore, 
we can now  take the $\limsup_{k\to \infty}$ 
of the fourth integral combining convergences
   \eqref{converg-interp-g} for $(\pwc G{\tau_k})_k$, \eqref{convergences-test-interpolants}
   for $(\pwc \varphi{\tau_k})$, and 
    \eqref{cvE2}, \eqref{cvP2}, \eqref{cvZ2}, \eqref{cvZ3}, \eqref{cvZ4}, \eqref{cvT5}, \eqref{cv_CC_DD}, 
 \eqref{strong-1-teta}, and again resorting to the Ioffe theorem. Finally, 
 the limit passage in the fifth integral term ensues from 
 \eqref{converg-interp-h} for $(\pwc h{\tau_k})_k$,  \eqref{convergences-test-interpolants}, and  
 \eqref{weak-1-teta}.  We thus conclude  the validity \eqref{entropy-ineq}, tested by 
 functions  $\varphi \in \rmC^0 ([0,T]; W^{1,\infty}(\Omega)) \cap H^1(0,T; L^{6/5}(\Omega))$, on the interval $(s,t)$ for almost 
 all $t\in (0,T)$ and almost all $s\in (0,t)$.
 \par
 With the very same argument as in the proof of \cite[Thm.\ 1]{Rossi2016},
  we establish the summability properties  \eqref{further-logteta} for  $\condu(\teta) \nabla \log(\teta)$.
  In view of  \eqref{further-logteta}, the entropy inequality \eqref{entropy-ineq}
   makes sense for all positive test functions
   $\varphi $ in $H^1(0,T; L^{6/5}(\Omega)) \cup L^\infty(0,T;W^{1,d+\epsilon}(\Omega))$ with $\epsilon>0$.
    Therefore, with a density argument we conclude it for this larger test space.
 \par\noindent 
\emph{Step $3$:  preliminary considerations for the damage and  plastic flow rules.}
  Convergences \eqref{cvP2}--\eqref{cvP3}, \eqref{cvZ}, and \eqref{cvS} ensure that the functions $(\teta, e,p,\zeta)$ fulfill 
\begin{equation}
\label{prelim-flow-rule}
\zeta +\dot p \ni \sigma_\dev \qquad \aein\, Q.
\end{equation}
To conclude  the plastic flow rule \eqref{plastic-flow-ptw}, it thus  remains to show that 
\begin{equation}
\label{identification-zeta}
\zeta \in \partial_{\dot p} \dipx{x}{z}{\teta}{ \dot p} \qquad \text{ a.e.\ in $Q$,}
\end{equation}
which is equivalent (cf.\ the considerations at the beginning of Sec.\ \ref{ss:2.3}),  to
\[
\begin{cases}
 \zeta : \eta \leq  \dipx{x}{z}{\teta}{ \eta}  \quad \text{for all } \eta \in \mt_\dev^{d\times d}
\\
 \zeta : \dot{p}  \geq \dipx{x}{z}{\teta}{ \dot p}
\end{cases}
 \qquad \text{ a.e.\ in $Q$,}
\]
by the  
$1$-homogeneity of the functional $\dot{p} \mapsto \dipx{x}{z}{\teta}{ \dot p}$. In fact, we will prove the (equivalent) relations
\begin{subequations}
\label{def-subdiff}
\begin{align}
& 
\label{def-subdiff-a}
\iint_Q \zeta{:} \eta\, \dd x \dd t \leq \int_0^T \Dip{z(t)}{\teta(t)}{ \eta(t)} \dd t \quad \text{for all } \eta \in L^2(Q;\mt_{\dev}^{d\times d}), 
\\
& 
\label{def-subdiff-b}
   \iint_Q \zeta{:} \dot{p} \,\dd x \dd t  \geq \int_0^T \Dip{z(t)}{\teta(t)}{ \dot{p}(t)} \dd t,
   \end{align}
\end{subequations}
featuring the plastic dissipation potential $\Dipname$ from \eqref{plastic-dissipation-functional}. 
With this aim, we will 
pass to the limit  
in the inequalities  satisfied at level $k$ using  the discrete subdifferential inclusion \eqref{eq-p-interp}, namely
\begin{subequations}
\label{def-subdiff-k}
\begin{align}
\label{def-subdiff-k-1}
&
\iint_Q  \pwc\zeta{\tau_k} {:} \eta \dd x \dd t \leq \int_0^T   \Dip{\pwc z{\tau_k}(t)}{\upwc\teta{\tau_k}(t)}{  \eta(t)} \dd t && \text{for all } \eta \in L^2(Q;\mt_{\dev}^{d\times d}), \\  
\label{def-subdiff-k-2}
&
\iint_Q  \pwc\zeta{\tau_k}  {:}  \pwl { \dot p}{\tau_k}  \dd x \dd t  \geq \int_0^T   \Dip{\pwc z{\tau_k}(t)}{\upwc\teta{\tau_k}(t)}{  \pwl {\dot{p}}{\tau_k}(t)} \dd t.  &&
\end{align}
\end{subequations}
In order to take the limit of \eqref{def-subdiff-k-1},
we use conditions \eqref{hypR} on the dissipation metric $\mathrm{H}$.    The strong convergences
\eqref{cvZ4} for $\pwc z{\tau_k}$  and 
 \eqref{cvT5} for  $\upwc\teta{\tau_k}$, combined with the  continuity property \eqref{hypR-cont},  ensure that 
 for every test function $\eta \in L^2(Q;\mt_\sym^{d\times d})$ there holds
 $\dip{\pwc z{\tau_k}}{\upwc \teta{\tau_k}}{\eta} \to \dip z\teta\eta$ almost everywhere in $Q$. Using that 
$\dip{\pwc z{\tau_k}}{\upwc \teta{\tau_k}}{\eta} \leq C_R |\eta|$ a.e.\ in $Q$ thanks to \eqref{linear-growth}, we conclude
via the Dominated Convergence Theorem that 
\begin{equation}
\label{dominated-quoted-later}
\lim_{k\to\infty} \iint_0^T  \Dip{\pwc z{\tau_k}(t)}{\upwc\teta{\tau_k}(t)}{  \eta(t)} \dd t =  
\int_0^T \Dip{z(t)}{\teta(t)}{  \eta(t)} \dd t \quad \text{ for every } \eta \in L^2(Q;\mt_\sym^{d\times d})\,.
\end{equation}
The limit passage on the left-hand side of \eqref{def-subdiff-k-1} is guaranteed by 
 convergence \eqref{cvZ} for $(\pwc\zeta{\tau_k})_k$. Hence, \eqref{def-subdiff-a} follows. 
 As for \eqref{def-subdiff-k-2}, 
   we use  \eqref{hypR-lsc}, the convexity of the map $\dot p \mapsto \dip{z}{\teta}{ \dot p}$
   and convergences \eqref{cvZ4}, 
     \eqref{cvP2}, and \eqref{cvT5},  to conclude via the Ioffe theorem \cite{Ioff77LSIF} that 
\begin{equation}\label{lscR}
\liminf_{k\to \infty} \int_0^T \Dip{\pwc z{\tau_k}(t)}{\upwc \teta{\tau_k}(t)}{ \pwl {\dot p}{\tau_k}(t)} \dd t \geq \int_0^T \Dip{z(t)}{\teta(t)}{ \dot{p}(t)} \dd t\,.
\end{equation}
It thus remains to prove that 
\begin{equation} \label{limsup-cond-p}
\limsup_{k\to \infty}  \iint_Q   \pwc\zeta{\tau_k} {:} \pwl { \dot p}{\tau_k} \dd x \dd t
 \leq  \iint_Q   \zeta {:} { \dot p} \dd x \dd t.
\end{equation}
\par
Analogously, in order to  pass to the limit in the discrete damage flow rule  and establish 
the subdifferential inclusion \eqref{dam-Hs-star-intro}, we need to identify the weak limit 
$\omega$ of $(\pwc \omega{\tau_k})_k$ as an element of  the subdifferential $\subd \Did{\dot{z}} \subset \spz^*$. 
We will also need to pass to the limit in the quadratic term $ -\tfrac12 \bbC'(\pwc z{\tau_k}) \upwc e{\tau_k} :  \upwc e{\tau_k} $ on the right-hand side
 of \eqref{eq-z-interp}, which will require establishing a suitable strong convergence for $(
 \upwc e{\tau_k} )_k$. 
 \par
 All of  these issues  will be addressed in the following steps: preliminarily, in Step $4$  we will pass to the limit in the discrete mechanical energy inequality and obtain the  upper mechanical energy estimate \eqref{uee-mech}. 
 Secondly, in Step $5$ we will prove the one-sided damage variational inequality \eqref{1-sided-intro}. 
  Finally, in Step $6$
 we will combine \eqref{1-sided-intro} with the mechanical energy inequality to establish 
 the mechanical energy balance \eqref{mech-enbal} and, moreover, the desired 
  \eqref{limsup-cond-p} 
   together with further convergence properties. 
  Hence, in Step $6$ and $7$ we will conclude the validity of the plastic and damage flow rules. 
 \par
 \noindent 
 \emph{Step $4$: ad the mechanical energy inequality.}
 We pass to the limit in  the discrete mechanical energy inequality \eqref{mech-ineq-discr}, written on the interval $(0,t)$ with
 $t \in (0,T]$.  As for the left-hand side, we use the pointwise weak convergence \eqref{cvU3-bis} for 
 $(\pwwll{u}{\tau_k}(\pwc {\mathsf{t}}{\tau_k}(\cdot))_k$ to take the $\liminf$ of   the first integral.
In view of convergences \eqref{cvE2}
 for $(\pwl {\dot e}{\tau_k})_k$
 and \eqref{cv_CC_DD}  for $(\bbD(\pwc z{\tau_k}))_k$ we also have 
 \[
 \liminf_{k\to\infty} \int_{0}^{\pwc{\mathsf{t}}{\tau}(t)} \int_\Omega
 \bbD(\pwc z{\tau_k}(r)) \pwl{\dot e}{\tau_k}(r) {:} \pwl{\dot e}{\tau_k}(r)  \dd x \dd r \geq 
 \int_0^t \int_\Omega \bbD(z(r)) \dot{e}(r) {:} \dot{e}(r) \dd x \dd r\,.
 \]
 We pass to the limit in the other dissipative contributions $\iint \did{\pwl{\dot z}{\tau_k}} \dd x \dd r, \ldots ,
 \iint |\pwl{\dot p}{\tau_k}|^2 \dd x \dd r$ invoking convergences \eqref{cvZ2}, \eqref{cvZ3}, \eqref{cvP2}. We also resort
   to the previously established  $\liminf$-inequality \eqref{lscR}.
  We use that 
  \begin{equation}
  \label{stability-e}
  \| \pwc e{\tau_k} {-}  \pwl e{\tau_k}\|_{L^\infty(0,T; L^2(\Omega;\mt_\sym^{d\times d})}\leq 
  \tau_k^{1/2} \|  \pwl {\dot e}{\tau_k}\|_{L^2(0,T; L^2(\Omega;\mt_\sym^{d\times d})} \leq S \tau_k^{1/2}
  \end{equation}
  to deduce from \eqref{cvE3-bis} that 
  \begin{equation}
  \label{pointwise-weak-e}
   \pwc e{\tau_k}(t) \weakto e(t) \quad \text{ in $L^2(\Omega;\mt_\sym^{d\times d})$ for every $t \in [0,T]$}.
  \end{equation}
  In the same way, we prove that 
  \begin{equation}
  \label{pointwise-strong-z}
   \pwc z{\tau_k}(t) \weakto z(t)  \quad \text{ in $\spz$,} \qquad 
  \pwc z{\tau_k}(t) \to z(t) \quad \text{ in $\rmC^0(\overline\Omega)$ for every $t \in [0,T]$.}
  \end{equation}  
  Therefore, again via the Ioffe theorem  we have that 
  $\liminf_{k\to\infty} \calQ( \pwc e{\tau_k}(t),  \pwc z{\tau_k}(t)) \geq \calQ(e(t),z(t))$;
  by the lower semicontinuity of the functional $\calG$ we also infer
  $\liminf_{k\to\infty} \calG(\pwc z{\tau_k}(t)) \geq \calG(z(t))$
   \MMM (where the functionals $\calQ$ and $\calG$ are from \eqref{stored-energy}). \EEE
  \par
  In order to pass to the limit on the right-hand side of  \eqref{mech-ineq-discr}, 
 we use convergences \eqref{approx-e_0}     to deal with the terms on 
  the left-hand side involving the  approximate initial data. 
  Convergences \eqref{converg-interp-L}, \eqref{converg-interp-w},  and \eqref{cvU1} allow us to take the limit
  of the fifth integral term on the right-hand side. As for the sixth and seventh ones, we combine the 
  strong convergences \eqref{cvT5} and
   \eqref{cv_CC_DD}  with the weak ones \eqref{cvE2}, \eqref{cvZ2}. 
   We handle the limit passage in the eighth, ninth, and tenth terms 
   by using convergences \eqref{cvU1} and \eqref{cvU3-bis} for $(\pwwll { u}{\tau_k}(\pwc {\mathsf{t}}{\tau_k}(\cdot))_k$,
   and  \eqref{converg-interp-w} for $(\pwl {\dot w}{\tau_k})_k$ and  $(\partial_t\pwwll {w}{\tau_k})_k$.
   Finally, the limit passage in the eleventh term results from  \eqref{converg-interp-w}
    combined with the stress convergence \eqref{cvS} and with \eqref{cvE3}.
    In fact, we use that 
    \[
     \int_{0}^{\pwc{\mathsf{t}}{\tau}(t)} \int_\Omega  \tau |\pwc e{\tau_k}|^{\gamma-2} \pwc e{\tau_k} : \sig{\pwl {\dot w}{\tau_k}} \dd x \dd r
     = \int_{0}^{\pwc{\mathsf{t}}{\tau}(t)} \int_\Omega  \tau^{1-\alpha_w} |\pwc e{\tau_k}|^{\gamma-2} \pwc e{\tau_k} : \tau^{\alpha_w}\sig{\pwl {\dot w}{\tau_k}} \dd x \dd r  \to 0
     \] 
     as $k\to\infty$ thanks to \eqref{converg-interp-w} and  \eqref{cvE3}.
    Therefore,
    \begin{equation}
    \label{houston}
     \lim_{k\to\infty}   \int_{0}^{\pwc{\mathsf{t}}{\tau}(t)} \int_\Omega 
     \pwc \sigma{\tau_k} {:} \sig{\pwl {\dot w}{\tau_k}} \dd x \dd t 
     = \int_0^t \int_\Omega \sigma {:} \sig{\dot w} \dd x \dd t\,.
    \end{equation}
    \par
    We have thus established the mechanical energy inequality  \eqref{uee-mech} on the interval $(0,t)$, for every $t\in (0,T]$.
         In fact, we have shown that 
    \begin{equation}
    \label{mech-chain}
    \begin{aligned}
    \text{l.h.s.\ of \eqref{uee-mech} on $(0,t)$} &  \leq 
    \liminf_{k\to\infty}  \text{l.h.s.\ of \eqref{mech-ineq-discr} on $(0,t)$} 
    \leq 
     \limsup_{k\to\infty}  \text{l.h.s.\ of \eqref{mech-ineq-discr} on $(0,t)$} 
     \\
     & \quad 
    \leq  \lim_{k\to\infty}  \text{r.h.s.\ of \eqref{mech-ineq-discr} on $(0,t)$} 
    =  \text{r.h.s.\ of \eqref{uee-mech} on $(0,t)$}\,.
    \end{aligned}
    \end{equation}
    Let us mention that the very same convergence arguments as in the above lines also give the upper total energy inequality on 
     $(0,t)$. 
  \par
  \noindent
  \emph{Step $5$: ad the one-sided variational inequality \eqref{1-sided-intro}.}
  Using the $1$-homogeneity of $\didname$ (cf.\ again Sec.\ \ref{ss:2.2}),  we reformulate the discrete damage flow rule
  as the system
  \begin{subequations}
  \label{Karush-Kuhn-Tucker}
  \begin{align}
  &
  \label{KKT1}
\Did{\zeta}  +  \nu \ass( \pwl{\dot z}{\tau},\zeta) {+} \ass(\pwc z\tau,\zeta) + \int_\Omega 
 \left(
 \pwl{\dot z}{\tau} {+}\beta'(\pwc z\tau)
  {-} \lambda_W \upwc z\tau   \right)\zeta \dd x \geq \int_\Omega 
 \left( {-}\tfrac12 \bbC'(\pwc z\tau) \upwc e\tau{:} \upwc e\tau {+} \upwc \teta\tau \right) \zeta \dd x \quad 
 \text{for all } \zeta \in H_-^{\mathrm{s}}(\Omega),
 \\
  &
  \label{KKT2} 
  \Did{\pwl {\dot z}{\tau}}   +
  \nu \ass(\pwl{\dot z}{\tau}, \pwl{\dot z}{\tau}) +\ass (\pwc z{\tau}, \pwl{\dot z}{\tau}) 
  +\int_\Omega \left( |\pwl{\dot z}{\tau}|^2
  {+}\beta'(\pwc z\tau)
 \pwl{\dot z}\tau {-} \lambda_W \upwc z\tau \pwl{\dot z}{\tau} \right) \dd x 
 \leq 
  \int_\Omega 
 \left( {-}\tfrac12 \bbC'(\pwc z\tau) \upwc e\tau{:} \upwc e\tau {+} \upwc \teta\tau \right) \pwl{\dot z}\tau \dd x\,.
  \end{align}
  \end{subequations}
  We now pass to the limit in \eqref{KKT1}, integrated along the interval $(0,t)$, using convergences
   \eqref{cvZ1}--\eqref{cvZ4}, as well as \eqref{cvT5}. 
   Let us only comment on the fact that, since $\beta \in \rmC^2(\R^+)$ and 
   the functions $(\pwc z{\tau_k})_k$, taking  values  in  the interval $[\zeta_*,1]$ by \eqref{discr-feasibility}, 
   converge to $z$ uniformly in $Q$, there holds
   \begin{equation}
   \label{unif-beta-conv}
   \beta'(\pwc z{\tau_k}) \to \beta'(z) \qquad \text{uniformly in }Q\,.
   \end{equation}
   Moreover, by the Ioffe Theorem  (recall that 
   $\zeta \leq 0$ in $\Omega$ and the positivity property \eqref{posbbC} of $\bbC'$),
  we have 
  \[
  \liminf_{k\to\infty}
   \int_0^t \int_\Omega  {-}\tfrac12 \bbC'(\pwc z{\tau_k}) \upwc e{\tau_k}{:} \upwc e{\tau_k} \zeta \dd x  \dd r
   \geq \int_0^t  \int_\Omega {-}\tfrac12 \bbC'(z) e{:} e \zeta \dd x \dd r\,.
  \]
  We have thus established
  \[
 \int_0^t  
 \left(
 \Did{\zeta}  {+} \nu \ass( \dot z,\zeta) {+} \ass(z,\zeta)   \right) \dd  r + \int_0^t  \int_\Omega 
 \left(
 \dot z   {+}\beta'(z)
  {-} \lambda_W z  \right)   \zeta \dd x  \dd r \geq \int_0^t  \int_\Omega 
 \left( {-}\tfrac12 \bbC'(z) e{:}e {+} \teta \right) \zeta \dd x \dd r  \quad 
  \]
for all $  \zeta \in H_-^{\mathrm{s}}(\Omega).$ 
 A localization argument yields the one-sided variational inequality \eqref{1-sided-intro}. 
 \par\noindent
 \emph{Step $6$: enhanced convergence properties and plastic flow rule.}   
 Taking the test function $\zeta = \dot{z}$ in  \eqref{1-sided-intro}, which is admissible since $\dot{z}(t) \in 
 H_-^{\mathrm{s}}(\Omega)$ for almost all $t \in (0,T)$, and integrating in time, yields 
 the converse of inequality \eqref{en-diss-ineq},
 namely
   \begin{equation}
 \label{lee-z}
 \int_0^t \left(  \Did{\dot{z}} \dd r {+}  \int_\Omega |\dot{z}|^2 \dd x \dd r
   {+} \nu  \ass (\dot{z}, \dot{z})  \right) \dd r 
   + \calG(z(t)) + \int_\Omega  \tfrac12 \bbC'(z) e {:}e \dot{z} \dd x \dd r   \geq  \calG(z(0))  + 
   \int_0^t \int_\Omega   \teta  \dot{z}  \dd x \dd r\,,
 \end{equation}
 where we have used the chain rule identity
 \[
 \int_0^t \left( \ass (z,\dot{z}) {+} \int_\Omega W'(z) \dot{z} \dd x \right) \dd r = \calG(z(t))  - 
 \calG(z(0))\,. 
 \]
 We now consider 
 \eqref{def-subdiff-a}, with the test function $\eta = \chi_{(0,t)} \dot{p}$ (and $\chi_{(0,t)}$ the characteristic
 function of $(0,t)$), 
 and deduce that 
    \begin{equation}
 \label{lee-p}
 \begin{aligned}
 \int_0^t \Dip{z(r)}{\teta(r)}{ \dot{p}(r)} \dd r   & \geq  
 \int_0^t \int_\Omega \zeta{:} \dot{p} \dd x \dd r
 \\ & 
  =  \int_0^t \int_\Omega \left( \bbD(z) \dot{e} {+} \bbC(z) e {-}\teta \bbC(z) \bbE \right) {:} \dot{p} \dd x \dd r 
  -\int_0^t \int_\Omega |\dot{p}|^2 \dd x \dd r \,.
  \end{aligned}
    \end{equation}
 Finally, we test the momentum balance by $\dot{u}-\dot{w}$, integrate 
 on $(0,t)$,  and add the resulting relation with \eqref{lee-z} and \eqref{lee-p}. We  observe that 
 some terms cancel out and repeat the very  same calculations as those leading to \eqref{mech-enbal}. 
 We thus conclude that the converse of the  mechanical energy inequality \eqref{uee-mech} holds.
 This establishes  the mechanical energy \emph{balance}  \eqref{mech-enbal} on the interval $(0,t)$. 
 \par
 Furthermore, we continue the chain of inequalities \eqref{mech-chain}  and conclude that 
\[
    \begin{aligned}
    \text{l.h.s.\ of \eqref{mech-enbal} on $(0,t)$} &  \leq 
    \liminf_{k\to\infty}  \text{l.h.s.\ of \eqref{mech-ineq-discr} on $(0,t)$} 
    \leq 
     \limsup_{k\to\infty}  \text{l.h.s.\ of \eqref{mech-ineq-discr} on $(0,t)$} 
     \\ & 
    \leq  \lim_{k\to\infty}  \text{r.h.s.\ of \eqref{mech-ineq-discr} on $(0,t)$} 
    =  \text{r.h.s.\ of \eqref{mech-enbal} on $(0,t)$} 
    \stackrel{(!)}{=}  \text{l.h.s.\ of \eqref{mech-enbal} on $(0,t)$}\,.
    \end{aligned}
    \]
    Therefore, all inequalities hold as equalities, the lower and upper limits coincide. Moreover,
     taking into account the $\liminf$-inequalities previously observed in Step $3$, 
     with a standard argument  
    we conclude  that \emph{each} of the terms on the left-hand side  of 
\eqref{mech-ineq-discr} does converge to its analogue  on the left-hand side of \eqref{mech-enbal}.
Thus, we have  in particular
\begin{subequations}
\label{we-do-have-convergences}
\begin{align}
&
\label{wdhc-1}
\lim_{k\to\infty}\int_\Omega |\pwl{\dot{u}}{\tau_k}(t)|^2 \dd x = \int_\Omega |\dot{u}(t)|^2 \dd x\,,
\\
&
\label{wdhc-2}
 \lim_{k\to\infty} \int_{0}^{\pwc{\mathsf{t}}{\tau}(t)} \int_\Omega
 \bbD(\pwc z{\tau_k}(r)) \pwl{\dot e}{\tau_k}(r) {:} \pwl{\dot e}{\tau_k}(r)  \dd x \dd r =
 \int_0^t \int_\Omega \bbD(z(r)) \dot{e}(r) {:} \dot{e}(r) \dd x \dd r\,,
 \\
 &
\label{wdhc-3}
 \lim_{k\to\infty} \int_{0}^{\pwc{\mathsf{t}}{\tau}(t)}  \Did{\pwl{\dot{z}}{\tau_k}(r)} \dd r =
  \int_0^t \Did{\dot{z}(r)} \dd x\,,
  \\
 &
\label{wdhc-4}
 \lim_{k\to\infty} \int_{0}^{\pwc{\mathsf{t}}{\tau}(t)} \int_\Omega  |\pwl{\dot{z}}{\tau_k}(r)|^2  \dd x \dd r =
 \int_0^t \int_\Omega |\dot{z}(r)|^2 \dd x \dd r\,,
\\
 &
\label{wdhc-5}
 \lim_{k\to\infty} \int_{0}^{\pwc{\mathsf{t}}{\tau}(t)}  \bar{\nu} \ass (\pwl{\dot{z}}{\tau_k}(r), 
 \pwl{\dot{z}}{\tau_k}(r)) \dd r = \int_0^t \bar{\nu} \ass(\dot{z}(r), \dot{z}(r)) \dd r,
 \\
 &
 \label{wdhc-6}
 \lim_{k\to\infty} \int_{0}^{\pwc{\mathsf{t}}{\tau}(t)}  \Dip{\pwc z{\tau_k}(r)}{\upwc\teta{\tau_k}(r)}{  
 \pwl {\dot p}{\tau_k}(r)} \dd r=  
\int_0^t \Dip{z(r)}{\teta(r)}{  \dot{p}(r)} \dd r\,,
\\
 &
\label{wdhc-7}
 \lim_{k\to\infty} \int_{0}^{\pwc{\mathsf{t}}{\tau}(t)} \int_\Omega  |\pwl{\dot{p}}{\tau_k}(r)|^2  \dd x \dd r =
 \int_0^t \int_\Omega |\dot{p}(r)|^2 \dd x \dd r\,.
 \end{align}
\end{subequations}
In particular,  from \eqref{wdhc-2}, repeating the very same arguments from the proof  of \cite[Thm.\ 5.3]{LRTT}, which substantially rely on the 
uniform positive definiteness of the tensor $\bbD$, we infer that 
\begin{equation}
\label{strong-cv-dote}
\pwl{\dot e}{\tau_k} \to \dot{e} \qquad \text{ in } L^2(Q; \mt_\sym^{d\times d})\,.
\end{equation}
Then, $\pwl e{\tau_k} \to e $ in $\rmC^0([0,T]; L^2(\Omega; \mt_\sym^{d\times d}))$. 
In view of \eqref{stability-e} we then infer that 
\begin{equation}
\label{strong-conv-e-needed}
\pwc e{\tau_k}, \,  \upwc e{\tau_k} \to e \quad \text{ in } L^\infty (0,T;L^2(\Omega; \mt_\sym^{d\times d}))\,.
\end{equation}
Let us also record that  \eqref{wdhc-4} \&  \eqref{wdhc-5}
yield
\begin{equation}
\label{strong-cv-dotz}
\pwl{\dot z}{\tau_k} \to \dot{z}  \text{ in } L^2(Q) \text{ and }
\pwl{\dot z}{\tau_k} \to \dot{z} \text{ in } L^2(0,T; \spz) \text{ if } \nu >0,
\end{equation}
which gives, by dominated convergence, 
\begin{equation}
\label{strong-cv-R}
\did{\pwl{\dot z}{\tau_k}} \to \did{\dot{z}} \text{ in } L^1(Q).
\end{equation}
Finally, \eqref{wdhc-7} gives
\begin{equation}
\label{strong-cv-dotp}
\pwl{\dot p}{\tau_k} \to \dot{p}  \text{ in } L^2(Q;\mt_\dev^{d\times d})\,,
\end{equation}
from which we deduce, taking into account convergences \eqref{cvZ4}, \eqref{cvT5}, the continuity of $\dipname$, and repeating the same arguments leading to \eqref{dominated-quoted-later}, that 
\begin{equation}
\label{strong-cv-H}
\dip{\pwc z{\tau_k}}{\upwc \teta{\tau_k}}{\pwl{\dot p}{\tau_k} } \to \dip z\teta{\dot p} \text{ in } L^1(Q).
\end{equation}
Furthermore, in view of the strong convergence \eqref{strong-cv-dotp}, the $\limsup$-inequality \eqref{limsup-cond-p} immediately follows.  This establishes the plastic flow rule.
\par
\noindent
\emph{Step $7$: ad the damage  flow rule.}
It follows from \eqref{elast-visc-tensors-2} and \eqref{cvZ4} that 
$\bbC'(\pwc z{\tau_k}) \weaksto \bbC'(z)$ in $L^\infty(Q;\mt_\sym^{d\times d})$.
Combining this with 
\eqref{strong-conv-e-needed} we find that 
\[
\bbC'(\pwc z{\tau_k}) \upwc e{\tau_k} : \upwc e{\tau_k}  \weakto \bbC'(z) e:e \qquad \text{in }
\MMM L^1(Q)\,. \EEE
\]
\MMM Combining this with convergences
 \eqref{cvZ1}--\eqref{cvZ5},
 \eqref{cvT5}, and \eqref{unif-beta-conv},  we are thus able to pass to the limit in an integrated version of the 
 discrete damage flow rule \eqref{eq-z-interp}, with test functions in $L^\infty(0,T;\spz)$ (which embeds continuously into $L^\infty(Q)$). 
 With a localization argument we then deduce that \EEE
the quadruple $(e,z,\teta,\omega)$ complies with 
 \begin{equation}
 \label{almost-dam}
 \omega +\dot z+ \nu \As(\dot z) + \As(z) + W'(z) = -\frac12 \bbC'(z) e:e + \teta \qquad \text{ in } \spz^* \quad
 \aein\, (0,T).
 \end{equation}
 In order to conclude the damage flow rule \eqref{dam-Hs-star-intro}, it remains to show
  that $\omega(t) \in \subd \Did{z(t)} $ for almost all $t\in (0,T)$, with $\calR$ from \eqref{abstract-dissip-potential}.
  In fact, this is equivalent to showing that 
 \begin{equation}
 \label{integral-identification}
 \omega \in \subd \overline{\calR}(\dot{z}) \ 
   \text{ with }  \  \overline{\calR}:
  L^2(0,T; \spz) \to [0,+\infty]  \ \text{ given by } \ 
 \overline{\calR}(v): =\int_0^T \Did{v(t)} \dd t
 \end{equation}
and $\subd \overline{\calR}:    L^2(0,T; \spz) \rightrightarrows   L^2(0,T; \spz^*)$ its subdifferential 
in the sense of convex analysis. Now, \eqref{integral-identification} directly  follows by 
passing to the limit in
its discrete version $\pwc\omega{\tau_k} \in \subd \overline{\calR}(\pwl{\dot{z}}{\tau_k})$, 
combining the 
weak convergence
\eqref{cvZ5} for $\pwc \omega{\tau_k}$
with the strong one 
\eqref{strong-cv-dotz} for $\pwl{\dot z}{\tau_k}$ and using
 the strong-weak closedness of the graph of $\partial\overline{\calR}$. 
This concludes the proof of the damage flow rule \eqref{dam-Hs-star-intro}. 
 \par
\noindent
\emph{Step $8$: ad the temperature equation.}
  We  pass to the limit in the approximate temperature equation \eqref{eq-teta-interp}, 
with the test functions from \eqref{discrete-tests-phi}
approximating  a fixed test function $
\varphi\in \rmC^0([0,T]; W^{1,\infty}(\Omega))\cap
H^1(0,T;L^{6/5}(\Omega))  $. 
Using formula \eqref{discr-by-part},  we integrate by parts in time   the term 
 \[
 \int_0^t \int_\Omega  \pwl{\dot \teta}{\tau_k} \pwc{\varphi}{\tau_k}\dd x \dd s
 = - \int_0^t \int_\Omega \upwc{\teta}{\tau_k} \pwl{\varphi}{\tau_k}\dd x \dd s 
 + \int_\Omega \pwc{\teta}{\tau_k}(t) \pwc{\varphi}{\tau_k}(t)\dd x - 
  \int_\Omega \pwc{\teta}{\tau_k}(0) \pwc{\varphi}{\tau_k}(0)\dd x
  \]
and for taking its limit  we exploit convergences \eqref{convergences-test-interpolants},  as 
well as  \eqref{cvT1} for $\upwc \teta{\tau_k}$   and  the pointwise  \eqref{cvT7} for $\pwc \teta{\tau_k}$. 
  For the limit passage in the term
   $\int_0^t \int_\Omega \condu(\pwc \teta{\tau_k}) \nabla
     \pwc \teta{\tau_k} \nabla \pwc \varphi{\tau_k} \dd x \dd t$, repeating the very same arguments 
     from the proofs of \cite[Thm.\ 2]{Rocca-Rossi} and \cite[Thm.\ 2]{Rossi2016}, \MMM which  in turn rely on the growth \eqref{hyp-K-stronger}
     and on estimates \eqref{est-temp-added?-bis}, \EEE
      we show that  
    \begin{equation}
    \label{specific-tilde-delta}
    \condu(\pwc \teta{\tau_k}) \nabla  \pwc \teta{\tau_k} 
    \weakto \condu(\teta) \nabla \teta \quad \text{ in $L^{1+\tilde\delta}(Q;\R^d)$, with } \quad
    \tilde{\delta} = \frac{2-3\mu+3\alpha}{3(\mu-\alpha+2)} \in \left(0,\frac13 \right)
    \end{equation}
    (cf.\ \eqref{further-k-teta}).
   Finally, observe that $\condu(\teta) \nabla \teta = \nabla (\hat{\condu}(\teta))$
      thanks to \cite{Marcus-Mizel}.  Since $\hat{\condu}(\teta)$
       itself is a function in $L^{1+\tilde\delta}(Q) $ (for $d=3$, this follows from the fact that $\hat{\condu}(\teta) \sim \teta^{\mu+1} \in L^{h/(\mu+1)}(Q)$ for  every $1 \leq h<\frac83$), we conclude \eqref{further-k-teta}.
 The limit passage on the r.h.s.\ of  the discrete heat equation \eqref{eq-teta-interp} results from   \eqref{converg-interp-g}, 
 \eqref{converg-interp-h}, 
 and from  the previously established strong convergences  \eqref{cvT6},  \eqref{cv_CC_DD},  \eqref{strong-cv-dote},   
 \eqref{strong-cv-dotz}--\eqref{strong-cv-H}. 
 All in all, we have established that the  limit quintuple $(u,e,p,z,\teta)$ complies
  with 
   \begin{equation}
\begin{aligned}
\label{eq-teta-interm}   &
\pairing{}{W^{1,\infty}(\Omega)}{\teta(t)}{\varphi(t)}
-\int_0^t\int_\Omega \teta \varphi_t \dd x \dd s     +\int_0^t \int_\Omega \condu(\teta) \nabla \teta\nabla\varphi \dd
x \dd s  
\\ &= \int_\Omega \teta_0 \varphi(0) \dd x  +
\int_0^t\int_\Omega \left(G{+} \bbD(z) \dot{e} {:} \dot{e}  
{-}\teta \bbC(z) \bbE {:}\dot e {+} \did{\dot z} {+} |\dot z|^2 {+} \bar \nu \ass (\dot{z}, \dot{z}) {-}\teta \dot{z} + \dip{z}{\teta}{\dot{p}} {+} |\dot{p}|^2 \right) 
\varphi \dd x \dd s
\\
&
\quad  
  +  \int_0^t \ \int_{\partial\Omega} h \varphi \dd S \dd s\,.
\end{aligned}
\end{equation}
for all   test functions
 $
\varphi\in \rmC^0([0,T]; W^{1,\infty}(\Omega))\cap
H^1(0,T;L^{6/5}(\Omega))  $
  and for all  $ t\in (0,T]. $
\par
Clearly, testing  \eqref{eq-teta-interm}  with functions $\varphi \in W^{1,\infty}(\Omega)$ independent of the time variable,
and repeating the very same arguments from the proof of \cite[Thm.\ 2]{Rossi2016}, we conclude that $\teta $ is in
$W^{1,1}(0,T;W^{1,\infty}(\Omega)^*)$ and that it complies with \eqref{eq-teta}. 
Again arguing as for  \cite[Thm.\ 2]{Rossi2016},
we can enlarge the space of test functions to $W^{1,q_{\tilde \delta}}(\Omega)$, with $q_{\tilde\delta} = 1+\tfrac1{\tilde\delta}$ and $\tilde \delta $ from 
\eqref{specific-tilde-delta}, and improve the regularity of $\teta$ to $W^{1,1}(0,T;W^{1,q_{\tilde \delta}}(\Omega)^*)$. 
We have thus completed the proof of Theorem \ref{mainth:2}. 
\QED
\par
\noindent 
\textbf{\underline{Proof of Theorem  \ref{mainth:1} (Entropic solutions):}} It is immediate to check that 
Steps $0$--$5$ of the proof of Thm.\    \ref{mainth:2} 
do not rely on the more stringent growth condition \eqref{hyp-K-stronger} on $\condu$. Therefore, they carry over
in the setting of  Theorem \ref{mainth:1}. We thus immediately establish the validity 
of the kinematic admissibility,
of the weak momentum balance,  of the one-sided damage variational inequality \eqref{1-sided-intro},  of the entropy inequality, and of the upper mechanical energy estimate \eqref{uee-mech}.
\par
From \eqref{def-subdiff-a} and \eqref{prelim-flow-rule} we also infer the plastic variational inequality \eqref{1-sided-pl}. The upper total energy estimate \eqref{total-uee} 
 follows from 
passing to the limit in \eqref{total-enid-discr} with the very same arguments used for obtaining  the mechanical energy inequality. 
\par
This concludes the proof.
\QED

\begin{remark}
\label{rmk:1-from-2}
\upshape
A few more comments on the proof of Theorem \ref{mainth:1} are in order.
\begin{enumerate}
\item Step $6$ in the proof of Thm.\ \ref{mainth:2} (and, consequently, Steps $7$ \& $8$) does not carry over to the setting of Thm.\ \ref{mainth:1}. Indeed, testing the one-sided variational inequality for damage \eqref{1-sided-intro} by $\dot{z}$ is no longer possible, as, setting $\nu=0$ we have lost the information that $\dot z \in \spz$. Hence, all the $\limsup$-arguments 
from Step $7$, which
strengthened the convergences of the approximate solutions, cannot be repeated. 
\item As already hinted at the end of Sec.\ \ref{ss:2.4}, it would be possible to establish the existence of entropic solutions to the (non-regularized) thermoviscoelastoplastic damage  system by passing to the limit as $\nu\down 0$ in its regularized version, featuring a family $(\condu_\nu)_\nu$ of heat conductivity functions converging to some $\condu$ that only complies with \eqref{hyp-K}. 
In fact,   the a priori estimates \eqref{aprioU1}--\eqref{aprio_Varlog}, performed on the time-discrete version of the regularized system, are inherited by the weak energy solutions: in particular, it is easy to check that the bounded variation estimate \eqref{aprio_Varlog}  is preserved by  lower semicontinuity arguments. Then,
a close perusal of Steps $0$--$5$ reveals that all the arguments performed for the time-discrete to continuous limit  carry over to the limit passage $\nu \downarrow 0$. 
\end{enumerate}
\end{remark}
\section{Proof of Proposition \ref{prop:continuous-dependence}} 
\label{s:6}
\noindent
\MMM Throughout this section we will work under the strongly simplifying condition that $\mathrm{H}$ neither depends on the damage variable nor on the temperature. \EEE 
Let $(u_i,e_i,z_i,p_i)$, $i=1,2$, be two (weak energy) solutions 
to the regularized viscolelastoplastic damage system  with a given temperature profile $\Theta$, supplemented with  initial data
$(u_i^0,e_i^0, z_i^0, p_i^0)$ and external data $(\calL_i,w_i)$, $i=1,2$ (where $\calL_i$ subsume the volume forces  $F_i$ and applied tractions $f_i$). 
 We will use the place-holders
\[
\begin{aligned}
&
\tilde{u}: = u_1-u_2, \quad \tilde{e}: = e_1-e_2, \quad \tilde{z}: = z_1-z_2,  \quad \tilde{p}: = p_1-p_2\,,\quad \tilde{\sigma}: = \sigma_1-\sigma_2\,,
\end{aligned}
\]
where  $\sigma_i: = \bbD(z_i) \dot{e}_1 + \bbC(z_i) e_i -\Theta_i \bbC(z_i) \bbE$ for $i=1,2$. Throughout the proof, we will also often use the short hand $\| \cdot\|_{L^p}$ for  $\| \cdot\|_{L^p(\Omega)}$,  $\| \cdot\|_{L^p(\Omega;\R^d)}$ and so forth.  
\par
We preliminarily observe that, at fixed $t\in (0,T)$ (which we omit) there holds 
\begin{equation}
\label{prelim-sigma}
\begin{aligned}
&
\| \tilde{\sigma}\|_{L^2(\Omega;\mt_\sym^{d\times d})}    \leq \|  \bbD(z_1) \partial_t \tilde{e} \|_{L^2(\Omega;\mt_\sym^{d\times d})}
+ \| \left( \bbD(z_1){-} \bbD(z_2) \right) \dot{e}_2  \|_{L^2(\Omega;\mt_\sym^{d\times d})}
+  \| \bbC(z_1)  \tilde{e} \|_{L^2(\Omega;\mt_\sym^{d\times d})}
\\
&
\quad 
+ \| \left( \bbC(z_1){-} \bbC(z_2) \right) {e}_2  \|_{L^2(\Omega;\mt_\sym^{d\times d})}
+ \| \Theta_1 \left( \bbC(z_1){-} \bbC(z_2)\right) \bbE  \|_{L^2(\Omega;\mt_\sym^{d\times d})} 
+ \| (\Theta_1{-}\Theta_2) \bbC(z_2) \bbE  \|_{L^2(\Omega;\mt_\sym^{d\times d})} 
\\
& \quad  \leq 
M_1\left( 
\| \partial_t \tilde{e} \|_{L^2(\Omega;\mt_\sym^{d\times d})}+ \| \tilde{e} \|_{L^2(\Omega;\mt_\sym^{d\times d})}+
 \|\tilde{z}\|_{L^\infty(\Omega)}+
 \|\Theta_1\|_{L^2(\Omega)} \|\tilde{z}\|_{L^\infty(\Omega)}+
  \|\Theta_1{-}\Theta_2 \|_{L^2(\Omega)} \right)\,.
\end{aligned}
\end{equation}
The last estimate follows from $(2.(\bbC,\bbD,\bbE))$, with a constant $M_1$ only depending   on the norms $\|z_1\|_{L^\infty}$, $\|z_2\|_{L^\infty}$.   In fact, 
$\|   \bbD(z_1)\|_{L^\infty},\, \|   \bbC(z_1)\|_{L^\infty}$ are estimated by a constant  depending on $\|z_1\|_{L^\infty}$. Analogously,  thanks to \eqref{elast-visc-tensors-4}, 
the Lipschitz estimates $\| \bbD(z_1){-} \bbD(z_2)\|_{L^\infty} \leq C \| z_1{-}z_2\|_{L^\infty}$ and 
$\| \bbC(z_1){-} \bbC(z_2)\|_{L^\infty} \leq \tilde{C} \| z_1{-}z_2\|_{L^\infty}$ hold with constants depending on  $\|z_1\|_{L^\infty}$, $\|z_2\|_{L^\infty}$. 
\par
In order to prove 
\eqref{cont-dependence-estimate}, we start by subtracting the weak momentum balance \eqref{w-momentum-balance} for $u_2$ from that for $u_1$, test the resulting relation by $\partial_t \tilde{u} -\partial_t (w_1{-}w_2)$, and integrate on an arbitrary time interval $(0,t)$. Elementary calculations (cf.\ \eqref{intermediate-mech-enbal}) lead to
\begin{equation}
\label{est-tilde-u}
\begin{aligned}
& 
\frac{\rho}2 \int_\Omega |\partial_t\tilde{u}(t)|^2 \dd x + \int_0^t \int_\Omega \tilde{\sigma} : \sig{\partial_t\tilde{u}} \dd x \dd r 
\\
&  = \frac{\rho}2 \int_\Omega |\dot{u}_1^0 {-}  \dot{u}_2^0 |^2 \dd x  + \int_0^t \pairing{}{H_\Dir^1 (\Omega;\R^d)}{\mathcal{L}_1 {-} \mathcal{L}_2 }{\partial_t \tilde{u} {-} \partial_t (w_1{-}w_2)} \dd r  +
\int_0^t \int_\Omega  \tilde{\sigma}  : \sig{\partial_t ( w_1{-}w_2)} \dd x \dd r 
\\
&\quad  +\rho \left( \int_\Omega \partial_t\tilde{u}(t) \partial_t(w_1{-}w_2)(t) \dd x -  \int_\Omega (\dot{u}_1^0 {-}  \dot{u}_2^0)  \partial_t(w_1{-}w_2)(0)
\dd x - \int_0^t \int_\Omega \partial_t\tilde{u}\partial_{tt} (w_1{-}w_2) \dd x \dd r \right)
\\
& \doteq I_1+I_2+I_3+I_4\,.
\end{aligned}
\end{equation}
We estimate
\[
\begin{aligned}
&
\left| I_1 \right|  \leq C \| \dot{u}_1^0 {-}  \dot{u}_2^0\|_{L^2(\Omega;\R^d)}^2,
\\
&
\left| I_2   \right| \leq \frac12 \| \mathcal{L}_1 {-} \mathcal{L}_2  \|_{L^2(0,T;H^1(\Omega;\R^d)^*)}^2 + \frac12 \|w_1 {-}w_2\|_{H^1(0,T;H^1(\Omega;\R^d))}^2,
\\
&
\begin{aligned}
\left| I_3   \right| \leq &  \eta \int_0^t  \left( 
\| \partial_t \tilde{e} \|_{L^2(\Omega;\mt_\sym^{d\times d})}^2{+} \| \tilde{e} \|_{L^2(\Omega;\mt_\sym^{d\times d})}^2{+}
 \|\tilde{z}\|_{L^\infty(\Omega)}^2{+}\|\Theta_1\|_{L^2(\Omega)}^2  \|\tilde{z}\|_{L^\infty(\Omega)}^2 \right)  \dd r 
 + C \| \Theta_1{-} \Theta_2\|_{L^2(Q)}^2
 \\& \quad 
 +C_\eta  \|w_1 {-}w_2\|_{H^1(0,T;H^1(\Omega;\R^d))}^2,
 \end{aligned}
 \\
 &
 \begin{aligned}
\left|  I_4  \right|  \leq C  \| \dot{u}_1^0 {-}  \dot{u}_2^0\|_{L^2(\Omega;\R^d)}^2
 + \frac{\rho}8 \int_\Omega |\partial_t\tilde{u}(t)|^2 \dd x &+ C\|w_1 {-}w_2\|_{W^{1,\infty}(0,T;L^2(\Omega;\R^d))}^2 
  \\ &  +\rho\int_0^t\| \partial_{tt} (w_{1}{-}w_2) \|_{L^2(\Omega;\R^d)}\|  \partial_t\tilde{u}\|_{L^2(\Omega;\R^d)} \dd r,
 \end{aligned}
\end{aligned}
\]
where the bound for $I_3$ follows from \eqref{prelim-sigma}, and the constant $\eta>0$  therein,  on which $C_\eta>0$ depends,  will be chosen suitably. 
As for the second term on the left-hand side of \eqref{est-tilde-u}, 
it follows from the kinematic admissibility condition that  
\[
 \int_0^t \int_\Omega \tilde{\sigma} : \sig{\partial_t\tilde{u}} \dd x \dd r  =  \int_0^t \int_\Omega \tilde{\sigma} : \partial_t \tilde{e} \dd x \dd r  +  \int_0^t  \int_\Omega \tilde{\sigma} : \partial_t \tilde{p} \dd x \dd r\doteq I_5 +I_6\,.
\]
We mention in advance that $I_6$ will cancel out with a term arising from the test of the plastic flow rules. As for $I_5$, we have that 
(cf.\ \eqref{prelim-sigma})
 \[
 \begin{aligned}
 I_5 &  =  \int_0^t \int_\Omega \bbD(z_1) \partial_t \tilde{e} : \partial_t \tilde{e} \dd x \dd r  +   \int_0^t \int_\Omega \left(\bbD(z_1){-} \bbD(z_2) \right) \dot{e}_2 : \partial_t \tilde{e}  \dd x  \dd r
 \\
 & \quad +  \int_0^t \int_\Omega \bbC(z_1)  \tilde{e} : \partial_t \tilde{e} \dd x \dd r +   \int_0^t \int_\Omega \left(\bbC(z_1){-} \bbC(z_2) \right) e_2 : \partial_t \tilde{e}  \dd x    \dd r 
 \\
 & \quad +\int_0^t \int_\Omega \Theta_2 (\bbC(z_2) {-} \bbC(z_1)) \bbE : \partial_t \tilde{e} \dd x  \dd r 
 +\int_0^t \int_\Omega (\Theta_2{-}\Theta_1)  \bbC(z_1) \bbE : \partial_t \tilde{e} \dd x  \dd r 
 \\
 &
  \doteq I_{5,1}+ I_{5,2}+ I_{5,3} +I_{5,4} +I_{5,5} +I_{5,6}\,.
 \end{aligned}
 \]
 By the uniform positive definiteness of $\bbD$ we have 
 \[
 I_{5,1} \geq C_{\bbD}^1 \int_0^t \int_\Omega |\partial_t \tilde{e}|^2 \dd x \dd r,
 \]
 while the other terms, to be moved to the right-hand side of estimate \eqref{est-tilde-u}, can be estimated with analogous arguments as for \eqref{prelim-sigma}. Namely, 
 \[
 \begin{aligned}
 &
 \left| I_{5,2} \right| \leq M_2 \int_0^t  \| \tilde{z}\|_{L^\infty(\Omega)}\| \partial_t \tilde{e}   \|_{L^2(\Omega;\mt_\sym^{d\times d})} \dd r
 \leq \varrho \int_0^t \| \partial_t \tilde{e}   \|_{L^2(\Omega;\mt_\sym^{d\times d})}^2 \dd r + C(\varrho,M_2) \int_0^t  \| \tilde{z}\|_{L^\infty(\Omega)}^2 \dd r,
 \\
 &
 \left| I_{5,3} \right| \leq   M_2  \int_0^t  \| \partial_t \tilde{e}   \|_{L^2(\Omega;\mt_\sym^{d\times d})} \|  \tilde{e}   \|_{L^2(\Omega;\mt_\sym^{d\times d})} \dd r \leq  \varrho \int_0^t \| \partial_t \tilde{e}   \|_{L^2(\Omega;\mt_\sym^{d\times d})}^2 \dd r 
 +C(\varrho,M_2) \int_0^t \|  \tilde{e}   \|_{L^2(\Omega;\mt_\sym^{d\times d})}^2 \dd r,
 \\
 &
 \left| I_{5,4} \right| \leq   M_2 \int_0^t  \| \tilde{z}\|_{L^\infty(\Omega)}\| \partial_t \tilde{e}   \|_{L^2(\Omega;\mt_\sym^{d\times d})} \dd r
 \leq \varrho \int_0^t \| \partial_t \tilde{e}   \|_{L^2(\Omega;\mt_\sym^{d\times d})}^2 \dd r + C(\varrho,M_2) \int_0^t  \| \tilde{z}\|_{L^\infty(\Omega)}^2 \dd r,
 \\
 &
 \begin{aligned}
 \left| I_{5,5} \right|  & \leq  M_2 \int_0^t  \| \tilde{z}\|_{L^\infty(\Omega)} \| \Theta\|_{L^2(\Omega)} \| \partial_t \tilde{e}   \|_{L^2(\Omega;\mt_\sym^{d\times d})} \dd r \\ & \leq   \varrho \int_0^t \| \partial_t \tilde{e}   \|_{L^2(\Omega;\mt_\sym^{d\times d})}^2 \dd r  + 
  C(\varrho,M_2) \int_0^t \| \Theta\|_{L^2(\Omega)}^2  \| \tilde{z}\|_{L^\infty(\Omega)}^2 \dd r,
    \end{aligned}
  \\
  & \begin{aligned}
 \left| I_{5,6} \right| \leq   \varrho \int_0^t \| \partial_t \tilde{e}   \|_{L^2(\Omega;\mt_\sym^{d\times d})}^2 \dd r  + 
  C(\varrho,M_2) \int_0^t \| \Theta_1{-}\Theta_2\|_{L^2(\Omega)}^2 \dd r.
  \end{aligned}
 \end{aligned}
 \]
where  the constant $M_2$ depends on $\|z_1\|_{L^\infty},$ $  \|z_2\|_{L^\infty}, $ and $\|e_2\|_{L^2}$. The positive constant $\varrho$ (note that   $C(\varrho,M_2)$ depends on $\varrho$ and $M_2$), will be specified later. 
 \par
 Next, we subtract the subdifferential inclusion \eqref{no-duality} for $z_2$ from that for $z_1$, test the resulting relation by $\partial_t \tilde{z}$, and integrate on $(0,t)$. We thus obtain
 \begin{equation}
\label{est-tilde-z}
\begin{aligned}
&
  \int_0^t  \int_\Omega |\partial_t \tilde{z}|^2 \dd x  \dd r +\nu \int_0^t \ass( \partial_t \tilde{z}, \partial_t \tilde{z}) \dd r + \frac12 \ass (\tilde{z}(t), \tilde{z}(t)) 
  \\
  & = \frac12 \ass (z_1^0{-}z_2^0, z_1^0{-}z_2^0)  -
  \int_0^t \pairing{}{\spz}{\omega_1{-}\omega_2}{\partial_t \tilde{z}} \dd r  +    \int_0^t  \int_\Omega\left(  W'(z_2) {-}  W'(z_1)\right) \partial_t \tilde{z} \dd x \dd r
  \\
  & \qquad
   +    \int_0^t  \int_\Omega \left( \tfrac12 \bbC'(z_2) |e_2|^2 {-}\tfrac12 \bbC'(z_1) |e_1|^2 \right)  \partial_t \tilde{z} \dd x \dd r 
 \doteq I_7+I_8+I_9+I_{10}\,,
 \end{aligned}
 \end{equation}
 with $\omega_i(t) \in \partial\Did{\dot{z}_i(t)}$ for almost all $t\in (0,T)$.
 Observe that $I_7\leq \| z_1^0{-}z_2^0\|_{\spz}^2$, while $I_8\leq 0$  by monotonicity and 
 \[
 \begin{aligned}
 &
 \left| I_9 \right| \leq M_3 \int_0^t \|\tilde z \|_{L^2(\Omega)} \|\partial_t \tilde z \|_{L^2(\Omega)}  \dd r \leq \frac14 \int_0^t  \|\partial_t \tilde z \|_{L^2(\Omega)}^2 \dd r +   C \int_0^t \|\tilde z \|_{L^2(\Omega)}^2 \dd r\,, 
 \\
 &
 \begin{aligned}
 I_{10}  & =  \int_0^t  \int_\Omega \tfrac12 \bbC'(z_2) (e_2{+}e_1) \tilde{e} \partial_t \tilde{z} \dd x \dd r + \int_0^t \int_\Omega \tfrac12 \left( \bbC'(z_2) {-} \bbC'(z_1) \right) |e_1|^2   \partial_t \tilde{z}  \dd x \dd r 
 \\
 & 
 \leq M_4 \int_0^t \|\tilde{e}\|_{L^2(\Omega;\mt_\sym^{d\times d})} \| \partial_t \tilde{z} \|_{L^\infty(\Omega)}
 \dd r  + M_4 \int_0^t \|\tilde{z}\|_{L^\infty(\Omega)}  \| \partial_t \tilde{z} \|_{L^\infty(\Omega)} \dd r
 \\
 & \leq \lambda \int_0^t  \| \partial_t \tilde{z} \|_{L^\infty(\Omega)}^2 \dd r + C(\lambda,M_4) \int_0^t \|\tilde{e}\|_{L^2(\Omega;\mt_\sym^{d\times d})}^2 \dd r + C(\lambda,M_4) \int_0^t \|\tilde{z}\|_{L^\infty(\Omega)}^2 \dd r.
\end{aligned}
 \end{aligned}
 \]
 The constant $M_3$ depends on  $\|z_1\|_{L^\infty},$ $  \|z_2\|_{L^\infty}, $
 and on the constants $\zeta_*^1,\, \zeta_*^2 >0$ such that for $i=1,2$ there holds $\zeta_*^i\leq z_i (x,t)  \leq 1$  for every $(x,t)\in Q$, observing that the restriction of $W$ to $[\min\{ \zeta_*^1, \zeta_*^2\}, 1]$ is of class $\mathrm{C}^2$. 
 The constant $M_4$ depends on $\|z_1\|_{L^\infty},$ $  \|z_2\|_{L^\infty}, $
  $\|e_1\|_{L^2}$,  and $\|e_2\|_{L^2}$. The positive constant $\lambda$ will be specified later; $C(\lambda,M_4) $ depends on $\lambda$ and $M_4$. 
 \par
 Finally, we subtract the pointwise plastic flow rule \eqref{plastic-flow-ptw} for $p_2$ from that for $p_1$, test the resulting  relation by 
 $\partial_t\tilde{p}$, and integrate in time. This leads to
  \begin{equation}
\label{est-tilde-p}
 \begin{aligned}
  \int_0^t\int_\Omega |\partial_t \tilde{p}|^2 \dd x \dd r &  = \int_0^t \int_\Omega (\zeta_2{-}\zeta_1) \partial_t\tilde{p} \dd x \dd r + \int_0^t\int_\Omega\tilde{\sigma}_\dev \partial_t \tilde{p} \dd x \dd r \doteq I_{10}+I_6
    \end{aligned}
 \end{equation}
with $\zeta_i  \in \partial \mathrm{H}(\dot{p}_i)$ a.e.\ in $Q$
fulfilling the plastic flow rules for $i=1,\, 2$. 
 By monotonicity  \MMM (it is crucial that $\mathrm{H}$ does not depend on the other variables), \EEE we clearly have $I_{10} \leq 0$, while the second integral coincides with $I_6$. 
 \par
In the end, we sum up \eqref{est-tilde-u}, \eqref{est-tilde-z}, and \eqref{est-tilde-p}. Taking into account all of the previous calculations, the cancellation of the last term on the right-hand side of  \eqref{est-tilde-p} with that arising from the test of the momentum balance, 
and choosing the constants $\eta$, $\varrho$, and $\lambda$ in such a way as to absorb the term 
$ \iint_{Q} |\partial_t \tilde{e}|^2  $ into its analogue on the left-hand side, and to absorb  $ \int_0^t  \|\partial_t \tilde{z}\|_{L^\infty}^2  $ into
$ \int_0^t \ass( \partial_t \tilde{z}, \partial_t \tilde{z}) \dd r$, 
 we obtain
\[
\begin{aligned}
&
c\left( \int_\Omega |\partial_t\tilde{u}(t)|^2 \dd x +  \int_0^t \int_\Omega |\partial_t \tilde{e}|^2 \dd x \dd r +
  \int_0^t  \int_\Omega |\partial_t \tilde{z}|^2 \dd x  \dd r + \int_0^t \ass( \partial_t \tilde{z}, \partial_t \tilde{z}) \dd r \right) 
 \\
 & \quad
 + \frac12 \ass (\tilde{z}(t), \tilde{z}(t)) +  \int_0^t\int_\Omega |\partial_t \tilde{p}|^2 \dd x \dd r
 \\
 & 
 \begin{aligned}
 \leq 
   C \Big(
 \| \dot{u}_1^0 {-}  \dot{u}_2^0\|_{L^2(\Omega;\R^d)}^2 + \| z_1^0 {-}z_2^0\|_{\spz}^2 &  + 
 \| \mathcal{L}_1 {-} \mathcal{L}_2  \|_{L^2(0,T;H^1(\Omega;\R^d)^*)}^2  + \|\Theta_1{-}\Theta_2\|_{L^2(Q)}^2
 \\
 & +
  \|w_1 {-}w_2\|_{H^1(0,T;H^1(\Omega;\R^d)) \cap W^{1,\infty}(0,T;L^2(\Omega;\R^d))}^2 \Big)
  \end{aligned}
  \\
  &
  \begin{aligned}
   \quad 
  + C \int_0^t  \| \tilde{e} \|_{L^2(\Omega;\mt_\sym^{d\times d})}^2 \dd r 
&  +C \int_0^t (1 {+} \|\Theta_1\|_{L^2(\Omega)}^2{+} \|\Theta_2\|_{L^2(\Omega)}^2)  \| \tilde{z} \|_{\spz}^2 \dd r 
\\
& 
  +\rho \int_0^t\| \partial_{tt} (w_{1}{-}w_2) \|_{L^2(\Omega;\R^d)}\|  \partial_t\tilde{u}\|_{L^2(\Omega;\R^d)} \dd r,
  \end{aligned}
\end{aligned}
\]
where we have also used the continuous embedding $\spz \subset L^\infty(\Omega)$. 
Hence, applying the Gronwall Lemma we obtain 
a continuous dependence estimate  in terms of the norms 
\[
\| \partial_t \tilde{u}\|_{L^\infty(0,t;L^2(\Omega;\R^d))}\,, \
\| \partial_t \tilde{e}\|_{L^2(0,t;L^2(\Omega;\mt_\sym^{d \times d}))}\,, \
\|\partial_t \tilde{z}\|_{L^2(0,t;\spz)}\,,  \  \| \tilde{z}\|_{L^\infty(0,t;\spz)}\,, \
\| \partial_t \tilde{p}\|_{L^2(0,t;L^2(\Omega;\mt_\sym^{d \times d}))}\,.
\]
Then, estimate \eqref{cont-dependence-estimate} follows by \MMM easy \EEE calculations.  This concludes the proof.
\QED

\MMM
\paragraph{\bf Acknowledgments.} I am very grateful to the two anonymous referees for reading this paper very carefully and offering several suggestions that have helped improving it. \EEE



\bibliographystyle{alpha}
\bibliography{ricky_lit.bib} 
 \end{document}